\documentclass[11pt]{article}
\usepackage{enumerate}
\usepackage{amssymb,a4wide,latexsym,makeidx,epsfig}
\usepackage{amsthm}
\usepackage{dsfont}
\usepackage{amsmath}
\usepackage{lipsum}
\usepackage{enumerate}
\usepackage{mathrsfs}
\usepackage{xcolor}
\usepackage{tikz}
\usepackage{setspace}
\usepackage{geometry}
\usepackage{appendix}
\usepackage{multirow}
\usepackage{subcaption}
\usepackage[colorlinks=true, linkcolor=black, anchorcolor=black, citecolor=black, urlcolor=black, CJKbookmarks=true]{hyperref}
\usepackage{microtype}
\usepackage{colonequals}
\usepackage{enumitem}%通过[leftmargin=0.55cm, itemindent=1cm]让item缩进
\allowdisplaybreaks

\newtheorem{theorem}{Theorem}[section]

\newtheorem{definition}[theorem]{Definition}
\newtheorem{lemma}[theorem]{Lemma}
\newtheorem{fact}[theorem]{Fact}

\newtheorem{proposition}[theorem]{Proposition}
\newtheorem{corollary}[theorem]{Corollary}
\newtheorem{problem}[theorem]{Problem}

\newtheorem{claim}{Claim}[theorem]

\numberwithin{equation}{section}

\newcommand{\pr}{{\rm Pr}}
\newcommand{\bernoulli}{{\rm Bernoulli}}
\newcommand{\binomial}{{\rm Binomial}}

\begin{document}
\textwidth 150mm \textheight 225mm

\title{Ramsey-type results for threshold graphs and beyond}

\author{
Xihe Li\footnote{School of Mathematics and Statistics, Shaanxi Normal University, Xi'an, Shaanxi 710119, China.}~\footnote{Corresponding author. E-mail address: xiheli@snnu.edu.cn.}
}
\date{}
\maketitle
\newcommand\blfootnote[1]{%
\begingroup
\renewcommand\thefootnote{}\footnote{#1}%
\addtocounter{footnote}{-1}%
\endgroup
}
%\blfootnote{E-mail addresses: xiheli@snnu.edu.cn.}
\begin{center}
\begin{minipage}{145mm}
%\vskip 0.3cm
\begin{center}
{\small {\bf Abstract}}
\end{center}
{\small
A {\it threshold graph} is a graph that can be constructed from the one-vertex graph by repeatedly adding either a dominating vertex or an isolated vertex.
Motivated by an induced Ramsey-type problem for this class, we define $r'_2(s)$ to be the minimum integer $n$ such that every $n$-vertex graph contains an induced threshold graph on $s$ vertices.
We establish exponential upper and lower bounds for $r'_2(s)$ and determine its exact values for $s\in\{3,4,5,6\}$.
To study this problem from an edge-coloring perspective, we use the notion of an orderable coloring, introduced by Richer [{\it J. Combin. Theory Ser. B}, 80(1) (2000), 172--177].
An edge-colored graph is {\it orderable} if its vertices can be ordered so that, for each vertex, all edges from it to later vertices have the same color.
Equivalently, $r'_2(s)$ is the minimum $n$ such that every $2$-edge-coloring of $K_n$ contains an orderable $K_s$.
We also determine the exact value of the unordered canonical Ramsey number $CR(s, 3)$ for all $s \ge 3$, where $CR(s,3)$ denotes the minimum integer $n$ such that every edge-coloring of $K_n$ contains either an orderable $K_s$ or a rainbow $K_3$.
More generally, for graphs $G$ and $H$, we study $r'_2(G)$, the corresponding $2$-color Ramsey number for an orderable $G$, and $CR(G,H)$, where the alternative is a rainbow $H$.
For complete bipartite graphs, we prove that for every fixed $s$, $r'_2(K_{s,t}) = CR(K_{s,t}, K_3)= \left(\frac{2^s}{s+1}+o(1)\right)t$ as $t\to\infty$.
For $s\in \{2,3\}$, we further determine the exact values of these parameters for infinitely many $t$, using constructions arising from strongly regular graphs, Hadamard matrices and conference matrices.
\hspace{2em}

}
\end{minipage}
\end{center}

\section{Introduction}
\label{sec:introduction}

Ramsey theory is a classical area of combinatorics concerned with the existence of prescribed substructures in sufficiently large host structures.
For positive integers $s$ and $t$, the {\it Ramsey number} $r(s,t)$ is the minimum integer $n$ such that every $n$-vertex graph contains either a clique of order $s$ or an independent set of size $t$.
In recent years, researchers have achieved significant breakthroughs in upper and lower bounds for Ramsey numbers (see, for example, \cite{BBCGHMST,CGMS,MaSX,MaVe}).
For further developments in Ramsey theory, we refer the interested reader to three surveys~\cite{Mor,Rad,Ver}.

In this paper, we study Ramsey-type problems for threshold graphs, which play an important role in graph theory and have wide applications in fields such as computer science, scheduling theory and psychology.
A {\it threshold graph} of order $s$ is an $s$-vertex graph that can be constructed from the one-vertex graph by repeatedly adding a new vertex that is either adjacent to all the previous vertices or nonadjacent to all the previous vertices.
Threshold graphs have several equivalent definitions, and we will introduce them in Section~\ref{sec:pf_complete};
see monograph~\cite{MaPe} for more information on threshold graphs.
Note that if $T$ is a threshold graph, then its complement $\overline{T}$ is also a threshold graph.
We extend the classical Ramsey number from complete graphs to threshold graphs as follows.

\begin{definition}\label{def:r-threshold}
{\rm
For any positive integer $s$, let $r'_2(s)$ be the minimum integer $n$ such that every $n$-vertex graph contains a threshold graph of order $s$ as an induced subgraph.
}
\end{definition}

We provide three remarks on the definition of $r'_2(s)$.
First, since the empty graph\footnote{An {\it empty graph} is a graph with an empty edge set.} is a threshold graph, every graph on at least $s$ vertices contains a threshold graph of order $s$ as an ordinary subgraph.
Hence, to make the definition meaningful, we require the threshold graph to be an induced subgraph of the host graph in Definition~\ref{def:r-threshold}.
Second, there are $2^{s-1}$ nonisomorphic threshold graphs of order $s$.
If we denote by $\mathcal{T}_s$ the set of all threshold graphs of order $s$, then $r'_2(s)$ is the minimum integer $n$ such that every $n$-vertex graph contains some $T\in \mathcal{T}_s$ as an induced subgraph.
Third, since complete graphs and empty graphs are threshold graphs, we have $r'_2(s)\leq r(s,s)$, so the parameter $r'_2(s)$ is well-defined.

Our first result provides upper and lower bounds on $r'_2(s)$.
The upper bound given by Theorem~\ref{thm:r2s} actually holds for every integer $s\geq 6$.
We will also prove that the exact values of $r'_2(s)$ are $3, 6, 10, 18$ for $s=3, 4, 5, 6$, respectively; see Theorems~\ref{thm:r34}, \ref{thm:r5} and \ref{thm:r6}.

\begin{theorem}\label{thm:r2s}
As $s\to \infty$, we have
$$\left(\sqrt{2}\ln 2-o(1)\right)2^{\frac{s}{2}}< r'_2(s)\leq \left\lceil\frac{13}{48}2^s\right\rceil.$$
\end{theorem}

To study this problem from the perspective of edge-colorings, we use the notion of an orderable coloring introduced by Richer~\cite{Ric} in 2000.
Throughout this paper, we use $\chi$ to denote an edge-coloring of a graph, and $\chi(e)$ to denote the color assigned to an edge $e$.
An edge-colored $s$-vertex graph $G$ is called {\it orderable} if there exist an ordering $(v_1, v_2, \ldots, v_s)$ of $V(G)$ and a list $(c_1, c_2, \ldots, c_{s-1})$ of not necessarily distinct colors such that, for every edge $v_iv_j\in E(G)$ with $i<j$, we have $\chi(v_iv_j)=c_i$.
We now extend the classical Ramsey number from monochromatic subgraphs to orderable subgraphs.

\begin{definition}\label{def:r}
{\rm
For any graph $G$, let $r'_2(G)$ be the minimum integer $n$ such that every 2-edge-coloring of $K_n$ contains an orderable copy of $G$.
}
\end{definition}

We now show that $r'_2(K_s)=r'_2(s)$.
On the one hand, consider a 2-edge-coloring of $K_n$ using red and blue, where $n=r'_2(s)$.
Let $H_r$ be the spanning subgraph of $K_n$ consisting of all the red edges.
Since $|V(H_r)|=n=r'_2(s)$, $H_r$ contains an induced threshold graph $T$ of order $s$.
By the definition of a threshold graph, we may assume that $V(T)=\{u_1, u_2, \ldots, u_s\}$ such that for each $2\leq i\leq s$, $u_i$ is either adjacent to or nonadjacent to all the vertices of $\{u_1, \ldots, u_{i-1}\}$ in $T$.
This implies that for each $2\leq i\leq s$, $u_i$ is joined to $\{u_1, \ldots, u_{i-1}\}$ by either all red edges or all blue edges in $K_n$.
Thus $K_n$ contains an orderable $K_s$ under the ordering $(u_s, u_{s-1}, \ldots, u_1)$, so $r'_2(K_s)\leq r'_2(s)$.
On the other hand, assume that $F$ is an $n$-vertex graph with $n=r'_2(K_s)$.
Let $F'$ be the 2-edge-colored $K_n$ obtained from $F$ by coloring all edges of $F$ with red and coloring all edges of the complement $\overline{F}$ with blue.
Since $|V(F')|=n=r'_2(K_s)$, $F'$ contains an orderable $K_s$ with an ordering $(v_1, v_2, \ldots, v_s)$.
We construct a threshold graph of order $s$ that is an induced subgraph of $F$ as follows.
Starting with $v_s$ and adding $v_{s-1}, \ldots, v_1$, each $v_i$ is adjacent to all of $v_{i+1}, \ldots, v_s$ if all edges from $v_i$ to $\{v_{i+1}, \ldots, v_s\}$ are red in $F'$,
and nonadjacent to all of $v_{i+1}, \ldots, v_s$ if all edges from $v_i$ to $\{v_{i+1}, \ldots, v_s\}$ are blue in $F'$.
This implies that $r'_2(s) \leq r'_2(K_s)$.

Richer's motivation for introducing orderable colorings arose from the Erd\H{o}s-Rado Canonical Ramsey Theorem.
An edge-colored graph is called {\it monochromatic} if all edges are colored the same,
{\it rainbow} if all edges are colored differently,
and {\it lexical} if there is a total order of its vertices such that two edges have the same color if and only if they share the same smaller endpoint.
The celebrated Erd\H{o}s-Rado Canonical Ramsey Theorem~\cite{ErRa} states that for any positive integer $s$, there exists a minimum integer $n$ such that every edge-coloring of $K_n$ with an arbitrary number of colors contains a monochromatic, a rainbow or a lexical copy of $K_s$.
Note that in the definition of orderable edge-colored graphs, if all the colors in the list $(c_1, c_2, \ldots, c_{s-1})$ are the same, then $G$ is monochromatic, and if all the colors in the list are distinct, then $G$ is lexically colored.
Therefore, the notion of an orderable coloring unifies the notions of monochromatic and lexical colorings.
Moreover, the Canonical Ramsey Theorem implies that for any positive integer $s$, if $n$ is large enough, then every edge-coloring of $K_n$ with any number of colors contains either an orderable copy or a rainbow copy of $K_s$.
Richer~\cite{Ric} defined the unordered canonical Ramsey number as follows.

\begin{definition}{\normalfont (\cite{Ric})}\label{def:CR}
{\rm
For positive integers $s$ and $t$, the {\it unordered canonical Ramsey number} $CR(s,t)$ is defined as the minimum integer $n$ such that every edge-coloring of $K_n$ with any number of colors contains either an orderable copy of $K_s$ or a rainbow copy of $K_t$.
}
\end{definition}

For general $s$ and $t$, Richer~\cite{Ric} proved that $$\left({t\choose 2}-1\right)^{s-2}+1\leq CR(s,t)\leq 7^{3-s}t^{4s-4}.$$
Before the present work, the asymptotic order of $CR(s,t)$ was known in the regime where $s$ is fixed and $t\to \infty$:
$$CR(s,t)=\Theta\left(\left(\frac{t^3}{\ln t}\right)^{s-2}\right),$$
where the lower bound was obtained by Jiang~\cite{Jiang} in 2009, and the upper bound was obtained by Araujo and Peng~\cite{ArPe} in 2025.
In this paper, we focus on the case in which $t$ is fixed.
We will mainly consider the first open case, namely $t=3$.
Our second result is the exact value of $CR(s, 3)$ for all integers $s\geq 3$.

\begin{theorem}\label{thm:CRs3}
For any integer $s\geq 3$, we have
$$CR(s,3)=
\left\{
   \begin{aligned}
    &2\cdot 5^{\frac{s-3}{2}}+1 & & \mbox{if $s$ is odd},\\
    &5^{\frac{s-2}{2}}+1 & & \mbox{if $s$ is even}.
   \end{aligned}
   \right.$$
\end{theorem}

In 2025, Brosch, Lidick\'{y}, Miyasaki and Puges~\cite{BLMP} generalized the unordered canonical Ramsey number from complete graphs to general graphs.

\begin{definition}{\normalfont (\cite{BLMP})}\label{def:CRGH}
{\rm
Given two graphs $G$ and $H$, the unordered canonical Ramsey number $CR(G,H)$ is defined as the minimum integer $n$ such that every edge-coloring of $K_n$ with any number of colors contains either an orderable copy of $G$ or a rainbow copy of $H$.
}
\end{definition}

Note that $CR(K_s,K_t)=CR(s,t)$.
To the best of our knowledge, apart from results for $CR(s,t)$, all currently known results on $CR(G,H)$ concern small bipartite graphs, which were obtained by Brosch et al.~\cite{BLMP} via flag algebras and integer linear programs.
In this paper, we study $CR(K_{s,t}, K_3)$, that is, the problem of finding an orderable $K_{s,t}$ or a rainbow $K_3$.
Since a 2-edge-coloring of $K_n$ certainly contains no rainbow $K_3$, we have $CR(G, K_3)\geq r'_2(G)$ for any graph $G$.
Our next result implies that $CR(K_{s,t}, K_3)$ and $r'_2(K_{s,t})$ are equal when $t$ is large enough.
This contrasts with the behavior for complete graphs, where the two parameters need not coincide.
Theorem~\ref{thm:CRKst} below also implies that the asymptotic values of $CR(K_{s,t}, K_3)$ and $r'_2(K_{s,t})$ are $\left(\frac{2^s}{s+1}+o(1)\right)t$ as $t\to\infty$.

\begin{theorem}\label{thm:CRKst}
For any fixed integer $s \geq 2$ and sufficiently large $t$, we have
$$\frac{2^s}{s+1}t-O_s\left(\sqrt{t \ln t}\right)\leq r'_2(K_{s,t})= CR(K_{s,t}, K_3) \leq \frac{2^s}{s+1}t+O_s(1).$$
\end{theorem}

For $K_{1,t}$, we have $r'_2(K_{1,t})= CR(K_{1,t}, K_3)=t+1$ (see Proposition~\ref{prop:forest}).
For $s\geq 2$, we can in fact deduce the following upper bounds for all $t\geq s$.

\begin{theorem}\label{thm:CRKstUpper}
For any integer $t\geq 2$, we have
$$r'_2(K_{2,t})\leq CR(K_{2,t}, K_3) \leq \left\lfloor \frac{4t}{3}\right\rfloor+2.$$
For any integers $t\geq s \geq 3$, we have
$$r'_2(K_{s,t})\leq CR(K_{s,t}, K_3) \leq \left\lceil \frac{2^s}{s+1}(t-1)+\frac{(s-1)(s+6)}{6}\right\rceil+1.$$
\end{theorem}

Note that for $s=3$, the upper bound given by Theorem~\ref{thm:CRKstUpper} is $2t+2$.
We now show that $r'_2(K_{2,t})= CR(K_{2,t}, K_3)=\left\lfloor \frac{4t}{3}\right\rfloor+2$ and $r'_2(K_{3,t})= CR(K_{3,t}, K_3)=2t+2$ for infinitely many $t$ using the properties of strongly regular graphs.
Hence, for $s\in \{2,3\}$, the upper bounds in Theorem~\ref{thm:CRKstUpper} are attained for infinitely many values of $t$.
A {\it strongly regular graph} with parameters $(n, k, \lambda, \mu)$, denoted by $SRG(n, k, \lambda, \mu)$, is a $k$-regular graph on $n$ vertices in which any two adjacent vertices have exactly $\lambda$ common neighbors, and any two nonadjacent vertices have exactly $\mu$ common neighbors.
A strongly regular graph with parameters $\left(n, \frac{n-1}{2}, \frac{n-5}{4}, \frac{n-1}{4}\right)$ is also called a {\it conference graph}.

\begin{theorem}\label{thm:CRK2t3t} Let $t$ be a positive integer.
\begin{itemize}
\item[{\rm (i)}] If $t\equiv 0 \pmod{3}$ and there exists an $SRG\left(n, \frac{n-1}{2}, \frac{n-5}{4}, \frac{n-1}{4}\right)$ with $n=\frac{4t}{3}+1$, then
$$r'_2(K_{2,t})= CR(K_{2,t}, K_3)=\frac{4t}{3}+2.$$
\item[{\rm (ii)}] If $t\equiv 0 \pmod{2}$ and there exists an $SRG\left(2t+1, t, \frac{t}{2}-1, \frac{t}{2}\right)$, then
$$r'_2(K_{3,t})= CR(K_{3,t}, K_3)=2t+2.$$
\end{itemize}
\end{theorem}

Let $q$ be a prime power with $q\equiv 1 \pmod{4}$.
The {\it Paley graph} of order $q$ is the graph whose vertices are the elements of the field $\mathbb{F}_q$, and two vertices $u$ and $v$ are adjacent if and only if $u-v$ is a non-zero quadratic residue in $\mathbb{F}_q$.
It is known that the Paley graph of order $q$ is a strongly regular graph with parameters $\left(q, \frac{q-1}{2}, \frac{q-5}{4}, \frac{q-1}{4}\right)$; see, for example, \cite[Section~1.1.9]{BrVaMa}.
Therefore, we have the following corollary.

\begin{corollary}\label{cor:CRK2t3t} Let $t$ be a positive integer.
\begin{itemize}
\item[{\rm (i)}] If $t\equiv 0 \pmod{3}$ and $\frac{4t}{3}+1$ is a prime power, then
$r'_2(K_{2,t})= CR(K_{2,t}, K_3)=\frac{4t}{3}+2.$
\item[{\rm (ii)}] If $t$ is even and $2t+1$ is a prime power, then
$r'_2(K_{3,t})= CR(K_{3,t}, K_3)=2t+2.$
\end{itemize}
\end{corollary}

Using properties of Hadamard matrices and conference matrices, we obtain additional exact values of $r'_2(K_{2,t})$ and $CR(K_{2,t}, K_3)$.
For any positive integer $n$, let $[n]\colonequals \{1, 2, \ldots, n\}$.
A {\it Hadamard matrix} $H$ of order $n$ is a square matrix with entries in $\{1,-1\}$ such that $HH^{\top}=nI_n$, where $H^{\top}$ is the transpose of $H$ and $I_n$ is the $n\times n$ identity matrix.
It is called {\it symmetric} if $H=H^{\top}$.
We say that $H=(h_{ij})$ has a {\it constant diagonal} if $h_{ii}=1$ for all $i\in [n]$ or $h_{ii}=-1$ for all $i\in [n]$.
The matrix $H$ is called {\it graphical} if $H$ is symmetric and has a constant diagonal.
A {\it conference matrix} $C=(c_{ij})$ of order $n$ is a square matrix with $c_{ii}=0$ for all $i\in [n]$, $c_{ij}\in \{1, -1\}$ for all $1\leq i\neq j\leq n$, and $C^{\top}C=(n-1)I_n$.
It is called {\it symmetric} if $C=C^{\top}$,
and {\it regular} if $C\mathbf{1}=\pm \sqrt{n-1} \mathbf{1}$, where $\mathbf{1}$ is the all-ones vector.

\begin{theorem}\label{thm:CRK2t+} Let $m$ be a positive integer.
\begin{itemize}
\item[{\rm (i)}] If there exists a Hadamard matrix of order $m$ $(m\geq 2)$, then
$$r'_2\Big(K_{2,\frac{3m^2}{4}-1}\Big)= CR\Big(K_{2,\frac{3m^2}{4}-1}, K_3\Big)=m^2.$$
\item[{\rm (ii)}] If there exists a graphical Hadamard matrix of order $4m^2$, then
$$r'_2(K_{2,3m^2-1})= CR(K_{2,3m^2-1}, K_3)=4m^2.$$
\item[{\rm (iii)}] If there exists a conference matrix of order $m+1$, then
$$r'_2\Big(K_{2,\frac{3m^2+1}{4}}\Big)= CR\Big(K_{2,\frac{3m^2+1}{4}}, K_3\Big)=m^2+2.$$
\item[{\rm (iv)}] If there exists a regular symmetric conference matrix of order $m^2+1$, then
$$r'_2\Big(K_{2,\frac{3m^2+1}{4}}\Big)= CR\Big(K_{2,\frac{3m^2+1}{4}}, K_3\Big)=m^2+2.$$
\end{itemize}
\end{theorem}

It is known that the order of a Hadamard matrix must be 1, 2, or a multiple of 4 (see, for example, \cite[Chapter~2]{Hor}).
A long-standing problem in the field is whether a Hadamard matrix of order $4k$ exists for every positive integer $k$.
Sylvester~\cite{Syl} constructed Hadamard matrices of order $2^k$ for every non-negative integer $k$.
This together with Theorem~\ref{thm:CRK2t+}~(i) implies Corollary~\ref{cor:CRK2t+}~(i) below.
Belevitch~\cite{Bel} proved that the order of a conference matrix must be even, and if $n$ is the order of a symmetric conference matrix, then it must hold that $n\equiv 2 \pmod{4}$ and $n-1$ is the sum of two squares.
The existence of conference matrices of order $n$ is known only for some values of $n$.
It is known that for any odd prime power $q$, there exists a regular symmetric conference matrix of order $q^2+1$ (see~\cite[Section~4.2]{GHKS}), so Corollary~\ref{cor:CRK2t+}~(ii) holds by Theorem~\ref{thm:CRK2t+}~(iv).

\begin{corollary}\label{cor:CRK2t+} The following statements hold.
\begin{itemize}
\item[{\rm (i)}] For any positive integer $k$, we have
$r'_2(K_{2,3\cdot 4^{k-1}-1})= CR(K_{2,3\cdot 4^{k-1}-1}, K_3)=4^{k}.$
\item[{\rm (ii)}] For any odd prime power $q$, we have
$r'_2\Big(K_{2,\frac{3q^2+1}{4}}\Big)= CR\Big(K_{2,\frac{3q^2+1}{4}}, K_3\Big)=q^2+2.$
\end{itemize}
\end{corollary}

The remainder of this paper is organized as follows.
In the next section, we introduce some additional terminology and notation, as well as several results that will be used in our proofs.
In Section~\ref{sec:pf_complete}, we prove Theorems~\ref{thm:r2s} and \ref{thm:CRs3}, and determine exact values of $r'_2(s)$ for $s\in \{3,4,5,6\}$.
In Section~\ref{sec:pf_complete_bipar}, we present our proofs of Theorems~\ref{thm:CRKst}, \ref{thm:CRKstUpper}, \ref{thm:CRK2t3t} and \ref{thm:CRK2t+}.
Finally, we conclude this paper with some remarks and open problems in Section~\ref{sec:conclu}.

\section{Preliminaries}
\label{sec:pre}

We begin with some terminology and notation.
Let $\chi$ be an edge-coloring of a graph $G$.
For disjoint subsets $U, V\subset V(G)$, let $E(U, V)\colonequals \{uv\in E(G)\colon\, u\in U, v\in V\}$ and $\chi(U, V)\colonequals \{\chi(e)\colon\, e\in E(U,V)\}$.
If $\left|\chi(U, V)\right|=1$, then we also use $\chi(U, V)$ to denote the unique color in $\chi(U, V)$.
The subgraph of $G$ induced by $U$ is denoted by $G[U]$, and $G-U$ is shorthand for $G[V(G)\setminus U]$.
If $U$ consists of a single vertex $u$, then we simply write $E(\{u\}, V)$, $\chi(\{u\}, V)$ and $G-\{u\}$ as $E(u, V)$, $\chi(u, V)$ and $G-u$, respectively.
For a set $E\subseteq E(G)$ of edges, the {\it subgraph induced by $E$} is the subgraph consisting of all the edges in $E$ and all the vertices that are incident with them.
For a color $c$, the {\it subgraph induced by color $c$} is the subgraph induced by the set of all edges of color $c$.
Given a vertex $v\in V(G)$, let $N_G(v)$ be the neighborhood of $v$ in $G$, and $d_G(v)\colonequals |N_G(v)|$ be the degree of $v$.
Note that $N_{\overline{G}}(v)=V(G)\setminus (N_G(v) \cup \{v\})$.
Let $\delta(G)$ and $\Delta(G)$ be the minimum degree and maximum degree of $G$, respectively.

In this paper, we use $P_n$ and $C_n$ to denote the path and the cycle on $n$ vertices, respectively.
Given two graphs $G$ and $H$, let $G\cup H$ be the disjoint union of $G$ and $H$.
For any positive integer $n$, we use $nG$ to denote the disjoint union of $n$ copies of $G$.

Given a set $X$ and a positive integer $n\leq |X|$, let ${X\choose n}\colonequals \{X'\subseteq X\colon\, |X'|=n\}$.
For an event $A$, let $\mathbf{1}_A$ denote its indicator function, defined by
$$\mathbf{1}_A \colonequals
\left\{
   \begin{aligned}
    &1 & & \mbox{if $A$ occurs},\\
    &0 & & \mbox{otherwise}.
   \end{aligned}
   \right.$$
We will use the standard Bachmann-Landau notation such as $O$, $o$ and $\Theta$ to indicate asymptotic growth rates of functions.
Subscripts will be used to indicate that a function depends on a parameter.
For example, for a non-negative real valued function $f(t)$, the notation $O_s\left(f(t)\right)$ denotes a quantity that is at most $C_sf(t)$ for some constant $C_s$ depending only on $s$.

An edge-coloring of a complete graph is called a {\it Gallai-coloring} if it contains no rainbow triangle, i.e., every subgraph $K_3$ is colored by at most two colors.
Gallai~\cite{Gallai} obtained the following characterization of the structure of a complete graph with a Gallai-coloring.

\begin{theorem}{\normalfont (\cite{Gallai})}\label{thm:Gallai}
In any Gallai-coloring of a complete graph on at least two vertices, the vertex set can be partitioned into nonempty sets $V_1, V_2, \ldots, V_m$ with $m\geq 2$ such that
\begin{itemize}
\item[{\rm (1)}] for every $1\leq i<j\leq m$, we have $\left|\chi(V_i, V_j)\right|=1$, and
\item[{\rm (2)}] $\left|\bigcup_{1\leq i<j\leq m}\chi(V_i, V_j)\right| \leq 2$.
\end{itemize}
\end{theorem}

A vertex partition satisfying the conclusions of Theorem~\ref{thm:Gallai} is called a {\it Gallai partition}.
Note that the subgraph induced by $\bigcup_{1\leq i<j\leq m}E(V_i, V_j)$ can be viewed as a blow-up of a 2-edge-colored $K_m$.
However, the above result provides no information regarding the sizes of the parts or the coloring within each part, except that these colorings contain no rainbow triangle.
For more Ramsey-type results on Gallai-colorings, we refer the interested reader to \cite{AJMP,AxJa05GC,FoGP,GSSS,LiBW,LiLS,LMSSS,MWMS,ZhSC}.

Finally, we show that for an arbitrarily edge-colored forest, there exists an ordering of its vertices such that the coloring is orderable.

\begin{proposition}\label{prop:forest}
Let $F$ be an arbitrarily edge-colored forest. Then there exists an ordering of $V(F)$ such that the coloring is orderable.
Moreover, for any forest $F$ and any graph $H$ with $|E(H)|\geq 2$, we have $r'_2(F)= CR(F, H)=|V(F)|$.
\end{proposition}

\noindent {\bf Proof.}
%\begin{proof}
Let $F_1, \ldots, F_m$ be the connected components of $F$.
Then each $F_i$ is a tree.
For any component $F_i$ and a fixed vertex $v_i\in V(F_i)$, let $\ell_i$ be the largest distance from $v_i$ to a vertex of $F_i$.
For $j=0, 1, \ldots, \ell_i$, let $N_{i}^{(j)}$ be the set of vertices at distance $j$ from $v_i$.
Note that $N_{i}^{(0)}=\{v_i\}$.
Moreover, since $F_i$ is a tree, we have that for every $j=1, 2, \ldots, \ell_i$, $N_{i}^{(j)}$ is an independent set,
and each vertex of $N_{i}^{(j)}$ has exactly one neighbor in $N_{i}^{(0)}\cup \cdots \cup N_{i}^{(j-1)}$.
We consider an ordering $\sigma$ of $V(F)$ such that
\begin{itemize}
\item for $1\leq i<i'\leq m$, the vertices of $F_i$ appear before those of $F_{i'}$;
\item for $i\in [m]$ and $\ell_i\geq j>j'\geq 0$, the vertices of $N_{i}^{(j)}$ appear before those of $N_{i}^{(j')}$.
\end{itemize}
Note that every non-root vertex has exactly one neighbor appearing after it, and every root $v_i$ has no neighbor appearing after it.
Thus $F$ is orderable under $\sigma$, regardless of the original coloring.
This implies that $r'_2(F)\leq |V(F)|$ and $CR(F, H)\leq |V(F)|$.
Since $K_{|V(F)|-1}$ contains no $F$, we have $r'_2(F)> |V(F)|-1$, and thus $r'_2(F)= |V(F)|$.
Since $|E(H)|\geq 2$, a monochromatic $K_{|V(F)|-1}$ contains neither an orderable $F$ nor a rainbow $H$.
Hence, $CR(F, H)> |V(F)|-1$, and thus $CR(F, H)= |V(F)|$.
%\end{proof}
\hfill$\blacksquare$
\vspace{0.2cm}

As a consequence of Proposition~\ref{prop:forest}, we can show that $r'_2(G)=|V(G)|$ and $CR(G, K_3)=|V(G)|$ for infinitely many graphs $G$, including all cycles.

\begin{corollary}\label{cor:forest}
Let $G$ be an $n$-vertex graph with a vertex $u$ such that $G-u$ is a forest.
\begin{itemize}
\item[{\rm (i)}] If $d_G(u)\leq
\left\{
   \begin{aligned}
    &(n+1)/2 & & \mbox{if $n\equiv 3 \pmod{4}$},\\
    &\big\lceil(n-1)/2\big\rceil & & \mbox{otherwise},
   \end{aligned}
   \right.$ then $r'_2(G)=|V(G)|$.
\item[{\rm (ii)}] If $d_G(u)\leq \big\lceil\frac{2n}{5}\big\rceil$, then $CR(G, K_3)=|V(G)|$.
\end{itemize}
\end{corollary}

\noindent {\bf Proof.}
%\begin{proof}
(i) For any $2$-edge-colored $K_n$ and a vertex $v\in V(K_n)$, there exists a color such that $v$ is incident with at least $\big\lceil(n-1)/2\big\rceil$ edges of this color.
If $n\equiv 3 \pmod{4}$, then $n$ is odd and $(n-1)/2$ is odd, so there is no $((n-1)/2)$-regular graph on $n$ vertices.
This implies that for $n\equiv 3 \pmod{4}$, there exists a vertex $v$ that is incident with at least $(n+1)/2$ edges of the same color.
We embed $G$ into $K_n$ such that $v$ plays the role of $u$ and $N_G(u)$ is embedded into the vertices that are connected to $v$ by edges of the same color.
Note that $G-u$ is a forest and the vertex set of the embedded copy of $G-u$ is $V(K_n)\setminus \{v\}$.
Then, by Proposition~\ref{prop:forest}, there exists an ordering $\sigma$ of $V(K_n)\setminus \{v\}$ such that the coloring of the embedded copy of $G-u$ is orderable.
Let $\sigma'$ be the ordering of $V(K_n)$ obtained from $\sigma$ by adding $v$ as the first vertex.
Then we get an orderable $G$ under the ordering $\sigma'$.
This implies that $r'_2(G)\leq n= |V(G)|$.
Since $K_{|V(G)|-1}$ contains no $G$, we have $r'_2(G)> |V(G)|-1$, and thus $r'_2(G)= |V(G)|$.

(ii) The case $n=1$ is trivial, so we assume that $n\geq 2$.
We first show that $CR(G, K_3)\leq |V(G)|=n$.
Consider an arbitrary edge-colored $K_n$ with any number of colors.
Suppose that there is no rainbow $K_3$, and we shall show that there must be an orderable $G$.
By Theorem~\ref{thm:Gallai}, $V(K_n)$ admits a Gallai partition $V_1, V_2, \ldots, V_m$ with $m\geq 2$.
Without loss of generality, we may assume that the colors between the parts belong to $\{\mbox{red, blue}\}$.
We now show that there exists a vertex $v$ that is incident with at least $\big\lceil\frac{2n}{5}\big\rceil$ edges of the same color.
If $2\leq m\leq 3$, then there exists a part $V_i$, such that all edges between $V_i$ and $V(K_n)\setminus V_i$ are of the same color, say red.
This implies that there exists a vertex $v$ such that $v$ is incident with at least $\big\lceil\frac{n}{2}\big\rceil\geq \big\lceil\frac{2n}{5}\big\rceil$ red edges.
If $m=4$ and $|V_i|> \frac{n}{5}$ for every $i\in [4]$, then since at least two of $V_2, V_3, V_4$ are joined to $V_1$ by edges of the same color, every vertex in $V_1$ is incident with at least $\big\lceil\frac{2n}{5}\big\rceil$ edges of the same color.
If $m=4$ and $|V_i|\leq \frac{n}{5}$ for some $i\in [4]$, then every vertex in $V_i$ is incident with at least $\big\lceil\frac{1}{2}(n-\frac{n}{5})\big\rceil= \big\lceil\frac{2n}{5}\big\rceil$ edges of the same color.
If $m\geq 5$, then we may assume that $|V_m|=\min_{i\in [m]}|V_i|$ without loss of generality, so $|V_m|\leq \frac{n}{m}$ and $|V(K_n)\setminus V_m|\geq (1-\frac{1}{m})n.$
Now for any vertex $v\in V_m$, $v$ is incident with at least $\big\lceil\frac{1}{2}(1-\frac{1}{m})n\big\rceil\geq \big\lceil\frac{1}{2}(1-\frac{1}{5})n\big\rceil = \big\lceil\frac{2n}{5}\big\rceil$ edges of the same color.

By the same embedding argument as in (i), there exists an orderable $G$, so $CR(G, K_3)\leq |V(G)|=n$.
Since a monochromatic $K_{|V(G)|-1}$ contains neither an orderable $G$ nor a rainbow $K_3$, we have $CR(G, K_3)> |V(G)|-1$, and thus $CR(G, K_3)= |V(G)|$.
%\end{proof}
\hfill$\blacksquare$
%\vspace{0.2cm}

\section{Results on orderable complete graphs}
\label{sec:pf_complete}

In this section, we study $r'_2(s)$ and $CR(s,3)$.
In particular, we present our proofs of Theorems~\ref{thm:r2s} and \ref{thm:CRs3}.
We divide the proof of Theorem~\ref{thm:r2s} into two parts: Lemma~\ref{le:r2slower} (lower bound) and Corollary~\ref{cor:r2supper} (upper bound).
Section~\ref{subsec:pf_complete_lower} provides three constructive lower bounds and a probabilistic lower bound on $r'_2(s)$, and Section~\ref{subsec:pf_complete_upper} gives a recursive upper bound on $r'_2(s)$.
In Section~\ref{subsec:pf_complete_small}, we determine exact values of $r'_2(s)$ for $s\in \{3,4,5,6\}$.
The proof of Theorem~\ref{thm:CRs3} is presented in Section~\ref{subsec:pf_complete_CR}.

We first collect several results that will be used throughout this section.
Recall that a graph is called a threshold graph if it can be constructed from the one-vertex graph by repeatedly adding a new vertex that is either adjacent to all the previous vertices or nonadjacent to all the previous vertices.
In fact, threshold graphs have several equivalent definitions; see Lemma~\ref{le:threshold-equi} below.
In our proofs, we will mainly use the definition given above and Lemma~\ref{le:threshold-equi}~(ii).

\begin{lemma}{\normalfont (see, for example, \cite[Section~1.2]{MaPe}, \cite{DiHJ})}\label{le:threshold-equi}
For a graph $T$, the following statements are equivalent.
\begin{itemize}
\item[{\rm (i)}] $T$ is a threshold graph.
\item[{\rm (ii)}] $T$ contains no induced $P_4$, $C_4$ or $2K_2$.
\item[{\rm (iii)}] Any induced subgraph of $T$ with at least one vertex has either an isolated vertex or a dominating vertex~\footnote{In a graph $T$, a vertex is called a {\it dominating vertex} if it has degree $|V(T)|-1$.}.
\item[{\rm (iv)}] There exist nonnegative real weights $w_v$ for each vertex $v\in V(T)$ and a threshold value $a$ such that $uv\in E(T)$ if and only if $w_u+w_v>a$.
\item[{\rm (v)}] There exist nonnegative real weights $w_v$ for each vertex $v\in V(T)$ and a threshold value $b$ such that for any subset $U\subseteq V(T)$, $U$ is an independent set if and only if $\sum_{v\in U}w_v\leq b$.
\end{itemize}
\end{lemma}

Recall that $\mathcal{T}_s$ is the set of all threshold graphs of order $s$.
We say that a graph is {\it induced $\mathcal{T}_s$-free} if it contains no $T\in \mathcal{T}_s$ as an induced subgraph.
We shall use the following facts in our proofs.

\begin{fact}\label{fa:threshold-1}
The following statements hold.
\begin{itemize}
\item[{\rm (i)}] If $G$ is induced $\mathcal{T}_s$-free, then the complement $\overline{G}$ is also induced $\mathcal{T}_s$-free.
\item[{\rm (ii)}] Let $G$ be an induced $\mathcal{T}_s$-free graph $(s\geq 2)$ and $v$ be a vertex of $G$. Then $G[N_G(v)]$ and $G[N_{\overline{G}}(v)]$ are induced $\mathcal{T}_{s-1}$-free.
\item[{\rm (iii)}] If $G$ is a $4$-vertex graph and $G\notin \mathcal{T}_4$, then $G\in \{P_4, C_4, 2K_2\}.$
\item[{\rm (iv)}] If $G$ is a $5$-vertex induced $\mathcal{T}_4$-free graph, then $G\cong C_5.$
\end{itemize}
\end{fact}

\noindent {\bf Proof.}
(i) Let $G$ be an induced $\mathcal{T}_s$-free graph.
Suppose for a contradiction that $\overline{G}$ contains an induced subgraph $H$ that is a threshold graph of order $s$.
Then $\overline{H}$ is also a threshold graph of order $s$.
However, $\overline{H}$ is an induced subgraph of $G$, which contradicts the assumption that $G$ is induced $\mathcal{T}_s$-free.
Hence, $\overline{G}$ is induced $\mathcal{T}_s$-free.

(ii) Suppose that $G[N_G(v)]$ or $G[N_{\overline{G}}(v)]$ contains an induced subgraph $H$ that is a threshold graph of order $s-1$.
If $H$ is an induced subgraph of $G[N_G(v)]$, then $G[V(H)\cup \{v\}]$ is a threshold graph in which $v$ is a dominating vertex; 
if $H$ is an induced subgraph of $G[N_{\overline{G}}(v)]$, then $G[V(H)\cup \{v\}]$ is a threshold graph in which $v$ is an isolated vertex.
In either case, this is a contradiction.

(iii) If $G$ is a $4$-vertex graph and $G\notin \{P_4, C_4, 2K_2\}$, then $G$ contains no induced $P_4$, $C_4$ or $2K_2$.
By Lemma~\ref{le:threshold-equi}~(ii), $G$ is a threshold graph, i.e., $G\in \mathcal{T}_4$, a contradiction.

(iv) We first show that $\delta(G)\geq 2$.
Suppose for a contradiction that there exists a vertex $v\in V(G)$ with $d_G(v)\leq 1$.
Then there exist at least three vertices $x, y, z\in V(G)\setminus \{v\}$ such that $E(v, \{x, y, z\})=\emptyset$.
Thus $\delta(G[\{v, x, y, z\}])=0$.
In particular, we have $G[\{v, x, y, z\}]\notin \{P_4, C_4, 2K_2\}$.
By Fact~\ref{fa:threshold-1}~(iii), $G[\{v, x, y, z\}]$ is a threshold graph of order $4$, contradicting the fact that $G$ is induced $\mathcal{T}_4$-free.
Hence, $\delta(G)\geq 2$.
Moreover, $\overline{G}$ is also induced $\mathcal{T}_4$-free by Fact~\ref{fa:threshold-1}~(i).
Hence, by the same argument applied to $\overline{G}$, we have $\delta(\overline{G})\geq 2$, so $\Delta(G)= 4-\delta(\overline{G})\leq 2$.
Combining $\delta(G)\geq 2$ and $\Delta(G)\leq 2$, we know that $G$ is a 2-regular graph.
Since the unique 2-regular graph on five vertices is $C_5$, we have $G\cong C_5.$
\hfill$\blacksquare$
\vspace{0.2cm}

It is easy to check that the number of unlabeled threshold graphs of order $s$ is $2^{s-1}$.
The number of labeled threshold graphs of order $s$ was determined by Beissinger and Peled~\cite{BePe}.

\begin{lemma}{\normalfont (\cite{BePe}, see also \cite[Equality~(2.3)]{DiHJ})}\label{le:threshold}
For $s\geq 2$, the number of labeled threshold graphs on $s$ vertices is $$s!\left(\left(\frac{1}{\ln 2}-1\right)\left(\frac{1}{\ln 2}\right)^{s}+\varepsilon_s\right),$$
where $|\varepsilon_s|\leq \frac{2\zeta(s)}{(2\pi)^s}$ and $\zeta(s)$ is the Riemann zeta function\footnote{The Riemann zeta function is a function of a complex variable $s$ with $Re(s)>1$ defined as $\zeta(s)=\sum\nolimits_{n=1}^{\infty}\frac{1}{n^s}$. The demonstration of the particular value $\zeta(2)=\sum\nolimits_{n=1}^{\infty}\frac{1}{n^2}=\frac{\pi^2}{6}$ is known as the Basel problem.}.
\end{lemma}

For an event $A$, let $\pr[A]$ be the probability of $A$.
We shall use the symmetric version of the Lov\'{a}sz Local Lemma.

\begin{lemma}{\normalfont (\cite{ErLo,Spe})}\label{le:Local_Lemma}
Let $A_1, A_2, \ldots, A_n$ be events in an arbitrary probability space.
Assume that each event $A_i$ is mutually independent of all but at most $d$ of the other events, and that $\pr[A_i]\leq p$ for all $i\in [n]$.
If $ep(d+1)\leq 1$, then $\pr\left[\bigwedge_{i\in [n]}\overline{A_i}\right]>0$.
\end{lemma}

Next, we introduce the lexicographic product of two graphs and the lexicographic product coloring of two edge-colorings.
Given two graphs $G$ and $H$, the {\it lexicographic product} $G \otimes H$ is the graph defined as follows:
its vertex set is $V(G)\times V(H)$, and two vertices $(x_1, y_1)$ and $(x_2, y_2)$ are adjacent if and only if either $x_1x_2\in E(G)$, or $x_1=x_2$ and $y_1y_2\in E(H)$.
Given an edge-coloring $\chi$ of $K_{m}$ and an edge-coloring $\chi'$ of $K_{n}$, the {\it lexicographic product coloring} $\chi \otimes \chi'$ of $K_{mn}$ is defined as follows:
identify $V(K_{mn})$ with $[m]\times [n]$, and the color of an edge $\{(u_1, v_1), (u_2, v_2)\}$ is $\chi(u_1u_2)$ if $u_1\neq u_2$, and $\chi'(v_1v_2)$ if $u_1= u_2$.
Let $\tau(G)\colonequals \max\{|S|\colon\, \mbox{$G[S]$ is a threshold graph}\}$.

\begin{fact}\label{fa:tau}
For any graphs $G$ and $H$, we have $\tau(G \otimes H)\leq \tau(G)\tau(H).$
\end{fact}

\noindent {\bf Proof.}
Let $S$ be a subset of $V(G \otimes H)$ with $|S|=\tau(G \otimes H)$ such that $S$ induces a threshold graph.
Let $V(G)=\{x_1, x_2, \ldots, x_{|V(G)|}\}$.
For each $x_i\in V(G)$, let $S_{x_i}\colonequals \{y\in V(H)\colon\, (x_i,y)\in S\}$.
Note that some of the sets $S_{x_1}, S_{x_2}, \ldots, S_{x_{|V(G)|}}$ might be empty.
After relabeling the vertices of $G$, we may assume that there exists some $1\leq m\leq |V(G)|$ such that $S_{x_i}\neq \emptyset$ for $1\leq i\leq m$, and $S_{x_i}= \emptyset$ for $m+1\leq i\leq |V(G)|$.

In order to prove the result, we use the property that every induced subgraph of a threshold graph is also a threshold graph.
First, for each $i\in [m]$, the subgraph of $G \otimes H$ induced by $\{(x_i,y)\colon\, y\in S_{x_i}\}$ is an induced subgraph of $(G \otimes H)[S]$ and is isomorphic to $H[S_{x_i}]$.
Thus $H[S_{x_i}]$ is a threshold graph, and $|S_{x_i}|\leq \tau(H)$.
Second, for each $i\in [m]$, choose one vertex $y_i\in S_{x_i}$.
Then the subgraph of $G \otimes H$ induced by $\{(x_i,y_i)\colon\, i\in [m]\}$ is an induced subgraph of $(G \otimes H)[S]$ and is isomorphic to $G[\{x_1, \ldots, x_m\}]$.
Thus $G[\{x_1, \ldots, x_m\}]$ is a threshold graph, and $m\leq \tau(G)$.
Therefore, we have $\tau(G \otimes H)=|S|=\sum_{i\in [m]}|S_{x_i}|\leq \sum_{i\in [m]}\tau(H)=m\tau(H)\leq \tau(G)\tau(H).$
\hfill$\blacksquare$
\vspace{0.2cm}

Finally, we introduce the Cartesian product of two graphs.
Given two graphs $G$ and $H$, the {\it Cartesian product} $G \Box H$ is the graph defined as follows: its vertex set is $V(G)\times V(H)$, and two vertices $(x_1, y_1)$ and $(x_2, y_2)$ are adjacent if and only if either $y_1=y_2$ and $x_1x_2\in E(G)$, or $x_1=x_2$ and $y_1y_2\in E(H)$.
Note that the Cartesian product $K_{m}\Box K_{n}$ of two complete graphs $K_m$ and $K_n$ is the line graph of the complete graph $K_{m,n}$.
The graph $K_{m}\Box K_{n}$ is also called an {\it $m\times n$ grid}, or an {\it $m\times n$ rook's graph}.
A graph is said to be {\it locally $m\times n$ grid} if the induced subgraph on the neighbourhood of any vertex is isomorphic to an $m\times n$ grid.
We shall use the following lemma in our proof of $r'_2(6)=18$.

\begin{lemma}{\normalfont (see \cite[Section~10.2]{BrVaMa})}\label{le:grid}~
\begin{itemize}
\item[{\rm (i)}] The unique strongly regular graph with parameters $(9,4,1,2)$ is $K_{3}\Box K_{3}$.
\item[{\rm (ii)}] There are precisely two connected locally $3\times 3$ grid graphs, one on $16$ vertices and the other one on $20$ vertices.
\end{itemize}
\end{lemma}

\subsection{Lower bounds on $r'_2(s)$}
\label{subsec:pf_complete_lower}

In this subsection, we present three constructive lower bounds (Lemmas~\ref{le:r2slower-construction-1}, \ref{le:r2slower-construction-2} and \ref{le:r2slower-construction-3}) and a probabilistic lower bound (Lemma~\ref{le:r2slower}) on $r'_2(s)$.
In general, the constructive lower bounds are weaker than the probabilistic lower bound for large $s$.
However, for several small values of $s$, the constructive lower bounds are tight.
In particular, the lower bound in Lemma~\ref{le:r2slower-construction-1} is tight for $s\in \{3,4,6\}$, and the lower bound in Lemma~\ref{le:r2slower-construction-2} is tight for $s=5$; see Theorems~\ref{thm:r34}, \ref{thm:r5} and \ref{thm:r6}.

\begin{lemma}\label{le:r2slower-construction-1}
For any positive integers $s,a,b$ with $a+b=s+2$, we have $r'_2(s)\geq r(a,b)$, so
$r'_2(s)\geq \max\limits_{a,b\geq 1, a+b=s+2}r(a,b)$.
\end{lemma}

\noindent {\bf Proof.}
Let $\omega(T)$ and $\alpha(T)$ be the clique number and independence number of a graph $T$.
First, we show that for any threshold graph $T$ of order $s$, we have $\omega(T)+\alpha(T)\geq s+1$.
By the definition, we may assume that $V(T)=\{u_1, u_2, \ldots, u_s\}$ such that for each $2\leq i\leq s$, $u_i$ is either adjacent to or nonadjacent to all the vertices of $\{u_1, \ldots, u_{i-1}\}$ in $T$.
Let $X\subseteq \{u_2, \ldots, u_s\}$ be the set of vertices $u_i$ such that $u_i$ is adjacent to all the vertices of $\{u_1, \ldots, u_{i-1}\}$,
and $Y\subseteq \{u_2, \ldots, u_s\}$ be the set of vertices $u_i$ such that $u_i$ is nonadjacent to all the vertices of $\{u_1, \ldots, u_{i-1}\}$.
Then $X\cup \{u_1\}$ induces a clique of $T$ and $Y\cup \{u_1\}$ induces an independent set of $T$.
Hence, $\omega(T)+\alpha(T)\geq |X|+1+|Y|+1=s+1$.

Now we show that $r'_2(s)\geq r(a,b)$ whenever $a+b=s+2$.
Let $G$ be a graph on $r(a,b)-1$ vertices with $\omega(G)\leq a-1$ and $\alpha(G)\leq b-1$.
Suppose that $G$ contains an induced subgraph that is a threshold graph $T$ of order $s$.
Then we have $\omega(T)+\alpha(T)\geq s+1=a+b-1$ by the preceding argument.
However, $\omega(T)+\alpha(T)\leq \omega(G)+\alpha(G)\leq a-1+b-1=a+b-2$, a contradiction.
Therefore, we have $r'_2(s)\geq |V(G)|+1=r(a,b).$
\hfill$\blacksquare$
%\vspace{0.2cm}

\begin{lemma}\label{le:r2slower-construction-2}
For any positive integer $s\geq 3$, we have $r'_2(s)\geq (s-2)^2+1$.
\end{lemma}

\noindent {\bf Proof.}
Let $G\cong K_{s-2}\Box K_{s-2}$ be the Cartesian product of $K_{s-2}$ and $K_{s-2}$.
Note that $|V(G)|=(s-2)^2$.
Let $T$ be an induced subgraph of $G$ such that $|V(T)|=\tau(G)$ and $T$ is a threshold graph.
It suffices to show that $|V(T)|\leq s-1$.

Note that $T$ is possibly disconnected, but $T$ contains at most one component with a nonempty edge set since $T$ contains no induced $2K_2$ by Lemma~\ref{le:threshold-equi}~(ii).
If $E(T)=\emptyset$, then $|V(T)|\leq \alpha(G)=\alpha(K_{s-2}\Box K_{s-2})= s-2$, and we are done.
Hence, we may assume that $T$ consists of one component $T'$ with $E(T')\neq \emptyset$ and a set $T''$ of isolated vertices.

Since $G\cong K_{s-2}\Box K_{s-2}$ is the line graph of the complete graph $K_{s-2,s-2}$, $V(T)$ corresponds to a subset of edges of $K_{s-2,s-2}$.
Let $X$ and $Y$ be the partite sets of $V(K_{s-2, s-2})$.
We first consider $|V(T')|$.
Let $E'$ be the set of edges of $K_{s-2,s-2}$ corresponding to $V(T')$.
Since $T'$ is a connected threshold graph, it contains a dominating vertex $v$.
Let $xy$ be the edge of $K_{s-2, s-2}$ corresponding to $v$, where $x\in X$ and $y\in Y$.
Since $v$ is dominating, every edge in $E'\setminus \{xy\}$ is adjacent to $xy$, and thus incident with either $x$ or $y$.
Let $X'\colonequals \{z\in X\setminus \{x\}\colon\, zy\in E'\}$ and $Y'\colonequals \{z\in Y\setminus \{y\}\colon\, zx\in E'\}$.
Then $|V(T')|=|E'|=1+|X'|+|Y'|$.
We claim that $\min\{|X'|, |Y'|\}\leq 1$.
Indeed, if $|X'|\geq 2$ and $|Y'|\geq 2$, say $x_1, x_2 \in X'$ and $y_1, y_2\in Y'$, then the four vertices of $T$ corresponding to the edges $xy_1, xy_2, yx_1, yx_2$ induce a $2K_2$, contradicting Lemma~\ref{le:threshold-equi}~(ii).
We next consider $|T''|$.
If $T''=\emptyset$, then the preceding argument implies that $|V(T)|=|V(T')|= 1+|X'|+|Y'|\leq 1+1+s-3=s-1$, and we are done.
If $T''\neq \emptyset$, then the edges corresponding to $T''$ form a matching in $K_{s-2, s-2}$ and these edges are not adjacent to any edge of $E'$.
Thus we have $|T''|\leq \min\{s-2-1-|X'|, s-2-1-|Y'|\}$.
Without loss of generality, we may assume that $|X'|\leq |Y'|$.
Then we have $|X'|\leq 1$ and $|T''|\leq s-3-|Y'|$, so $|V(T)|=|V(T')|+|T''|\leq 1+|X'|+|Y'|+s-3-|Y'|=|X'|+s-2\leq s-1$.
This completes the proof of Lemma~\ref{le:r2slower-construction-2}.
\hfill$\blacksquare$
%\vspace{0.2cm}

\begin{lemma}\label{le:r2slower-construction-3}
For any positive integers $s$ and $t$, we have $r'_2((s-1)(t-1)+1)\geq (r'_2(s)-1)(r'_2(t)-1)+1$.
\end{lemma}

\noindent {\bf Proof.}
By the definition of $r'_2(s)$, choose an induced $\mathcal{T}_s$-free graph $G$ on $r'_2(s)-1$ vertices, and an induced $\mathcal{T}_t$-free graph $H$ on $r'_2(t)-1$ vertices.
We consider the lexicographic product $G \otimes H$ of $G$ and $H$.
By Fact~\ref{fa:tau}, we have $\tau(G \otimes H)\leq \tau(G)\tau(H)\leq (s-1)(t-1).$
Hence, we have $r'_2((s-1)(t-1)+1)\geq |V(G \otimes H)|+1= (r'_2(s)-1)(r'_2(t)-1)+1$.
\hfill$\blacksquare$
%\vspace{0.2cm}

\begin{lemma}\label{le:r2slower}
For any integer $s\geq 3$, we have
$$r'_2(s)> \left\lfloor2^{\frac{s+1}{2}+\frac{1}{s-2}\left(\frac{\ln\ln 2}{\ln 2}(s+1)+2-\frac{1+4\ln s}{\ln 2}\right)}\right\rfloor=\left(\sqrt{2}\ln 2-o(1)\right)2^{\frac{s}{2}}.$$
\end{lemma}

\noindent {\bf Proof.}
%\begin{proof}
Let $n=\left\lfloor2^{\frac{s+1}{2}+\frac{1}{s-2}\left(\frac{\ln\ln 2}{\ln 2}(s+1)+2-\frac{1+4\ln s}{\ln 2}\right)}\right\rfloor.$
If $n<s$, then $r'_2(s)\geq s>n$, so the result is immediate.
Hence, assume $n\geq s$.
Consider a random graph $G\sim G(n, \frac{1}{2})$, that is, the graph on $n$ labeled vertices, obtained by selecting each pair of vertices to be an edge randomly and independently with probability $\frac{1}{2}$.
For any fixed set $S\in {V(G)\choose s}$, let $A_S$ be the event that $S$ induces a threshold graph on $s$ vertices.
By Lemma~\ref{le:threshold}, the number of labeled threshold graphs on $s$ vertices is at most
$s!\left(\left(\frac{1}{\ln 2}-1\right)\left(\frac{1}{\ln 2}\right)^{s}+\frac{2\zeta(s)}{(2\pi)^s}\right),$
where $\zeta(s)$ is the Riemann zeta function.
Since $s\geq 3>2$, we have $\zeta(s)<\zeta(2)=\frac{\pi^2}{6}$, so $-\left(\frac{1}{\ln 2}\right)^{s}+\frac{2\zeta(s)}{(2\pi)^s}< -\left(\frac{1}{\ln 2}\right)^{2}+\frac{2\zeta(2)}{(2\pi)^2}=-\left(\frac{1}{\ln 2}\right)^{2}+\frac{2\frac{\pi^2}{6}}{(2\pi)^2}=-\left(\frac{1}{\ln 2}\right)^{2}+\frac{1}{12}<0.$
Thus $s!\left(\left(\frac{1}{\ln 2}-1\right)\left(\frac{1}{\ln 2}\right)^{s}+\frac{2\zeta(s)}{(2\pi)^s}\right)<s!\left(\frac{1}{\ln 2}\right)^{s+1}.$
Therefore,
$$\pr[A_S]\leq \frac{s!\left(\frac{1}{\ln 2}\right)^{s+1}}{2^{{s\choose 2}}}=s!(\ln 2)^{-(s+1)}2^{-\frac{s(s-1)}{2}}.$$
Let $p\colonequals s!(\ln 2)^{-(s+1)}2^{-\frac{s(s-1)}{2}}.$

Note that each event $A_S$ is mutually independent of all the events $A_{S'}$ with $|S\cap S'|\leq 1$.
Let $d$ be the number of events $A_{S'}$ with $S\neq S'$ and $|S\cap S'|\geq 2$.
Then $d\leq {s\choose 2}{n-2 \choose s-2}-1.$
We next show that $ep(d+1)\leq 1$.
Since $$n=\left\lfloor2^{\frac{s+1}{2}+\frac{1}{s-2}\left(\frac{\ln\ln 2}{\ln 2}(s+1)+2-\frac{1+4\ln s}{\ln 2}\right)}\right\rfloor\leq 2^{\frac{s+1}{2}+\frac{1}{s-2}\big((s+1)\log_2(\ln 2)+2-\log_2 e-4\log_2 s\big)},$$
we have
\begin{align*}
  \log_2 (ep(d+1)) \leq &~\log_2 \left(es!(\ln 2)^{-(s+1)}2^{-\frac{s(s-1)}{2}}{s\choose 2}{n-2 \choose s-2}\right) \\
  < &~\log_2 \left(es!(\ln 2)^{-(s+1)}2^{-\frac{s(s-1)}{2}}\frac{s^2}{2}\frac{n^{s-2}}{(s-2)!}\right) \\
  < &~\log_2 \left(es^4(\ln 2)^{-(s+1)}2^{-\frac{s(s-1)}{2}-1}n^{s-2}\right) \\
  = &~\log_2 e + 4\log_2 s - (s+1)\log_2(\ln 2)-\frac{s(s-1)}{2}-1+(s-2)\log_2 n\\
  \leq &~\log_2 e + 4\log_2 s - (s+1)\log_2(\ln 2)-\frac{(s+1)(s-2)}{2}-2+ \\
  ~&~(s-2)\left(\frac{s+1}{2}+\frac{1}{s-2}\big((s+1)\log_2(\ln 2)+2-\log_2 e-4\log_2 s\big)\right) \\
  = &~0.
\end{align*}
Hence, $ep(d+1)\leq 1$.
By Lemma~\ref{le:Local_Lemma}, we have
$$\pr\Big[\bigwedge\nolimits_{S\in {V(G)\choose s}}\overline{A_S}\Big]>0,$$
that is, with positive probability, no event $A_S$ occurs.
Thus there exists an induced $\mathcal{T}_s$-free graph on $n$ vertices.
Therefore, we have $r'_2(s)> n= \left\lfloor2^{\frac{s+1}{2}+\frac{1}{s-2}\left(\frac{\ln\ln 2}{\ln 2}(s+1)+2-\frac{1+4\ln s}{\ln 2}\right)}\right\rfloor.$

Moreover, since
$$\frac{1}{s-2}\left(\frac{\ln\ln 2}{\ln 2}(s+1)+2-\frac{1+4\ln s}{\ln 2}\right)=\frac{s+1}{s-2}\frac{\ln\ln 2}{\ln 2}+ \frac{2}{s-2} -\frac{1+4\ln s}{(s-2)\ln 2}=\log_2(\ln 2)-o(1)$$
as $s\to \infty$, we have
$$r'_2(s)>\left\lfloor2^{\frac{s+1}{2}+\frac{1}{s-2}\left(\frac{\ln\ln 2}{\ln 2}(s+1)+2-\frac{1+4\ln s}{\ln 2}\right)}\right\rfloor=\left(\sqrt{2}\ln 2-o(1)\right)2^{\frac{s}{2}}.$$
The proof of Lemma~\ref{le:r2slower} is complete.
%\end{proof}
\hfill$\blacksquare$
%\vspace{0.2cm}

\subsection{Upper bounds on $r'_2(s)$}
\label{subsec:pf_complete_upper}

In this subsection, we present an upper bound on $r'_2(s)$.
We start with a recursive upper bound.

\begin{lemma}\label{le:r2supper-recurrence}
Let $s\geq m\geq 1$ be integers.
Then
$r'_2(s)\leq \Big\lceil \frac{\left(3r'_2(m)-\varepsilon\right)2^{s-m}}{3}\Big\rceil,$
where $\varepsilon=1$ if $r'_2(m)$ is odd, and $\varepsilon=2$ if $r'_2(m)$ is even.
\end{lemma}

\noindent {\bf Proof.}
For $s\geq 2$, let
$$n(s)\colonequals
\left\{
   \begin{aligned}
    &2\cdot r'_2(s-1) & & \mbox{if $r'_2(s-1)$ is odd},\\
    &2\cdot r'_2(s-1)-1 & & \mbox{if $r'_2(s-1)$ is even}.
   \end{aligned}
   \right.$$

\begin{claim}\label{cl:r2supper-1}
For $s\geq 2$, we have $r'_2(s)\leq n(s)$.
\end{claim}

\begin{proof}
For a contradiction, suppose that $G$ is an induced $\mathcal{T}_s$-free graph on $n(s)$ vertices.
For an arbitrarily fixed vertex $v\in V(G)$, let $A\colonequals N_G(v)$ and $B\colonequals N_{\overline{G}}(v)$.
By Fact~\ref{fa:threshold-1}~(ii), we have $|A|\leq r'_2(s-1)-1$ and $|B|\leq r'_2(s-1)-1$.
Note that
$$\max\{|A|, |B|\}\geq \left\lceil\frac{n(s)-1}{2}\right\rceil=
\left\{
   \begin{aligned}
    &r'_2(s-1) & & \mbox{if $r'_2(s-1)$ is odd},\\
    &r'_2(s-1)-1 & & \mbox{if $r'_2(s-1)$ is even}.
   \end{aligned}
   \right.$$
Hence, we must have that $r'_2(s-1)$ is even and $|A|=|B|= r'_2(s-1)-1$.
Since the vertex $v$ was chosen arbitrarily, this implies that $G$ is a $(r'_2(s-1)-1)$-regular graph on $n(s)=2\cdot r'_2(s-1)-1$ vertices.
However, since both $r'_2(s-1)-1$ and $2\cdot r'_2(s-1)-1$ are odd, this is impossible.
This contradiction completes the proof of Claim~\ref{cl:r2supper-1}.
\end{proof}

\begin{claim}\label{cl:r2supper-2}
If $r'_2(m)$ is odd, then
$$\left\lceil \frac{\left(3r'_2(m)-\varepsilon\right)2^{s-m}}{3}\right\rceil=\left\lceil \frac{\left(3r'_2(m)-1\right)2^{s-m}}{3}\right\rceil=
\left\{
   \begin{aligned}
    &\frac{\left(3r'_2(m)-1\right)2^{s-m}+1}{3} & & \mbox{if $s-m$ is even},\\
    &\frac{\left(3r'_2(m)-1\right)2^{s-m}+2}{3} & & \mbox{if $s-m$ is odd}.
   \end{aligned}
   \right.$$
If $r'_2(m)$ is even, then
$$\left\lceil \frac{\left(3r'_2(m)-\varepsilon\right)2^{s-m}}{3}\right\rceil=\left\lceil \frac{\left(3r'_2(m)-2\right)2^{s-m}}{3}\right\rceil=
\left\{
   \begin{aligned}
    &\frac{\left(3r'_2(m)-2\right)2^{s-m}+2}{3} & & \mbox{if $s-m$ is even},\\
    &\frac{\left(3r'_2(m)-2\right)2^{s-m}+1}{3} & & \mbox{if $s-m$ is odd}.
   \end{aligned}
   \right.$$
\end{claim}

\begin{proof}
Note that $2^{s-m} \equiv 1 \pmod{3}$ if $s-m$ is even, and $2^{s-m} \equiv 2 \pmod{3}$ if $s-m$ is odd.
First, assume that $r'_2(m)$ is odd.
Note that $3r'_2(m)-1 \equiv 2 \pmod{3}$.
Thus $\left(3r'_2(m)-1\right)2^{s-m} \equiv 2 \pmod{3}$ when $s-m$ is even, and $\left(3r'_2(m)-1\right)2^{s-m} \equiv 1 \pmod{3}$ when $s-m$ is odd.
This proves the claim when $r'_2(m)$ is odd.
Second, assume that $r'_2(m)$ is even.
Note that $3r'_2(m)-2 \equiv 1 \pmod{3}$.
Thus $\left(3r'_2(m)-2\right)2^{s-m} \equiv 1 \pmod{3}$ when $s-m$ is even, and $\left(3r'_2(m)-2\right)2^{s-m} \equiv 2 \pmod{3}$ when $s-m$ is odd.
This proves the claim when $r'_2(m)$ is even.
\end{proof}

Define $q_0\colonequals r'_2(m)$ and
$$q_{i}=
\left\{
   \begin{aligned}
    &2q_{i-1} & & \mbox{if $q_{i-1}$ is odd},\\
    &2q_{i-1}-1 & & \mbox{if $q_{i-1}$ is even}
   \end{aligned}
   \right.$$
for $i\geq 1$.
By Claim~\ref{cl:r2supper-1}, we have $r'_{2}(m+i)\leq q_i$ for every $i\geq 0$.
Set $\delta_i\colonequals 0$ if $q_i$ is odd, and $\delta_i\colonequals 1$ if $q_i$ is even.
Then $q_{i+1}=2q_i-\delta_i$ and
\begin{equation}\label{eq:r2supper-1}
r'_{2}(s)\leq q_{s-m}= 2^{s-m}q_0-\sum\nolimits_{i=0}^{s-m-1}2^{s-m-1-i}\delta_i= 2^{s-m}r'_2(m)-\sum\nolimits_{i=0}^{s-m-1}2^{s-m-1-i}\delta_i.
\end{equation}

We divide the rest of the proof into two cases according to the parity of $r'_2(m)$.
\vspace{0.2cm}

{\bf Case~1.} $q_0=r'_2(m)$ is odd.
\vspace{0.2cm}

In this case, $q_0, q_2, q_4, \cdots$ are odd, and $q_1, q_3, q_5, \cdots$ are even, so $\delta_i= 0$ if $i$ is even, and $\delta_i= 1$ if $i$ is odd.
Thus $\sum\nolimits_{i=0}^{s-m-1}2^{s-m-1-i}\delta_i=\sum\nolimits_{\scriptsize\mbox{$0\leq i\leq s-m-1$, $i$ is odd}}2^{s-m-1-i}\delta_i.$
If $s-m=2\ell$ is even, then
$$\sum\nolimits_{i=0}^{s-m-1}2^{s-m-1-i}\delta_i=2^{2\ell-2}+2^{2\ell-4}+\cdots+1=1+4+\cdots+4^{\ell-1}=\frac{4^{\ell}-1}{3}=\frac{2^{s-m}-1}{3}.$$
Combining with Inequality~(\ref{eq:r2supper-1}), we have
$r'_{2}(s)\leq 2^{s-m}r'_2(m)-\frac{2^{s-m}-1}{3}=\frac{\left(3r'_2(m)-1\right)2^{s-m}+1}{3}.$
If $s-m=2\ell+1$ is odd, then
$$\sum\nolimits_{i=0}^{s-m-1}2^{s-m-1-i}\delta_i=2^{2\ell-1}+2^{2\ell-3}+\cdots+2=2(1+4+\cdots+4^{\ell-1})=\frac{2(4^{\ell}-1)}{3}=\frac{2^{s-m}-2}{3}.$$
Combining with Inequality~(\ref{eq:r2supper-1}), we have
$r'_{2}(s)\leq 2^{s-m}r'_2(m)-\frac{2^{s-m}-2}{3}=\frac{\left(3r'_2(m)-1\right)2^{s-m}+2}{3}.$
By Claim~\ref{cl:r2supper-2}, we further have $r'_2(s)\leq \Big\lceil \frac{\left(3r'_2(m)-\varepsilon\right)2^{s-m}}{3}\Big\rceil.$
\vspace{0.2cm}

{\bf Case~2.} $q_0=r'_2(m)$ is even.
\vspace{0.2cm}

In this case, $q_0, q_2, q_4, \cdots$ are even, and $q_1, q_3, q_5, \cdots$ are odd, so $\delta_i= 1$ if $i$ is even, and $\delta_i= 0$ if $i$ is odd.
Thus $\sum\nolimits_{i=0}^{s-m-1}2^{s-m-1-i}\delta_i=\sum\nolimits_{\scriptsize\mbox{$0\leq i\leq s-m-1$, $i$ is even}}2^{s-m-1-i}\delta_i.$
If $s-m=2\ell$ is even, then
$$\sum\nolimits_{i=0}^{s-m-1}2^{s-m-1-i}\delta_i=2^{2\ell-1}+2^{2\ell-3}+\cdots+2=2(1+4+\cdots+4^{\ell-1})=\frac{2(4^{\ell}-1)}{3}=\frac{2(2^{s-m}-1)}{3}.$$
Combining with Inequality~(\ref{eq:r2supper-1}), we have
$r'_{2}(s)\leq 2^{s-m}r'_2(m)-\frac{2(2^{s-m}-1)}{3}=\frac{\left(3r'_2(m)-2\right)2^{s-m}+2}{3}.$
If $s-m=2\ell+1$ is odd, then
$$\sum\nolimits_{i=0}^{s-m-1}2^{s-m-1-i}\delta_i=2^{2\ell}+2^{2\ell-2}+\cdots+1=1+4+\cdots+4^{\ell}=\frac{4^{\ell+1}-1}{3}=\frac{2^{s-m+1}-1}{3}.$$
Combining with Inequality~(\ref{eq:r2supper-1}), we have
$r'_{2}(s)\leq 2^{s-m}r'_2(m)-\frac{2^{s-m+1}-1}{3}=\frac{\left(3r'_2(m)-2\right)2^{s-m}+1}{3}.$
By Claim~\ref{cl:r2supper-2}, we further have $r'_2(s)\leq \Big\lceil \frac{\left(3r'_2(m)-\varepsilon\right)2^{s-m}}{3}\Big\rceil.$
\hfill$\blacksquare$
\vspace{0.2cm}

Using Theorem~\ref{thm:r6} below, Lemma~\ref{le:r2supper-recurrence} immediately yields an upper bound on $r'_2(s)$.

\begin{corollary}\label{cor:r2supper}
For any integer $s\geq 6$, we have $r'_2(s)\leq \left\lceil\frac{13}{48}2^s\right\rceil.$
\end{corollary}

\noindent {\bf Proof.}
By Theorem~\ref{thm:r6}, we have $r'_2(6)=18$.
Combining with Lemma~\ref{le:r2supper-recurrence}, we have
$r'_2(s)\leq \left\lceil \frac{\left(3r'_2(6)-2\right)2^{s-6}}{3}\right\rceil= \left\lceil\frac{\left(54-2\right)2^{s-6}}{3}\right\rceil=\left\lceil\frac{13}{48}2^s\right\rceil.$
\hfill$\blacksquare$
%\vspace{0.2cm}

\subsection{Exact values of $r'_2(s)$ for small values of $s$}
\label{subsec:pf_complete_small}

In this subsection, we show that the exact values of $r'_2(s)$ are $3, 6, 10, 18$ for $s=3, 4, 5, 6$, respectively.

\begin{theorem}\label{thm:r34}
$r'_2(3)=3$ and $r'_2(4)=6$.
\end{theorem}

\noindent {\bf Proof.}
By Corollary~\ref{cor:forest}~(i), we have $r'_2(3)=r'_2(K_3)=|V(K_3)|=3$.
In the following, we show that $r'_2(4)=6$.
For the lower bound, we have $r'_2(4)\geq r(3,3)=6$ by Lemma~\ref{le:r2slower-construction-1}.
For the upper bound, let $G$ be an arbitrary graph on six vertices.
Suppose for a contradiction that $G$ is induced $\mathcal{T}_4$-free.
Let $v\in V(G)$.
Then one of $|N_G(v)|\geq 3$ and $|N_{\overline{G}}(v)|\geq 3$ holds.
By Fact~\ref{fa:threshold-1}~(ii), both $G[N_G(v)]$ and $G[N_{\overline{G}}(v)]$ are induced $\mathcal{T}_{3}$-free.
This contradicts $r'_2(3)=3$, and thus completes the proof.
\hfill$\blacksquare$
\vspace{0.2cm}

We shall use the following concept in our proofs of $r'_2(5)=10$ and $r'_2(6)=18$.
Given a graph $G$, we call a 3-set $\{x,y,z\}\in {V(G)\choose 3}$ {\it homogeneous} if either $G[\{x,y,z\}]\cong K_3$ or $G[\{x,y,z\}]\cong \overline{K_3}$, and {\it nonhomogeneous} otherwise.
For each vertex $v\in V(G)$, let $h(v)$ be the number of homogeneous 3-sets containing $v$.
Let $h(G)$ be the number of homogeneous 3-sets of $G$, and $h'(G)$ be the number of nonhomogeneous 3-sets of $G$.
Note that
\begin{equation}\label{eq:hom-1}
h(G)+h'(G)={|V(G)|\choose 3},
\end{equation}
\begin{equation}\label{eq:hom-0}
h(v)=\left|E\left(G[N_G(v)]\right)\right|+\left|E\left(\overline{G}[N_{\overline{G}}(v)]\right)\right|,
\end{equation}
\begin{equation}\label{eq:hom-2}
h(G)= \frac{1}{3}\sum\nolimits_{v\in V(G)}h(v),
\end{equation}
\begin{equation}\label{eq:hom-3}
h'(G)= \frac{1}{2}\sum\nolimits_{v\in V(G)}|N_G(v)|\cdot |N_{\overline{G}}(v)|.
\end{equation}
We shall also use the following fact.

\begin{fact}\label{fa:r56}
Let $G$ be an induced $\mathcal{T}_5$-free graph.
Then the following statements hold.
\begin{itemize}
\item[{\rm (i)}] $G$ contains no five vertices $v_1, v_2, \ldots, v_5$ such that $v_1$ and $v_2$ are isolated vertices in $G[\{v_1, v_2, \ldots, v_5\}]$.
\item[{\rm (ii)}] $G$ contains no five vertices $v_1, v_2, \ldots, v_5$ with $v_1v_i\in E(G)$ for all $i\in \{2,3,4,5\}$ and $v_2v_j\notin E(G)$ for all $j\in \{3,4,5\}$.
\item[{\rm (iii)}] $G$ contains no five vertices $v_1, v_2, \ldots, v_5$ with $v_1v_i\notin E(G)$ for all $i\in \{2,3,4,5\}$ and $v_2v_j\in E(G)$ for all $j\in \{3,4,5\}$.
\item[{\rm (iv)}] For any vertex $v\in V(G)$, we have $|N_G(v)|\leq 5$ and $|N_{\overline{G}}(v)|\leq 5$.
\end{itemize}
\end{fact}

\noindent {\bf Proof.}
Note that every graph on three vertices is a threshold graph of order 3.
Since $G$ is induced $\mathcal{T}_5$-free, we must have (i), (ii) and (iii).
Otherwise, we can obtain a threshold graph of order 5 as follows: start with the remaining three vertices in the ordering satisfying the definition of the threshold graph,
and add $v_1, v_2$ as isolated vertices in (i);
add $v_2$ as an isolated vertex and then $v_1$ as a dominating vertex in (ii);
add $v_2$ as a dominating vertex and then $v_1$ as an isolated vertex in (iii).

Next, we prove (iv).
For any vertex $v\in V(G)$, both $G[N_G(v)]$ and $G[N_{\overline{G}}(v)]$ are induced $\mathcal{T}_{4}$-free by Fact~\ref{fa:threshold-1}~(ii).
Since $r'_2(4)=6$ by Theorem~\ref{thm:r34}, we have $|N_G(v)|\leq 5$ and $|N_{\overline{G}}(v)|\leq 5$.
\hfill$\blacksquare$
%\vspace{0.2cm}

\begin{theorem}\label{thm:r5}
$r'_2(5)=10$.
\end{theorem}

\noindent {\bf Proof.}
For the lower bound, we have $r'_2(5)\geq (5-2)^2+1=10$ by Lemma~\ref{le:r2slower-construction-2}.
For the upper bound, let $G$ be an arbitrary graph on ten vertices.
Suppose for a contradiction that $G$ is induced $\mathcal{T}_5$-free.

For any vertex $v\in V(G)$, we have $|N_G(v)|\leq 5$ and $|N_{\overline{G}}(v)|\leq 5$ by Fact~\ref{fa:r56}~(iv).
Note that $|N_G(v)|+|N_{\overline{G}}(v)|=9$.
Thus either $|N_G(v)|=5$ and $|N_{\overline{G}}(v)|=4$, or $|N_G(v)|=4$ and $|N_{\overline{G}}(v)|=5$.

\begin{claim}\label{cl:r5-1}
For every vertex $v\in V(G)$, we have $h(v)\geq 7$.
\end{claim}

\begin{proof}
If $|N_G(v)|=5$ and $|N_{\overline{G}}(v)|=4$, then $G[N_G(v)]\cong C_5$ and $G[N_{\overline{G}}(v)]\in \{P_4, C_4, 2K_2\}$ by Fact~\ref{fa:threshold-1}~(ii), (iii) and (iv).
Hence, there are five homogeneous 3-sets containing $v$ and two vertices from $N_G(v)$,
and at least two homogeneous 3-sets containing $v$ and two vertices from $N_{\overline{G}}(v)$.
Hence, we have $h(v)\geq 7$.
If $|N_G(v)|=4$ and $|N_{\overline{G}}(v)|=5$, then we also have $h(v)\geq 7$ by symmetry.
\end{proof}

By Claim~\ref{cl:r5-1} and Equality~(\ref{eq:hom-2}), we have $h(G)= \frac{1}{3}\sum_{v\in V(G)}h(v)\geq \frac{1}{3}\cdot 10\cdot 7>23$.
Moreover, we have $|N_G(v)|\cdot |N_{\overline{G}}(v)|=20$, and thus $h'(G)= \frac{1}{2}\sum_{v\in V(G)}|N_G(v)|\cdot |N_{\overline{G}}(v)|= \frac{1}{2}\cdot 10\cdot 20=100$ by Equality~(\ref{eq:hom-3}).
Then $h(G)+h'(G)> 23+100=123$.
However, we have $h(G)+h'(G)={10\choose 3}=120<123$ by Equality~(\ref{eq:hom-1}), a contradiction.
This completes the proof of Theorem~\ref{thm:r5}.
\hfill$\blacksquare$
%\vspace{0.2cm}

\begin{theorem}\label{thm:r6}
$r'_2(6)=18$.
\end{theorem}

In order to prove Theorem~\ref{thm:r6}, we first state and prove two additional lemmas.

\begin{lemma}\label{le:r6-8-vtx}
Let $G$ be an $8$-vertex induced $\mathcal{T}_5$-free graph.
Then $|E(G)|\geq 10$, with equality if and only if the degree sequence of $G$ is $(2,2,2,2,3,3,3,3)$.
\end{lemma}

\noindent {\bf Proof.}
For any vertex $v\in V(G)$, we have $|N_G(v)|\leq 5$ and $|N_{\overline{G}}(v)|\leq 5$ by Fact~\ref{fa:r56}~(iv).
Note that $|N_G(v)|+|N_{\overline{G}}(v)|=7$.
Thus we further have $|N_G(v)|\geq 2$ and $|N_{\overline{G}}(v)|\geq 2$, so $2\leq \delta(G)\leq \Delta(G)\leq 5$.

\begin{claim}\label{cl:r6-8-vtx-1}
$|E(G)|\geq 10.$
\end{claim}

\begin{proof}
For a contradiction, suppose that $|E(G)|\leq 9.$
This implies that $\delta(G)\leq 2$, which together with $\delta(G)\geq 2$ implies that $G$ contains a vertex $v$ of degree 2.
Let $N_G(v)=\{x, y\}$.
By Fact~\ref{fa:threshold-1}~(ii) and (iv), we have $G[N_{\overline{G}}(v)]\cong C_5$.
Now $\left|E\left(N_G(v), N_{\overline{G}}(v)\right)\right|\leq 9-2-5=2$.
Then there exist at least three vertices $v_1, v_2, v_3\in N_{\overline{G}}(v)$ such that $E(\{x, y\},\{v_1, v_2, v_3\})=\emptyset$.
Since $G[N_{\overline{G}}(v)]\cong C_5$, at least two vertices of $\{v_1, v_2, v_3\}$ are nonadjacent in $G$, say $v_1$ and $v_2$.
Then the five vertices $v_1, v_2, x, y, v$ violate Fact~\ref{fa:r56}~(i), a contradiction.
\end{proof}

Now assume that $|E(G)|= 10$, and we show that the degree sequence of $G$ is $(2,2,2,2,3,3,3,3)$.

\begin{claim}\label{cl:r6-8-vtx-2}
$\Delta(G)\leq 4.$
\end{claim}

\begin{proof}
For a contradiction, suppose that $d_G(v)\geq 5$ for some $v\in V(G)$.
This together with $\Delta(G)\leq 5$ implies that $d_G(v)=5$.
By Fact~\ref{fa:threshold-1}~(ii) and (iv), we have $G[N_{G}(v)]\cong C_5$.
Now $G[\{v\}\cup N_{G}(v)]$ already has ten edges, so the two vertices of $N_{\overline{G}}(v)$ are isolated vertices in $G$, contradicting $\delta(G)\geq 2$.
\end{proof}

\begin{claim}\label{cl:r6-8-vtx-3}
$\Delta(G)\leq 3.$
\end{claim}

\begin{proof}
For a contradiction, suppose that $d_G(v)\geq 4$ for some $v\in V(G)$.
Combining with Claim~\ref{cl:r6-8-vtx-2}, we have $d_G(v)=4$.
Let $A\colonequals N_{G}(v)$ and $B\colonequals N_{\overline{G}}(v)$.
Then $|A|=4$ and $|B|=3$.
By Fact~\ref{fa:threshold-1}~(ii) and (iii), we have $G[A]\in \{P_4, C_4, 2K_2\}.$
Moreover,
\begin{equation}\label{eq:r6-8-vtx-3-1}
|E(G[B])|+|E(A,B)|=|E(G)|-4-|E(G[A])|=10-4-|E(G[A])|=6-|E(G[A])|.
\end{equation}
Since $\delta(G)\geq 2$, we have
\begin{equation}\label{eq:r6-8-vtx-3-2}
2|E(G[B])|+|E(A,B)|=\sum\nolimits_{u\in B}d_G(u)\geq 2|B|=6.
\end{equation}

If $G[A]\cong P_4$, then $|E(G[B])|+|E(A,B)|=3$ by Equality~(\ref{eq:r6-8-vtx-3-1}).
This together with Inequality~(\ref{eq:r6-8-vtx-3-2}) implies that $|E(G[B])|=3$ and $|E(A,B)|=0$.
Let $v_1$ and $v_2$ be two nonadjacent vertices of $A$.
Then $v_1$, $v_2$ and the three vertices of $B$ violate Fact~\ref{fa:r56}~(i), a contradiction.

If $G[A]\cong C_4$, then $|E(G[B])|+|E(A,B)|=2$ by Equality~(\ref{eq:r6-8-vtx-3-1}).
This implies that $2|E(G[B])|+|E(A,B)|\leq 4<6,$ contradicting Inequality~(\ref{eq:r6-8-vtx-3-2}).

If $G[A]\cong 2K_2$, then $|E(G[B])|+|E(A,B)|=4$ by Equality~(\ref{eq:r6-8-vtx-3-1}).
This together with Inequality~(\ref{eq:r6-8-vtx-3-2}) implies that $|E(G[B])|\geq 2$.
First, suppose that $|E(G[B])|=3$.
Then $|E(A,B)|=4-|E(G[B])|=1$.
This implies that at least three vertices of $A$ are nonadjacent to any vertex of $B$.
Thus we can choose two nonadjacent vertices of $A$ that together with the three vertices of $B$ violate Fact~\ref{fa:r56}~(i), a contradiction.
Therefore, we have $|E(G[B])|=2$ and $|E(A,B)|=4-|E(G[B])|=2$.
Then there exist vertices $u_1, u_2\in A$ and $v_1, v_2\in B$ such that $E(\{u_1, u_2\}, B)=\emptyset$ and $v_1v_2\notin E(G)$.
Then the five vertices $v_1, v_2, u_1, u_2, v$ violate Fact~\ref{fa:r56}~(i), a contradiction.
\end{proof}

By Claim~\ref{cl:r6-8-vtx-3}, we have $\Delta(G)\leq 3$, so $2\leq \delta(G)\leq \Delta(G)\leq 3$.
Then since $|V(G)|=8$ and $\sum_{v\in V(G)}d_G(v)=2|E(G)|=20$, $G$ contain exactly four vertices of degree 2 and four vertices of degree 3.
Hence, the degree sequence of $G$ is $(2,2,2,2,3,3,3,3)$ when $|E(G)|=10$.
Conversely, this degree sequence has degree sum 20, and thus $|E(G)|=10$.
The proof of Lemma~\ref{le:r6-8-vtx} is complete.
\hfill$\blacksquare$
%\vspace{0.2cm}

\begin{lemma}\label{le:r6-9-vtx}
Up to isomorphism, the unique $9$-vertex induced $\mathcal{T}_5$-free graph is $K_{3}\Box K_{3}$.
\end{lemma}

\noindent {\bf Proof.}
By the proof of Lemma~\ref{le:r2slower-construction-2}, $K_{3}\Box K_{3}$ is a $9$-vertex induced $\mathcal{T}_5$-free graph.
It remains to prove the uniqueness.
Let $G$ be a $9$-vertex induced $\mathcal{T}_5$-free graph.
For any vertex $v\in V(G)$, we have $|N_G(v)|\leq 5$ and $|N_{\overline{G}}(v)|\leq 5$ by Fact~\ref{fa:r56}~(iv).
Note that $|N_G(v)|+|N_{\overline{G}}(v)|=8$.
Thus we further have $|N_G(v)|\geq 3$ and $|N_{\overline{G}}(v)|\geq 3$, so $3\leq \delta(G)\leq \Delta(G)\leq 5$.

\begin{claim}\label{cl:r6-9-vtx-1}
For any vertex $v\in V(G)$, if $d_G(v)=3$, then $E(G[N_G(v)])\neq \emptyset$ and $G[N_{\overline{G}}(v)]\cong C_5$; if $d_G(v)=5$, then $E(\overline{G}[N_{\overline{G}}(v)])\neq \emptyset$ and $G[N_G(v)]\cong C_5$.
\end{claim}

\begin{proof}
We first show that if $d_G(v)=3$, then $E(G[N_G(v)])\neq \emptyset$ and $G[N_{\overline{G}}(v)]\cong C_5$.
Let $N_G(v)=\{a_1, a_2, a_3\}$ and $B\colonequals N_{\overline{G}}(v)$.
By Fact~\ref{fa:threshold-1}~(ii) and (iv), we have $G[N_{\overline{G}}(v)]=G[B]\cong C_5$.
Now suppose for a contradiction that $E(G[N_G(v)])= \emptyset$.
For each $i\in [3]$, let $S_i\colonequals N_G(a_i)\cap B$.

Since $\delta(G)\geq 3$ and $E(G[N_G(v)])= \emptyset$, we have $|S_i|\geq 2$ for all $i\in [3]$.
On the other hand, we also have $|S_i|\leq 2$ for all $i\in [3]$.
To see this, if $a_i$ has three neighbors $b_1, b_2, b_3$ in $B$, then the five vertices $a_i, v, b_1, b_2, b_3$ violate Fact~\ref{fa:r56}~(ii), a contradiction.
Therefore, we have $|S_i|= 2$ for all $i\in [3]$.
Moreover, we claim that the two vertices in $S_i$ are nonadjacent.
Indeed, if the two vertices in $S_i$, say $b_1, b_2$, are adjacent, then since $G[B]\cong C_5$, there exists a vertex $b_3\in B\setminus S_i$ such that $b_3b_1, b_3b_2 \notin E(G)$.
Now the five vertices $b_3, a_i, v, b_1, b_2$ violate Fact~\ref{fa:r56}~(iii), a contradiction.

Furthermore, we claim that $S_1\cup S_2\cup S_3=B$.
Indeed, if there exists a vertex $b_1\in B\setminus (S_1\cup S_2\cup S_3)$, then the five vertices $b_1, v, a_1, a_2, a_3$ violate Fact~\ref{fa:r56}~(iii), a contradiction.
Therefore, we have $|S_1\cup S_2\cup S_3|=5$.
This together with $|S_1|=|S_2|=|S_3|=2$ implies that two of $S_1, S_2, S_3$, say $S_1, S_2$, satisfy $|S_1\cap S_2|=1$ and $(S_1\cup S_2)\cap S_3=\emptyset$.
Recall that the two vertices in $S_3$, say $b_1, b_2$, are nonadjacent.
Then the five vertices $b_1, b_2, a_1, a_2, v$ violate Fact~\ref{fa:r56}~(i), a contradiction.
Therefore, we have $E(G[N_G(v)])\neq \emptyset$.

Note that $\overline{G}$ is also a $9$-vertex induced $\mathcal{T}_5$-free graph, and moreover, $\overline{C_5}\cong C_5$.
Since $d_G(v)+d_{\overline{G}}(v)=8$, we have $d_{\overline{G}}(v)=3$ when $d_G(v)=5$.
Hence, applying the same argument to $\overline{G}$, we have that if $d_G(v)=5$, then $E(\overline{G}[N_{\overline{G}}(v)])\neq \emptyset$ and $G[N_G(v)]\cong C_5$.
\end{proof}

Let $x$ be the number of vertices of degree 3 or 5 in $G$.
By Equality~(\ref{eq:hom-3}) and since $3\leq \delta(G)\leq \Delta(G)\leq 5$, we have
\begin{align*}
h'(G)= &~\frac{1}{2}\sum\nolimits_{v\in V(G)}|N_G(v)|\cdot |N_{\overline{G}}(v)| =\frac{1}{2}\sum\nolimits_{v\in V(G)} d_G(v)(8-d_G(v))\\
= &~\frac{1}{2}(x\cdot 3\cdot 5 + (9-x)\cdot 4\cdot 4) = 72-\frac{x}{2}.
\end{align*}
Combining with Equalities~(\ref{eq:hom-1}) and (\ref{eq:hom-2}), we further have
\begin{equation}\label{eq:r6-9-vtx-1}
\sum\nolimits_{v\in V(G)}h(v)=3h(G)=3\left({9\choose 3}-h'(G)\right) = 3\left(84-\left(72-\frac{x}{2}\right)\right)=36+\frac{3 x}{2}.
\end{equation}

\begin{claim}\label{cl:r6-9-vtx-2}
If $d_G(v)\in \{3,5\}$, then $h(v)\geq 6$.
If $d_G(v)=4$, then $h(v)\geq 4$.
Moreover, if $d_G(v)=4$ and $h(v)= 4$, then $G[N_G(v)]\cong 2K_2$ and $G[N_{\overline{G}}(v)]\cong C_4$.
\end{claim}

\begin{proof}
By Claim~\ref{cl:r6-9-vtx-1}, if $d_G(v)\in \{3,5\}$, then either $E(G[N_G(v)])\neq \emptyset$ and $G[N_{\overline{G}}(v)]\cong C_5$, or $E(\overline{G}[N_{\overline{G}}(v)])\neq \emptyset$ and $G[N_G(v)]\cong C_5$. Thus $h(v)\geq 1+5=6$ by Equality~(\ref{eq:hom-0}).

If $d_G(v)=4$, then $G[N_G(v)], G[N_{\overline{G}}(v)]\in \{P_4, C_4, 2K_2\}$ by Fact~\ref{fa:threshold-1}~(ii) and (iii).
Thus $h(v)\geq 2+2=4$ by Equality~(\ref{eq:hom-0}), with equality if and only if $G[N_G(v)]\cong 2K_2$ and $G[N_{\overline{G}}(v)]\cong C_4$.
\end{proof}

By Claim~\ref{cl:r6-9-vtx-2}, we have $\sum\nolimits_{v\in V(G)}h(v)\geq 6x+4(9-x)=36+2x$.
Combining with Equality~(\ref{eq:r6-9-vtx-1}), we have $36+2x\leq \sum\nolimits_{v\in V(G)}h(v) =36+\frac{3 x}{2},$ so $x=0$ and $\sum\nolimits_{v\in V(G)}h(v)=36$.
In particular, $G$ is a $4$-regular graph, and for all vertices $v\in V(G)$, we have $h(v)=4$, $G[N_G(v)]\cong 2K_2$ and $G[N_{\overline{G}}(v)]\cong C_4$ by Claim~\ref{cl:r6-9-vtx-2}.
It follows that for every vertex $u\in N_G(v)$, we have $|N_G(u)\cap N_G(v)|=1$.
For every vertex $u\in N_{\overline{G}}(v)$, it has two neighbors inside the $C_4\cong G[N_{\overline{G}}(v)]$, and since $d_G(u)=4$, its other two neighbors lie in $N_G(v)$. Hence, we have $|N_G(u)\cap N_G(v)|=2$.
Therefore, $G$ is a strongly regular graph with parameters $(9,4,1,2)$.
By Lemma~\ref{le:grid}~(i), we have $G\cong K_{3}\Box K_{3}$.
\hfill$\blacksquare$
\vspace{0.2cm}

Now we have all the ingredients to present our proof of Theorem~\ref{thm:r6}.
\vspace{0.2cm}

\noindent {\bf Proof of Theorem~\ref{thm:r6}.}
For the lower bound, we have $r'_2(6)\geq r(4,4)=18$ (see~\cite{GrGl}) by Lemma~\ref{le:r2slower-construction-1}.
For the upper bound, let $G$ be an arbitrary graph on 18 vertices.
Suppose for a contradiction that $G$ is induced $\mathcal{T}_6$-free.

For any vertex $v\in V(G)$, we have $|N_G(v)|\leq 9$ and $|N_{\overline{G}}(v)|\leq 9$ by Fact~\ref{fa:threshold-1}~(ii) and Theorem~\ref{thm:r5}.
Note that $|N_G(v)|+|N_{\overline{G}}(v)|=17$.
Thus either $|N_G(v)|=8$ and $|N_{\overline{G}}(v)|=9$, or $|N_G(v)|=9$ and $|N_{\overline{G}}(v)|=8$.
In particular, we have $d_G(v)\in \{8,9\}$ for every vertex $v\in V(G)$.
Thus by Equality~(\ref{eq:hom-3}), we have
$$h'(G)= \frac{1}{2}\sum\nolimits_{v\in V(G)}|N_G(v)|\cdot |N_{\overline{G}}(v)| =\frac{1}{2}\cdot 18 \cdot 8 \cdot 9=648.$$
Combining with Equalities~(\ref{eq:hom-1}) and (\ref{eq:hom-2}), we further have
\begin{equation}\label{eq:r6-1}
\sum\nolimits_{v\in V(G)}h(v)=3h(G)=3\left({18\choose 3}-h'(G)\right) = 3\left(816-648\right)=3\cdot 168=18\cdot 28.
\end{equation}

%\begin{claim}\label{cl:r6--1}
%If $d_G(v)=8$, then the degree sequence of $G[N_G(v)]$ is $(2,2,2,2,3,3,3,3)$.
%If $d_G(v)=9$, then $G[N_G(v)]\cong K_{3}\Box K_{3}$.
%\end{claim}

\begingroup
\renewcommand{\theclaim}{\ref{thm:r6}.1}
\begin{claim}\label{cl:r6--1}
If $d_G(v)=8$, then the degree sequence of $G[N_G(v)]$ is $(2,2,2,2,3,3,3,3)$.
If $d_G(v)=9$, then $G[N_G(v)]\cong K_{3}\Box K_{3}$.
\end{claim}
\endgroup

\begin{proof}
By Fact~\ref{fa:threshold-1}~(ii), both $G[N_G(v)]$ and $G[N_{\overline{G}}(v)]$ are induced $\mathcal{T}_5$-free.
Then, by Fact~\ref{fa:threshold-1}~(i), both $\overline{G}[N_G(v)]$ and $\overline{G}[N_{\overline{G}}(v)]$ are also induced $\mathcal{T}_5$-free.
Recall that we have either $|N_G(v)|=8$ and $|N_{\overline{G}}(v)|=9$, or $|N_G(v)|=9$ and $|N_{\overline{G}}(v)|=8$.
By Lemmas~\ref{le:r6-8-vtx} and \ref{le:r6-9-vtx}, we have $|E(G[N_G(v)])|\geq 10$ and $|E(\overline{G}[N_{\overline{G}}(v)])|=|E(K_{3}\Box K_{3})|=18$ in the former case,
and $|E(G[N_G(v)])|=|E(K_{3}\Box K_{3})|=18$ and $|E(\overline{G}[N_{\overline{G}}(v)])|\geq 10$ in the latter case.
Thus for every vertex $v\in V(G)$, we have $h(v)=|E(G[N_G(v)])|+|E(\overline{G}[N_{\overline{G}}(v)])|\geq 28$ by Equality~(\ref{eq:hom-0}).
Hence, $\sum\nolimits_{v\in V(G)}h(v)\geq 18\cdot 28$.
Combining with Equality~(\ref{eq:r6-1}), we can deduce that for every vertex $v\in V(G)$, $h(v)=28$, $|E(G[N_G(v)])|= 10$ when $d_G(v)=8$, and $G[N_G(v)]\cong K_{3}\Box K_{3}$ when $d_G(v)=9$.
Moreover, by Lemma~\ref{le:r6-8-vtx}, the degree sequence of $G[N_G(v)]$ is $(2,2,2,2,3,3,3,3)$ when $d_G(v)=8$.
\end{proof}

Let $X\colonequals \{v\in V(G)\colon\, d_G(v)=8\}$ and $Y\colonequals \{v\in V(G)\colon\, d_G(v)=9\}$.

%\begin{claim}\label{cl:r6--2}
%We have that either $X=\emptyset$ or $Y=\emptyset$.
%\end{claim}

\begingroup
\renewcommand{\theclaim}{\ref{thm:r6}.2}
\begin{claim}\label{cl:r6--2}
We have that either $X=\emptyset$ or $Y=\emptyset$.
\end{claim}
\endgroup

\begin{proof}
We first show that $E(X, Y)=\emptyset$.
Suppose that there exist two vertices $x\in X$ and $y\in Y$ with $xy\in E(G)$.
Then $y\in N_G(x)$ and $x\in N_G(y)$.
On the one hand,
since $d_G(x)=8$, the degree sequence of $G[N_G(x)]$ is $(2,2,2,2,3,3,3,3)$ by Claim~\ref{cl:r6--1}.
Since $y\in N_G(x)$, we further have $|N_G(y)\cap N_G(x)|\in \{2,3\}$.
On the other hand,
since $d_G(y)=9$, we have $G[N_G(y)]\cong K_{3}\Box K_{3}$ by Claim~\ref{cl:r6--1}.
Since $x\in N_G(y)$ and $K_{3}\Box K_{3}$ is 4-regular, we further have $|N_G(x)\cap N_G(y)|=4$.
This contradiction implies that $E(X, Y)=\emptyset$.

Moreover, note that $d_G(v)=8$ for every vertex $v\in X$ and $d_G(v)=9$ for every vertex $v\in Y$.
Since $E(X, Y)=\emptyset$, if $X\neq \emptyset$ and $Y\neq \emptyset$, then
every vertex of $X$ has all its eight neighbors in $X$, so $|X|\geq 9$; every vertex of $Y$ has all its nine neighbors in $Y$, so $|Y|\geq 10$.
Since $X\cap Y=\emptyset$, we further have $|V(G)|=|X|+|Y|\geq 19>18$, a contradiction.
\end{proof}

By Claim~\ref{cl:r6--2}, $G$ is either 9-regular or 8-regular.
If $G$ is 9-regular, then $G[N_G(v)]\cong K_{3}\Box K_{3}$ for every $v\in V(G)$ by Claim~\ref{cl:r6--1}, so $G$ is a locally $3\times 3$ grid graph.
If $G$ is 8-regular, then $\overline{G}$ is 9-regular.
By Fact~\ref{fa:threshold-1}~(i) and (ii), $\overline{G}$ is induced $\mathcal{T}_6$-free and $\overline{G}[N_{\overline{G}}(v)]$ is induced $\mathcal{T}_5$-free.
Then, by Lemma~\ref{le:r6-9-vtx}, $\overline{G}[N_{\overline{G}}(v)]\cong K_{3}\Box K_{3}$ for every $v\in V(G)$, so $\overline{G}$ is a locally $3\times 3$ grid graph.
Moreover, an 18-vertex 9-regular graph must be connected since every component of a 9-regular graph has at least ten vertices.
Hence, either $G$ or $\overline{G}$ is a connected locally $3\times 3$ grid graph.
By Lemma~\ref{le:grid}~(ii), there are precisely two connected locally $3\times 3$ grid graphs, one on $16$ vertices and the other one on $20$ vertices.
However, both $G$ and $\overline{G}$ have 18 vertices.
This contradiction completes the proof of Theorem~\ref{thm:r6}.
\hfill$\blacksquare$
%\vspace{0.2cm}

\subsection{Exact values of $CR(s,3)$}
\label{subsec:pf_complete_CR}

We now present our proof of Theorem~\ref{thm:CRs3}, which follows from Lemma~\ref{le:CRs3lower} (lower bound) and Lemma~\ref{le:CRs3upper} (upper bound) below.
For $s\geq 3$, let
$$n_s\colonequals
\left\{
   \begin{aligned}
    &2\cdot 5^{\frac{s-3}{2}} & & \mbox{if $s$ is odd},\\
    &5^{\frac{s-2}{2}} & & \mbox{if $s$ is even}.
   \end{aligned}
   \right.$$
For small values of $s$, the exact values of $CR(s,3)$ were obtained by Brosch et al.~\cite{BLMP}.
They showed that the values of $CR(s,3)$ are $3, 6, 11, 26$ for $s=3, 4, 5, 6$, respectively.
Thus $CR(s,3)= n_s+1$ for $3\leq s\leq 6$.
Hence, it remains to consider the case $s\geq 7$.

\begin{lemma}\label{le:CRs3lower}
For any integer $s\geq 3$, we have $CR(s,3)\geq n_s+1$.
\end{lemma}

\noindent {\bf Proof.}
%\begin{proof}
We may assume that $s\geq 7$ by the above-mentioned results for $3\leq s\leq 6$.
To give a lower bound construction, we use lexicographic product colorings.

Let $\chi_2$ be a 2-edge-coloring of $K_5$ using colors 1 and 2 in which each color induces a $C_5$.
Note that there is no rainbow $K_3$ or monochromatic $K_3$ in this edge-coloring.
Let $\chi'_4$ be a 2-edge-coloring of $K_5$ using colors 3 and 4 in which each color induces a $C_5$,
and let $\chi_4$ be the lexicographic product coloring $\chi'_4 \otimes \chi_2$ of $K_{25}$.
Assume that we have constructed $\chi_2, \chi_4, \ldots, \chi_i$ for some even number $i\leq s-4$, and we shall construct $\chi_{i+2}$ in an analogous way:
let $\chi'_{i+2}$ be a 2-edge-coloring of $K_5$ using colors $i+1$ and $i+2$ in which each color induces a $C_5$,
and let $\chi_{i+2}$ be the lexicographic product coloring $\chi'_{i+2} \otimes \chi_i$ of $K_{5^{(i+2)/2}}$.
If $s$ is even, then let $G$ be the $(s-2)$-edge-colored $K_{5^{(s-2)/2}}$ under the coloring $\chi_{s-2}$.
If $s$ is odd, then let $\chi'$ be an edge-coloring of $K_2$ using color $s-2$, and let $G$ be the $(s-2)$-edge-colored $K_{2\cdot 5^{(s-3)/2}}$ under the coloring $\chi' \otimes \chi_{s-3}$.

From the construction process, we know that $G$ contains neither a rainbow $K_3$ nor a monochromatic $K_3$.
Indeed, if $s$ is even, then a triangle either lies in one of the five copies of the $(s-4)$-edge-colored $K_{5^{(s-4)/2}}$ under the coloring $\chi_{s-4}$, or meets at least two of them.
In the former case, the assertion follows inductively.
In the latter case, if a triangle meets three copies of the $K_{5^{(s-4)/2}}$, then it cannot be rainbow or monochromatic since the cross-edges form a blow-up of a 2-edge-colored $K_5$ without monochromatic $K_3$; if a triangle meets exactly two copies of the $K_{5^{(s-4)/2}}$, then it cannot be rainbow since it contains two edges of the same color, and it cannot be monochromatic since the colors of the cross-edges and those of the edges within one copy of the $K_{5^{(s-4)/2}}$ are distinct.
The case when $s$ is odd is analogous.

We next show that $G$ contains no orderable $K_s$.
Suppose that $G$ contains an orderable $K_s$ under the ordering $(v_1, v_2, \ldots, v_s)$ and the corresponding colors $(c_1, c_2, \ldots, c_{s-1})$.
If $c_i=c_j$ for some $i< j$, then $G[\{v_i, v_j, v_{j+1}\}]$ is a monochromatic $K_3$.
But $G$ contains no monochromatic $K_3$ by the preceding argument, so the $s-1$ colors are pairwise distinct.
This contradicts the fact that $G$ uses only $s-2$ colors.
Hence, $G$ contains no orderable $K_s$.
This implies that $CR(s,3)\geq |V(G)|+1=n_s+1$, and the proof is complete.
%\end{proof}
\hfill$\blacksquare$
%\vspace{0.2cm}

\begin{lemma}\label{le:CRs3upper}
For any integer $s\geq 3$, we have $CR(s,3)\leq n_s+1$.
\end{lemma}

\noindent {\bf Proof.}
%\begin{proof}
It suffices to show that for every edge-colored complete graph $G$ with neither an orderable $K_s$ nor a rainbow $K_3$, we have $|V(G)|\leq n_s$.
We prove this by induction on $s$.
The base case $3\leq s\leq 6$ follows from the above-mentioned results of Brosch et al.~\cite{BLMP}.
Assume that the statement holds for all values from $3$ to $s-1$, and we shall prove it for $s\geq 7$.
Note that
\begin{equation}\label{eq:ns-1}
n_{s-1}\leq \frac{1}{2}n_s, ~~~~~~ n_{s-2}= \frac{1}{5}n_s, ~~~~~~ n_{s-3}\leq \frac{1}{10}n_s, ~~~~~~ n_{s-4}= \frac{1}{25}n_s
\end{equation}
for any $s\geq 7$.

Let $G$ be an edge-colored complete graph with neither an orderable $K_s$ nor a rainbow $K_3$.
For a contradiction, suppose that $|V(G)|>n_s$.
By Theorem~\ref{thm:Gallai}, $V(G)$ admits a Gallai partition $V_1, V_2, \ldots, V_m$ with $m\geq 2$.
Without loss of generality, we may assume that all edges between distinct parts have colors in $\{1,2\}$.
If there exists a part $V_i$ such that all edges between $V_i$ and $V(G)\setminus V_i$ are of the same color, then to avoid an orderable $K_s$, neither $G[V_i]$ nor $G[V(G)\setminus V_i]$ contains an orderable $K_{s-1}$.
By the induction hypothesis and Inequality~(\ref{eq:ns-1}), we have $|V(G)|=|V_i|+|V(G)\setminus V_i|\leq n_{s-1}+n_{s-1}\leq n_s$, a contradiction.
Hence, for every part $V_i$, there exist two parts $V_{j_1}$ and $V_{j_2}$ with $\chi(V_i, V_{j_1})=1$ and $\chi(V_i, V_{j_2})=2$.
This implies that
\begin{equation}\label{eq:ns-2}
|V_i|\leq n_{s-2}
\end{equation}
for every $i\in [m]$.
To see this, we may assume that $\chi(V_{j_1}, V_{j_2})=1$ without loss of generality.
Now for any vertices $v_1\in V_{j_1}$ and $v_2\in V_{j_2}$, we have $\chi(v_1, V_i\cup \{v_2\})=1$ and $\chi(v_2, V_i)=2$.
If $G[V_i]$ contains an orderable $K_{s-2}$, say with ordering $(u_1, u_2, \ldots, u_{s-2})$, then $G[\{v_1, v_2, u_1, u_2, \ldots, u_{s-2}\}]$ is an orderable $K_s$ under the ordering $(v_1, v_2, u_1, u_2, \ldots, u_{s-2})$, a contradiction.
Thus $G[V_i]$ contains no orderable $K_{s-2}$, and thus $|V_i|\leq n_{s-2}$ by the induction hypothesis.

If $m\leq 5$, then $|V(G)|\leq 5n_{s-2}=n_s$ by Inequalities~(\ref{eq:ns-1}) and (\ref{eq:ns-2}), a contradiction.
Hence, we have $m\geq 6$.
Without loss of generality, we may assume that $|V_m|=\min_{i\in [m]}|V_i|$.
Let
\begin{align*}
  A &\colonequals \{i\in [m]\setminus \{m\}\colon\, \chi(V_m, V_i)=1\},\hspace{-1.5cm} & &a\colonequals |A|,\hspace{-1.5cm} & &V_A\colonequals \bigcup\nolimits_{i\in A}V_i, \\
  B &\colonequals \{i\in [m]\setminus \{m\}\colon\, \chi(V_m, V_i)=2\},\hspace{-1.5cm} & &b\colonequals |B|,\hspace{-1.5cm} & &V_B\colonequals \bigcup\nolimits_{i\in B}V_i.
\end{align*}
Then $a\geq 1$, $b\geq 1$, $a+b=m-1$ and $V_A\cup V_B=V(G)\setminus V_m=\bigcup_{i\in [m-1]}V_i$.
In order to avoid an orderable $K_s$ with the first vertex in $V_m$, there is no orderable $K_{s-1}$ in $G[V_A]$ or $G[V_B]$.
Thus
\begin{equation}\label{eq:ns-3}
|V_A|\leq n_{s-1} \mbox{~~~~~and~~~~~} |V_B|\leq n_{s-1}
\end{equation}
by the induction hypothesis.
Then we further have that $a\geq 2$ and $b\geq 2$; otherwise if $a\leq 1$ or $b\leq 1$, then $|V(G)|=|V_m|+|V_A|+|V_B|\leq 2n_{s-2}+n_{s-1}< n_s$ by Inequalities~(\ref{eq:ns-1}), (\ref{eq:ns-2}) and (\ref{eq:ns-3}), a contradiction.

Without loss of generality, we may assume that $m-1\in A$, $|V_{m-1}|=\min_{i\in A}|V_i|$, $m-2\in B$ and $|V_{m-2}|=\min_{i\in B}|V_i|$.
Let
\begin{align*}
  X &\colonequals \{i\in A\setminus \{m-1\}\colon\, \chi(V_{m-1}, V_i)=1\},\hspace{-1.5cm} & &x\colonequals |X|,\hspace{-1.5cm} & &V_X\colonequals \bigcup\nolimits_{i\in X}V_i, \\
  Y &\colonequals \{i\in A\setminus \{m-1\}\colon\, \chi(V_{m-1}, V_i)=2\},\hspace{-1.5cm} & &y\colonequals |Y|,\hspace{-1.5cm} & &V_Y\colonequals \bigcup\nolimits_{i\in Y}V_i, \\
  U &\colonequals \{i\in B\setminus \{m-2\}\colon\, \chi(V_{m-2}, V_i)=1\},\hspace{-1.5cm} & &u\colonequals |U|,\hspace{-1.5cm} & &V_U\colonequals \bigcup\nolimits_{i\in U}V_i, \\
  W &\colonequals \{i\in B\setminus \{m-2\}\colon\, \chi(V_{m-2}, V_i)=2\},\hspace{-1.5cm} & &w\colonequals |W|,\hspace{-1.5cm} & &V_W\colonequals \bigcup\nolimits_{i\in W}V_i.
\end{align*}
Then $a=x+y+1$, $b=u+w+1$ and $V(G)=V_m\cup V_{m-1}\cup V_{m-2}\cup V_X\cup V_Y\cup V_U\cup V_W.$
To avoid an orderable $K_s$, none of $G[V_X]$, $G[V_Y]$, $G[V_U]$, $G[V_W]$ contains an orderable $K_{s-2}$.
Thus, by the induction hypothesis, we have
\begin{equation}\label{eq:ns-4}
|V_X|\leq n_{s-2}, ~~~~~~ |V_Y|\leq n_{s-2}, ~~~~~~ |V_U|\leq n_{s-2}, ~~~~~~ |V_W|\leq n_{s-2}.
\end{equation}

\begin{claim}\label{cl:CRs3upper-1}
$x, y, u, w\leq 3.$
\end{claim}

\begin{proof}
By symmetry, it suffices to prove that $x\leq 3.$
For a contradiction, suppose that $x\geq 4$.
Then since $|V_m|=\min_{i\in [m]}|V_i|$ and $|V_{m-1}|=\min_{i\in A}|V_i|$, we have $|V_m|+|V_{m-1}|\leq \frac{1}{2}|V_X|$.
Combining with Inequalities~(\ref{eq:ns-1}), (\ref{eq:ns-3}) and (\ref{eq:ns-4}), we have
\begin{align*}
  |V(G)| = &~|V_m|+|V_{m-1}|+|V_X|+|V_Y|+|V_B| \leq \frac{1}{2}|V_X|+|V_X|+|V_Y|+|V_B|\\
  \leq &~\frac{5}{2}n_{s-2}+n_{s-1} \leq \frac{5}{2}\cdot \frac{1}{5}n_{s}+\frac{1}{2}n_{s} = n_s,
\end{align*}
a contradiction.
\end{proof}

\begin{claim}\label{cl:CRs3upper-2}
$x, y, u, w\leq 2$, so $a,b\leq 5$ and $m\leq 11$.
\end{claim}

\begin{proof}
By symmetry, it suffices to prove that $x\leq 2.$
For a contradiction, suppose that $x\geq 3$, so $x=3$ by Claim~\ref{cl:CRs3upper-1}.
Note that in every 2-edge-coloring of a triangle, some vertex is incident with two edges of the same color.
Without loss of generality, we may assume that $X=\{1,2,3\}$ and $\chi(V_1, V_3)=\chi(V_2, V_3)$.
Then there is no orderable $K_{s-3}$ in $G[V_1\cup V_2]$;
otherwise we can choose vertices $v_1\in V_m$, $v_2\in V_{m-1}$ and $v_3\in V_3$ such that $\{v_1, v_2, v_3\}$ together with an orderable $K_{s-3}$ in $G[V_1\cup V_2]$ forms an orderable $K_{s}$, a contradiction.
Similarly, there is no orderable $K_{s-3}$ in $G[V_3]$;
otherwise we can choose vertices $v_1\in V_m$, $v_2\in V_{m-1}$ and $v_3\in V_1\cup V_2$ such that $\{v_1, v_2, v_3\}$ together with an orderable $K_{s-3}$ in $G[V_3]$ forms an orderable $K_{s}$, a contradiction.
By the induction hypothesis, we now have $|V_1\cup V_2|\leq n_{s-3}$ and $|V_3|\leq n_{s-3}$.
Moreover, since $|V_m|=\min_{i\in [m]}|V_i|$ and $|V_{m-1}|=\min_{i\in A}|V_i|$, we also have $|V_m|+|V_{m-1}|\leq |V_1|+|V_2|\leq n_{s-3}$.
Combining with Inequalities~(\ref{eq:ns-1}), (\ref{eq:ns-3}) and (\ref{eq:ns-4}), we have
\begin{align*}
  |V(G)| = &~|V_m|+|V_{m-1}|+|V_1|+|V_2|+|V_3|+|V_Y|+|V_B| \\
  \leq &~3n_{s-3}+n_{s-2}+n_{s-1}\leq \frac{3}{10}n_{s}+\frac{1}{5}n_{s}+\frac{1}{2}n_{s} = n_s,
\end{align*}
a contradiction.
Hence, $x, y, u, w\leq 2$, and thus $a=x+y+1\leq 5$, $b=u+w+1\leq 5$ and $m=a+b+1\leq 11$.
\end{proof}

\begin{claim}\label{cl:CRs3upper-3}
$a,b\geq 3.$
\end{claim}

\begin{proof}
Suppose for a contradiction that $a\leq 2$ or $b\leq 2$, say $a\leq 2$.
Recall that $a\geq 2$ and $a+b+1=m\geq 6$.
Thus $a=2$ and $b\geq 3$, say $A=\{m-1, 1\}$.

For any $i\in A$, if there exist two distinct $j_1, j_2\in B$ with $\chi(V_i, V_{j_1})=\chi(V_i, V_{j_2})=1$, then regardless of the color $\chi(V_{j_1}, V_{j_2})$, there is no orderable $K_{s-3}$ in $G[V_{j_1}]$.
Indeed, if there is an orderable $K_{s-3}$ in $G[V_{j_1}]$, then we can choose three vertices $v_1\in V_i$, $v_2\in V_{m}$ and $v_3\in V_{j_2}$ that together with this $K_{s-3}$ form an orderable $K_{s}$, a contradiction.
By the induction hypothesis, we have $|V_{j_1}|\leq n_{s-3}$.
Moreover, since $|V_m|=\min_{i\in [m]}|V_i|$, we also have $|V_m|\leq |V_{j_1}|\leq n_{s-3}$.
Combining with Inequalities~(\ref{eq:ns-1}), (\ref{eq:ns-2}) and (\ref{eq:ns-3}), we have
$$|V(G)|=|V_m|+|V_{m-1}|+|V_1|+|V_B|\leq n_{s-3}+n_{s-2}+n_{s-2}+n_{s-1}\leq \frac{1}{10}n_{s}+\frac{1}{5}n_{s}+\frac{1}{5}n_{s}+\frac{1}{2}n_{s}= n_s,$$ a contradiction.
Therefore, for any $i\in A$, there exists at most one $j\in B$ with $\chi(V_i, V_{j})=1$.
Since $b\geq 3$, there exists a $j\in B$ such that $\chi(V_j, V_{m-1})=\chi(V_j, V_1)=2$.
By an analogous argument as above, there is no orderable $K_{s-3}$ in $G[V_{m-1}]$ or $G[V_1]$, so $|V_{m-1}|\leq n_{s-3}$, $|V_1|\leq n_{s-3}$ and $|V_m|\leq n_{s-3}$.
Combining with Inequalities~(\ref{eq:ns-1}) and (\ref{eq:ns-3}), we have
$|V(G)|=|V_m|+|V_{m-1}|+|V_1|+|V_B|\leq 3n_{s-3}+n_{s-1}\leq \frac{3}{10}n_{s}+\frac{1}{2}n_{s}< n_s$, a contradiction.
\end{proof}

\begin{claim}\label{cl:CRs3upper-4}
$a=b=5$ and $m=11$.
\end{claim}

\begin{proof}
It suffices to prove $a=b=5$, since then $m=11$ follows from $m=a+b+1$.
Suppose for a contradiction that $a\neq 5$ or $b\neq 5$, say $a\neq 5$.
Now by Claims~\ref{cl:CRs3upper-2} and \ref{cl:CRs3upper-3}, we have $3\leq a\leq 4$.

If there exists an $i\in A$ such that all edges between $V_i$ and $V_A\setminus V_i$ are of the same color, then to avoid an orderable $K_{s}$, neither $G[V_i]$ nor $G[V_A\setminus V_i]$ contains an orderable $K_{s-2}$.
By the induction hypothesis, we have $|V_i|\leq n_{s-2}$ and $|V_A\setminus V_i|\leq n_{s-2}$.
Moreover, since $a\geq 3$, we have $|A\setminus \{i\}|\geq 2$.
Combining with $|V_m|=\min_{i\in [m]}|V_i|$, we have $|V_m|\leq \frac{1}{2}|V_A\setminus V_i|\leq \frac{1}{2}n_{s-2}$.
By Inequalities~(\ref{eq:ns-1}) and (\ref{eq:ns-3}), we have
$|V(G)|=|V_m|+|V_i|+|V_A\setminus V_i|+|V_B|\leq \frac{1}{2}n_{s-2}+n_{s-2}+n_{s-2}+n_{s-1}\leq \frac{5}{2}\cdot \frac{1}{5}n_{s}+\frac{1}{2}n_{s}=n_s$, a contradiction.

Therefore, for any $i\in A$, there exist $j_1, j_2\in A\setminus \{i\}$ such that $\chi(V_i, V_{j_1})=1$ and $\chi(V_i, V_{j_2})=2$.
This implies that $G[V_i]$ contains no orderable $K_{s-3}$ for every $i\in A$.
Indeed, suppose for a contradiction that $G[V_i]$ contains an orderable $K_{s-3}$ for some $i\in A$.
If $\chi(V_{j_1}, V_{j_2})=1$, then we can choose three vertices $v_1\in V_m$, $v_2\in V_{j_1}$ and $v_3\in V_{j_2}$ that together with this $K_{s-3}$ form an orderable $K_{s}$;
if $\chi(V_{j_1}, V_{j_2})=2$, then we can choose three vertices $v_1\in V_m$, $v_2\in V_{j_2}$ and $v_3\in V_{j_1}$ that together with this $K_{s-3}$ form an orderable $K_{s}$.
In both cases, we deduce a contradiction.
Now by the induction hypothesis, we have $|V_i|\leq n_{s-3}$ for every $i\in A$.
Moreover, since $|V_m|=\min_{i\in [m]}|V_i|$, we also have $|V_m|\leq n_{s-3}$.
By Inequalities~(\ref{eq:ns-1}), (\ref{eq:ns-3}) and since $a\leq 4$, we have
$|V(G)|=|V_m|+|V_A|+|V_B|\leq n_{s-3}+4n_{s-3}+n_{s-1}\leq 5\cdot\frac{1}{10}n_{s}+\frac{1}{2}n_{s}=n_s$, a contradiction.
\end{proof}

Combining with Claims~\ref{cl:CRs3upper-2} and \ref{cl:CRs3upper-4}, we have $a=b=5$, $x=y=u=w=2$ and $m=11$.

\begin{claim}\label{cl:CRs3upper-5}
For every $i\in [m]$, we have $|V_i|\leq n_{s-3}.$
\end{claim}

\begin{proof}
Let $X=\{i_1, i_2\}$.
In order to avoid an orderable $K_s$, neither $G[V_{i_1}]$ nor $G[V_{i_2}]$ contains an orderable $K_{s-3}$;
otherwise if there is an orderable $K_{s-3}$ in $G[V_{i_1}]$ or $G[V_{i_2}]$, say $G[V_{i_1}]$, then this $K_{s-3}$ together with vertices $v_1\in V_m$, $v_2\in V_{m-1}$ and $v_3\in V_{i_2}$ forms an orderable $K_s$, a contradiction.
Thus by the induction hypothesis, we have $|V_{i_1}|\leq n_{s-3}$ and $|V_{i_2}|\leq n_{s-3}$.
By symmetry, we also have $|V_i|\leq n_{s-3}$ for any $i\in Y\cup U\cup W$.
Moreover, since $|V_m|=\min_{i\in [m]}|V_i|$, $|V_{m-1}|=\min_{i\in A}|V_i|$ and $|V_{m-2}|=\min_{i\in B}|V_i|$, we further have $|V_i|\leq n_{s-3}$ for $i\in \{m-2, m-1, m\}$.
\end{proof}

We now complete the proof of Lemma~\ref{le:CRs3upper}.
Suppose that for some $i\in A$, there exist $j_1, j_2, j_3\in B$ such that $\chi(V_i, V_{j_1})=\chi(V_i, V_{j_2})=\chi(V_i, V_{j_3})=1$.
Note that at least two of $\chi(V_{j_1}, V_{j_2})$, $\chi(V_{j_1}, V_{j_3})$ and $\chi(V_{j_2}, V_{j_3})$ are of the same color, say $\chi(V_{j_1}, V_{j_2})=\chi(V_{j_1}, V_{j_3}).$
In order to avoid an orderable $K_s$, there is no orderable $K_{s-4}$ in $G[V_{j_2}]$ or $G[V_{j_3}]$;
otherwise if there is an orderable $K_{s-4}$ in $G[V_{j_2}]$ or $G[V_{j_3}]$, say $G[V_{j_2}]$, then this $K_{s-4}$ together with vertices $v_1\in V_i$, $v_2\in V_{m}$, $v_3\in V_{j_1}$ and $v_4\in V_{j_3}$ forms an orderable $K_s$, a contradiction.
Thus by the induction hypothesis, we have $|V_{j_2}|\leq n_{s-4}$ and $|V_{j_3}|\leq n_{s-4}$.
Since $|V_m|=\min_{i\in [m]}|V_i|$, we also have $|V_m|\leq n_{s-4}$.
Then, by Inequality~(\ref{eq:ns-1}) and Claim~\ref{cl:CRs3upper-5}, we have
\begin{align*}
  |V(G)| = &~|V_m|+|V_{j_2}|+|V_{j_3}|+\sum\nolimits_{j\in [m]\setminus\{m, j_2, j_3\}}|V_j| \\
  \leq &~n_{s-4}+n_{s-4}+n_{s-4}+8n_{s-3}\leq \frac{3}{25}n_{s}+\frac{8}{10}n_{s}< n_s,
\end{align*}
a contradiction.
Hence, for every $i\in A$, there exist at most two indices $j\in B$ such that $\chi(V_i, V_{j})=1$.
By symmetry, for every $j\in B$, there exist at most two indices $i\in A$ such that $\chi(V_j, V_{i})=2$.
Now we define an auxiliary 2-edge-colored complete bipartite graph $H$ as follows: $V(H)=A\cup B$, and color edge $ij$ with $\chi(V_i, V_j)$ for every $i\in A$ and $j\in B$.
Then the preceding argument implies that $H$ has at most $10$ edges of color 1 and $10$ edges of color 2, so $25=|A|\cdot |B|=|E(H)|\leq 10+10=20$, a contradiction.
This completes the proof of Lemma~\ref{le:CRs3upper}.
%\end{proof}
\hfill$\blacksquare$
%\vspace{0.2cm}

\section{Results on orderable complete bipartite graphs}
\label{sec:pf_complete_bipar}

In this section, we present our proofs of Theorems~\ref{thm:CRKst}, \ref{thm:CRKstUpper}, \ref{thm:CRK2t3t} and \ref{thm:CRK2t+}.
Since $r'_2(K_{s,t})\leq CR(K_{s,t}, K_3)$, it suffices to provide upper bounds on $CR(K_{s,t}, K_3)$ and lower bounds on $r'_2(K_{s,t})$.
Section~\ref{subsec:pf_complete_bipar_upper} provides upper bounds on $CR(K_{s,t}, K_3)$ (and thus proves Theorem~\ref{thm:CRKstUpper}), Section~\ref{subsec:pf_complete_bipar_lower} gives several lower bounds on $r'_2(K_{s,t})$, and
Section~\ref{subsec:pf_complete_bipar_equal} proves that $r'_2(K_{s,t})=CR(K_{s,t}, K_3)$ for sufficiently large $t$.
Then Theorem~\ref{thm:CRKst} follows from Theorem~\ref{thm:CRKstUpper}, Lemma~\ref{le:CRKstlower} and Theorem~\ref{thm:CRKst-equal} below.
The proofs of Theorems~\ref{thm:CRK2t3t} and \ref{thm:CRK2t+} are given in Section~\ref{subsec:pf_complete_bipar_lower}.

We first collect several auxiliary results that will be used in this section.
To study orderable $K_{s,t}$, we use properties of lonesum matrices.
A matrix $A$ with entries in $\{0, 1\}$ is called a {\it lonesum matrix} if $A$ is uniquely determined by its row and column sum vectors.
Ryser~\cite{Ryser957CJM} obtained several equivalent conditions for a matrix to be lonesum.
For a 2-edge-coloring $\chi$ of $K_{s,t}$ with bipartition $\{x_1, \ldots, x_s\}\cup \{y_1, \ldots, y_t\}$ using red and blue, we define an $s\times t$ matrix $M(\chi)=(m_{ij})$ with
$$m_{ij}\colonequals
\left\{
   \begin{aligned}
    &1 & & \mbox{if $\chi(x_iy_j)$ is red},\\
    &0 & & \mbox{if $\chi(x_iy_j)$ is blue}.
   \end{aligned}
   \right.$$
For $0\leq i\leq s$, let $1^{i}0^{s-i}$ be the vector with $1$ as the first $i$ entries and $0$ as the last $s-i$ entries.
Let $$\mathcal{C}_s\colonequals \left\{1^{i}0^{s-i}\colon\, 0\leq i\leq s\right\}=\left\{1^{0}0^{s}, 1^{1}0^{s-1}, \ldots, 1^{s}0^{0}\right\}.$$
Then we have the following equivalent conditions for a 2-edge-coloring of $K_{s,t}$ to be orderable.

\begin{lemma}\label{le:lonesum}
For any 2-edge-coloring $\chi$ of $K_{s,t}$, the following statements are equivalent.
\begin{itemize}
\item[{\rm (i)}] The 2-edge-coloring $\chi$ of $K_{s,t}$ is orderable.
\item[{\rm (ii)}] $M(\chi)$ is a lonesum matrix.
\item[{\rm (iii)}] $M(\chi)$ does not contain
$\left(
\begin{array}{cc}
1 & 0 \\
0 & 1 \\
\end{array}
\right)$
or
$\left(
\begin{array}{cc}
0 & 1 \\
1 & 0 \\
\end{array}
\right)$
as a $2\times 2$ submatrix.
\item[{\rm (iv)}] After permuting the rows, every column of $M(\chi)$ belongs to $\mathcal{C}_s$.
\end{itemize}
\end{lemma}

\noindent {\bf Proof.}
The equivalence of (ii), (iii), (iv) was proved by Ryser~\cite{Ryser957CJM}; see also \cite{Kam}.
In order to prove the lemma, it suffices to show that (i) $\Rightarrow$ (iii) and (iv) $\Rightarrow$ (i).

(i) $\Rightarrow$ (iii):
Assume that the 2-edge-coloring $\chi$ of $K_{s,t}$ is orderable.
Note that every subgraph of this $K_{s,t}$ is orderable under the inherited ordering.
If $M(\chi)$ contains
$\left(
\begin{array}{cc}
1 & 0 \\
0 & 1 \\
\end{array}
\right)$
or
$\left(
\begin{array}{cc}
0 & 1 \\
1 & 0 \\
\end{array}
\right)$
as a submatrix, then $K_{s,t}$ contains an alternating $C_4$ (that is, a $C_4$ whose edges are alternately red and blue).
However, the alternating $C_4$ is not an orderable $C_4$.
Indeed, in any ordering of an alternating $C_4$, the first vertex is incident with two edges of distinct colors.
Hence, (iii) holds.

(iv) $\Rightarrow$ (i):
We prove the result by induction on $s+t$.
If $\min\{s,t\}=1$, then $K_{s,t}$ is a star and is orderable by Proposition~\ref{prop:forest}, so the assertion holds.
Now assume that the result holds for $s+t-1$, and we shall prove it for $s+t$ with $\min\{s,t\}\geq 2$.
Assume (iv) holds, that is, after permuting the rows, we assume that every column of $M(\chi)$ belongs to $\mathcal{C}_s$.
If some column of $M(\chi)$ is $1^{0}0^{s}$, then the corresponding vertex $y$ is incident with only blue edges.
We delete this column from $M(\chi)$ and delete $y$ from $K_{s,t}$.
By the induction hypothesis, the resulting graph $K_{s,t-1}$ is orderable.
Then the original graph $K_{s,t}$ is also orderable by putting $y$ as the first vertex.
If every column of $M(\chi)$ has the form $1^{j}0^{s-j}$ with $j\neq 0$, then the first row of $M(\chi)$ is the all-ones vector.
Thus the vertex $x$ corresponding to the first row is incident with only red edges.
We delete the first row from $M(\chi)$ and delete $x$ from $K_{s,t}$.
By the induction hypothesis, the resulting graph $K_{s-1,t}$ is orderable.
Then the original graph $K_{s,t}$ is also orderable by putting $x$ as the first vertex.
Hence, (i) holds.
\hfill$\blacksquare$
\vspace{0.2cm}

Let $\chi$ be a 2-edge-coloring of $K_{n}$ using red and blue.
For an ordered $s$-tuple $\mathbf{x}=(x_1, \ldots, x_s)$ of $s$ distinct vertices of $K_n$, define
$$Y_{\chi}(\mathbf{x})\colonequals \left\{y\in V(K_n)\setminus \{x_1, \ldots, x_s\}\colon\, \left(\mathbf{1}_{\scriptsize\{\mbox{$\chi(x_1y)$ is red}\}}, \ldots, \mathbf{1}_{\scriptsize\{\mbox{$\chi(x_sy)$ is red}\}}\right)\in \mathcal{C}_s\right\}.$$
The vertices in $Y_{\chi}(\mathbf{x})$ are called {\it compatible} with $\mathbf{x}$.
In other words, a vertex $y\in V(K_n)\setminus \{x_1, \ldots, x_s\}$ is compatible with $\mathbf{x}$ if there exists an integer $j$ with $0\leq j\leq s$ such that $x_iy$ is red for all $1\leq i\leq j$, and $x_iy$ is blue for all $j+1\leq i\leq s$.

\begin{lemma}\label{le:Y}
A 2-edge-coloring $\chi$ of $K_{n}$ contains an orderable $K_{s,t}$ if and only if there exists an ordered $s$-tuple $\mathbf{x}=(x_1, \ldots, x_s)$ of $s$ distinct vertices such that $\left|Y_{\chi}(\mathbf{x})\right|\geq t.$
In particular, we can choose the orderable $K_{s,t}$ with partite sets $X, Y$ such that $X=\{x_1, \ldots, x_s\}$ and $Y\subseteq Y_{\chi}(\mathbf{x}).$
\end{lemma}

\noindent {\bf Proof.}
Assume that there exists an $\mathbf{x}=(x_1, \ldots, x_s)$ with $\left|Y_{\chi}(\mathbf{x})\right|\geq t.$
Choose distinct vertices $y_1, \ldots, y_t \in Y_{\chi}(\mathbf{x})$.
Let $K$ be the copy of $K_{s,t}$ with partite sets $\{x_1, \ldots, x_s\}$ and $\{y_1, \ldots, y_t\}$, and let $\chi'$ be the edge-coloring of $K$ inherited from $\chi$.
Then every column of $M(\chi')$ belongs to $\mathcal{C}_s$.
By the implication (iv)~$\Rightarrow$~(i) in Lemma~\ref{le:lonesum}, the edge-coloring $\chi'$ of $K$ is orderable, so $K_{n}$ contains an orderable $K_{s,t}$.

Assume that the 2-edge-coloring $\chi$ of $K_{n}$ contains an orderable $K_{s,t}$.
Let $\{x_1, \ldots, x_s\}$ and $\{y_1, \ldots, y_t\}$ be the partite sets of such a $K_{s,t}$.
By the implication (i)~$\Rightarrow$~(iv) in Lemma~\ref{le:lonesum}, there exists a permutation $\pi$ of the vertices $x_1, \ldots, x_s$
such that the vector $\big(\mathbf{1}_{\scriptsize\{\mbox{$\chi(\pi(x_1)y_j)$ is red}\}}, \ldots, \mathbf{1}_{\scriptsize\{\mbox{$\chi(\pi(x_s)y_j)$ is red}\}}\big)$ belongs to $\mathcal{C}_s$ for each $j\in [t]$.
Thus for $\mathbf{x}=(\pi(x_1), \ldots, \pi(x_s))$, we have $y_1, \ldots, y_t\in Y_{\chi}(\mathbf{x})$, so $\left|Y_{\chi}(\mathbf{x})\right|\geq t.$
\hfill$\blacksquare$
\vspace{0.2cm}

A set $X$ of vertices in an edge-colored complete graph $K_n$ is called {\it color-homogeneous} if for each vertex $v\in V(K_n)\setminus X$, all edges between $v$ and $X$ are of the same color (the color may depend on $v$).
For example, every part in a Gallai partition is a color-homogeneous set.
The reader should not confuse this with the homogeneous 3-set from Section~\ref{subsec:pf_complete_small}.
We will use the following fact.

\begin{fact}\label{fa:homo}
Let $G$ be an edge-colored complete graph $K_n$, and let $X$ be a color-homogeneous set.
\begin{itemize}
\item[{\rm (i)}] If $|X|\geq s$ and $|V(G)\setminus X|\geq t$, or $|X|\geq t$ and $|V(G)\setminus X|\geq s$, then $G$ contains an orderable $K_{s,t}$.
\item[{\rm (ii)}] If $G[X]$ contains an orderable $K_{a,b}$, then $G$ contains an orderable $K_{a+|V(G)\setminus X|,b}$ and an orderable $K_{a,b+|V(G)\setminus X|}$.
\end{itemize}
\end{fact}

\noindent {\bf Proof.}
(i) If $|X|\geq s$ and $|V(G)\setminus X|\geq t$, then let $X'=\{u_1, \ldots, u_{s}\}\subseteq X$ and $Y=\{v_1, \ldots, v_{t}\}\subseteq V(G)\setminus X$.
Note that the edges between $X'$ and $Y$ induce a $K_{s,t}$.
Since $X$ is a color-homogeneous set, for each $v_i\in Y$, all edges between $v_i$ and $X$ are of the same color.
Thus this $K_{s,t}$ is orderable under the ordering $(v_1, \ldots, v_{t}, u_1, \ldots, u_{s})$.
The proof of the case $|X|\geq t$ and $|V(G)\setminus X|\geq s$ is similar.

(ii) Assume that $G[X]$ contains an orderable $K_{a,b}$ with an ordering $(u_1, \ldots, u_{a+b})$, and let $A, B$ be its partite sets.
Let $V(G)\setminus X = \{v_1, \ldots, v_{|V(G)\setminus X|}\}$.
We can obtain a copy of $K_{a+|V(G)\setminus X|, b}$ from this $K_{a,b}$ by adding the vertices of $V(G)\setminus X$ to $A$ and adding the edges between $V(G)\setminus X$ and $B$.
Since $X$ is a color-homogeneous set, this $K_{a+|V(G)\setminus X|, b}$ is orderable under the ordering $(v_1, \ldots, v_{|V(G)\setminus X|}, u_1, \ldots, u_{a+b})$.
Similarly, we can obtain an orderable $K_{a,b+|V(G)\setminus X|}$.
\hfill$\blacksquare$
\vspace{0.2cm}

We shall use the following version of Hoeffding's inequality~\cite{Hof}.

\begin{lemma}{\normalfont (\cite{Hof})}\label{le:Chernoff}
Let $X_1, X_2, \ldots, X_n$ be independent Bernoulli random variables such that for each $i\in [n]$, $X_i\sim \bernoulli(p)$, where $0 < p < 1$.
Then for any $a\geq 0$, we have $\pr\left[\sum_{i\in [n]}X_i\geq pn+a\right]\leq e^{-2a^2/n}.$
\end{lemma}

\subsection{Upper bounds on $CR(K_{s,t}, K_3)$}
\label{subsec:pf_complete_bipar_upper}

In this subsection, we prove Theorem~\ref{thm:CRKstUpper}.
We start with some additional terminology.
Let $x$ be a real number and $n$ be a nonnegative integer.
We view ${x\choose n}$ as a function on the reals by defining
$${x\choose n}\colonequals
\left\{
   \begin{aligned}
    &x(x-1)\cdots(x-n+1)/n! & & \mbox{if $x\geq n-1$},\\
    &0 & & \mbox{if $x<n-1$}.
   \end{aligned}
\right.$$
As usual, $0!\colonequals 1$ and ${x\choose 0}\colonequals 1$.
We will also use the falling factorial that is defined as
$(x)_0\colonequals 1$ and
\begin{equation*}
(x)_n\colonequals
\left\{
   \begin{aligned}
    &x(x-1)\cdots (x-n+1) & & \mbox{if $x\geq n-1$},\\
    &0 & & \mbox{if $x<n-1$}.
   \end{aligned}
\right.
\end{equation*}
for $n\geq 1$.
For convenience, we further set $(x)_{n'}\colonequals 0$ for any negative integer $n'$.
For nonnegative integers $x, m, n$, we shall use the following identities:
\begin{equation}\label{eq:falling-1}
(x)_n={x\choose n}n!,
\end{equation}
\begin{equation}\label{eq:falling-2}
(x)_{m+n}=(x)_m(x-m)_n=(x)_n(x-n)_m,
\end{equation}
\begin{equation}\label{eq:falling-3}
(x+1)_n=(x)_n+n(x)_{n-1}.
\end{equation}
For convenience, we introduce some notation that will be used throughout this section.
Let
\begin{equation}\label{eq:caD}
c_s\colonequals \frac{2^s}{s+1}, ~~~~~~~~ \alpha_s\colonequals \frac{1}{c_s}=\frac{s+1}{2^s}, ~~~~~~~~ D_s\colonequals \frac{(s-1)(s+6)}{6}.
\end{equation}
For $s\geq 1$ and $t\geq 1$, we define $N_s(t)$ as follows:
$$N_1(t)\colonequals t+1, ~~~~~~ N_2(t)\colonequals \left\lfloor \frac{4t}{3}\right\rfloor+2, ~~~~~~ N_s(t)\colonequals \left\lceil c_s(t-1)+D_s\right\rceil+1~~\mbox{for $s\geq 3$.}$$
Now Theorem~\ref{thm:CRKstUpper} is equivalent to $CR(K_{s,t}, K_3)\leq N_s(t)$ for $t\geq s\geq 2$.
For integers $q\geq s\geq 1$ and $0\leq d\leq q$, define
$$F_{q,s}(d)\colonequals \sum\nolimits_{j=0}^{s}(d)_j(q-d)_{s-j}, ~~~~~~~~ A_s(q)\colonequals \frac{F_{q,s}\left(\left\lfloor q/2\right\rfloor\right)}{(q)_{s-1}}.$$

The following lemma provides a lower bound on $F_{q,s}(d)$.
Moreover, Lemma~\ref{le:CRKstUpper-2colored} below provides a condition that forces a 2-edge-colored complete graph to contain an orderable $K_{s,t}$.

\begin{lemma}\label{le:CRKstUpper-balancing}
For integers $q\geq s\geq 1$ and $0\leq d\leq q$, we have $F_{q,s}(d)\geq F_{q,s}\left(\left\lfloor q/2\right\rfloor\right).$
\end{lemma}

\noindent {\bf Proof.}
%\begin{proof}
For nonnegative integers $a$ and $b$, define $\Phi_s(a,b)\colonequals \sum\nolimits_{j=0}^{s}(a)_j(b)_{s-j}.$
Note that $F_{q,s}(d)=\Phi_s(d,q-d)$, and $\Phi_s(a,b)$ is symmetric in the sense $\Phi_s(a,b)=\Phi_s(b,a)$.
Hence, it suffices to show that $\Phi_s(a+1,b-1)\leq \Phi_s(a,b)$ whenever $a\leq b-1$.
By Equality~(\ref{eq:falling-3}), we have
\begin{align*}
  ~ &~\Phi_s(a+1,b-1)-\Phi_s(a,b) = \sum\nolimits_{j=0}^{s}\big((a+1)_j(b-1)_{s-j}-(a)_j(b)_{s-j}\big) \\
  = &~\sum\nolimits_{j=0}^{s}\Big(\big((a)_j+j(a)_{j-1}\big)(b-1)_{s-j}-(a)_j\big((b-1)_{s-j}+(s-j)(b-1)_{s-j-1}\big)\Big) \\
  = &~\sum\nolimits_{j=0}^{s}\big(j(a)_{j-1}(b-1)_{s-j}-(s-j)(a)_j(b-1)_{s-j-1}\big) \\
  = &~\sum\nolimits_{j=1}^{s}j(a)_{j-1}(b-1)_{s-j} - \sum\nolimits_{j=0}^{s-1}(s-j)(a)_j(b-1)_{s-j-1} \\
  = &~\sum\nolimits_{i=0}^{s-1}(i+1)(a)_{i}(b-1)_{s-i-1} + \sum\nolimits_{i=0}^{s-1}(i-s)(a)_i(b-1)_{s-i-1} \\
  = &~\sum\nolimits_{i=0}^{s-1}(2i+1-s)(a)_{i}(b-1)_{s-i-1}.
\end{align*}
When $s-1$ is even, the central term (i.e. the term with $i=\frac{s-1}{2}$) has coefficient $0$.
Hence, we further have the following equality regardless of the parity of $s$:
\begin{align}\label{eq:CRKstUpper-balancing-1}
  ~ &~\Phi_s(a+1,b-1)-\Phi_s(a,b) \nonumber\\
  = &~\sum\nolimits_{i=0}^{\left\lfloor \frac{s-1}{2}\right\rfloor}(2i+1-s)(a)_{i}(b-1)_{s-i-1}
   + \sum\nolimits_{i=\left\lceil \frac{s-1}{2}\right\rceil}^{s-1}(2i+1-s)(a)_{i}(b-1)_{s-i-1} \nonumber\\
  = &~\sum\nolimits_{\ell=0}^{\left\lfloor \frac{s-1}{2}\right\rfloor}(2\ell+1-s)(a)_{\ell}(b-1)_{s-\ell-1}
   + \sum\nolimits_{\ell=0}^{\left\lfloor \frac{s-1}{2}\right\rfloor}(2(s-1-\ell)+1-s)(a)_{s-1-\ell}(b-1)_{s-(s-\ell-1)-1} \nonumber\\
  = &~\sum\nolimits_{\ell=0}^{\left\lfloor \frac{s-1}{2}\right\rfloor}(2\ell+1-s)\big((a)_{\ell}(b-1)_{s-\ell-1}-(a)_{s-1-\ell}(b-1)_{\ell}\big).
\end{align}
Now since $2\ell+1-s\leq 0$ for $0\leq \ell\leq \left\lfloor \frac{s-1}{2}\right\rfloor$, it suffices to show that $(a)_{\ell}(b-1)_{s-\ell-1}\geq (a)_{s-1-\ell}(b-1)_{\ell}$ for any $0\leq \ell \leq \left\lfloor \frac{s-1}{2}\right\rfloor$.
This holds immediately if $(a)_{s-1-\ell}(b-1)_{\ell}=0$.
If $(a)_{s-1-\ell}(b-1)_{\ell}\neq 0$, then $(a)_{s-1-\ell}(b-1)_{\ell}> 0$, so by Equality~(\ref{eq:falling-2}) we have
$$\frac{(a)_{\ell}(b-1)_{s-\ell-1}}{(a)_{s-1-\ell}(b-1)_{\ell}} = \frac{(b-1-\ell)_{s-1-2\ell}}{(a-\ell)_{s-1-2\ell}} \geq 1,$$
so $(a)_{\ell}(b-1)_{s-\ell-1}\geq (a)_{s-1-\ell}(b-1)_{\ell}$.
This completes the proof.
%\end{proof}
\hfill$\blacksquare$
%\vspace{0.2cm}

\begin{lemma}\label{le:CRKstUpper-2colored}
For integers $q\geq s\geq 1$ and $t\geq 1$, if $A_s(q)>t-1$, then every 2-edge-coloring of $K_{q+1}$ contains an orderable $K_{s,t}$.
\end{lemma}

\noindent {\bf Proof.}
%\begin{proof}
Consider a 2-edge-coloring of $K_{q+1}$ using red and blue.
Let $X$ denote the number of pairs $(\mathbf{x},y)$, where $\mathbf{x}$ is an ordered $s$-tuple of distinct vertices and $y$ is a vertex compatible with $\mathbf{x}$.
On the one hand, we have
\begin{equation}\label{eq:CRKstUpper-2colored-1}
X=\sum\nolimits_{\mathbf{x}}\left|Y_{\chi}(\mathbf{x})\right|,
\end{equation}
where the sum is taken over all ordered $s$-tuples of distinct vertices.
On the other hand, for any vertex $y$, if $y$ is compatible with an ordered $s$-tuple $\mathbf{x}=(x_1, \ldots, x_s)$, then there exists an integer $j$ with $0\leq j\leq s$ such that $x_iy$ is red for all $1\leq i\leq j$, and $x_iy$ is blue for all $j+1\leq i\leq s$.
Let $d_r(y)$ be the number of red edges incident with $y$.
Then $q-d_r(y)$ is the number of blue edges incident with $y$,
and $y$ is compatible with $\sum\nolimits_{j=0}^{s}(d_r(y))_j(q-d_r(y))_{s-j}=F_{q,s}(d_r(y))$ ordered $s$-tuples.
Thus, combining with Lemma~\ref{le:CRKstUpper-balancing}, we have
\begin{equation}\label{eq:CRKstUpper-2colored-2}
X=\sum\nolimits_{y\in V(K_{q+1})}F_{q,s}(d_r(y)) \geq \sum\nolimits_{y\in V(K_{q+1})}F_{q,s}\left(\left\lfloor q/2\right\rfloor\right) = (q+1)F_{q,s}\left(\left\lfloor q/2\right\rfloor\right).
\end{equation}
Note that there are $(q+1)_s$ ordered $s$-tuples of distinct vertices.
Combining with Equality~(\ref{eq:CRKstUpper-2colored-1}) and Inequality~(\ref{eq:CRKstUpper-2colored-2}),
there exists an $\mathbf{x}$ such that
$$\left|Y_{\chi}(\mathbf{x})\right|\geq \frac{1}{(q+1)_s}(q+1)F_{q,s}\left(\left\lfloor q/2\right\rfloor\right) = \frac{F_{q,s}\left(\left\lfloor q/2\right\rfloor\right)}{(q)_{s-1}} = A_s(q).$$
Since $A_s(q)>t-1$ and $\left|Y_{\chi}(\mathbf{x})\right|$ is an integer, we have $\left|Y_{\chi}(\mathbf{x})\right|\geq t$.
By Lemma~\ref{le:Y}, there is an orderable $K_{s,t}$.
%\end{proof}
\hfill$\blacksquare$
\vspace{0.2cm}

The following two technical lemmas provide expansions of $(2m)_r$ and $\frac{F_{2m,s}(m)}{s+1}$ that will be used in our proof of Lemma~\ref{le:CRKstUpper-Asq} below.
For a polynomial $f(z)$, we use $[z^n]f(z)$ to denote the coefficient of $z^n$ in $f(z)$.
We will also use the following identity related to the beta function (see, for example, \cite[Section~8.38]{GrRy}): for positive integers $m$ and $n$,
\begin{equation}\label{eq:beta-integral}
\int_{0}^{1} x^{m-1}(1-x)^{n-1}dx = \frac{(m-1)!(n-1)!}{(m+n-1)!}.
\end{equation}

\begin{lemma}\label{le:CRKstUpper-Eq-1}
For integers $m$ and $r$ with $2m\geq r\geq 0$, we have $$(2m)_r=\sum\nolimits_{k=0}^{\left\lfloor r/2 \right\rfloor}\frac{r!2^{r-2k}}{k!(r-2k)!}(m)_{r-k}.$$
\end{lemma}

\noindent {\bf Proof.}
%\begin{proof}
We compute the coefficient of $z^r$ in the polynomial $(1+2z+z^2)^m$.
On the one hand, we have
\begin{equation}\label{eq:CRKstUpper-Eq-1-1}
[z^r](1+2z+z^2)^m = [z^r](1+z)^{2m} = {2m \choose r} =\frac{(2m)_r}{r!},
\end{equation}
by Equality~(\ref{eq:falling-1}).
On the other hand, to obtain $z^r$ in $(1+2z+z^2)^m$, we choose $k$ factors contributing $z^2$ and $r-2k$ factors contributing $2z$.
By Equalities~(\ref{eq:falling-1}) and (\ref{eq:falling-2}), we have
\begin{align}\label{eq:CRKstUpper-Eq-1-2}
[z^r](1+2z+z^2)^m = &~\sum\nolimits_{k=0}^{\left\lfloor r/2 \right\rfloor} {m\choose k}{m-k \choose r-2k}2^{r-2k} = \sum\nolimits_{k=0}^{\left\lfloor r/2 \right\rfloor} \frac{(m)_k}{k!}\frac{(m-k)_{r-2k}}{(r-2k)!}2^{r-2k} \nonumber\\
= &~\sum\nolimits_{k=0}^{\left\lfloor r/2 \right\rfloor} \frac{(m)_{r-k}2^{r-2k}}{k!(r-2k)!}.
\end{align}
Combining with Equalities~(\ref{eq:CRKstUpper-Eq-1-1}) and (\ref{eq:CRKstUpper-Eq-1-2}), we have
$(2m)_r=\sum\nolimits_{k=0}^{\left\lfloor r/2 \right\rfloor}\frac{r!2^{r-2k}}{k!(r-2k)!}(m)_{r-k}.$
%\end{proof}
\hfill$\blacksquare$
%\vspace{0.2cm}

\begin{lemma}\label{le:CRKstUpper-Eq-2}
For integers $m$ and $s$ with $2m\geq s\geq 1$, we have $$\frac{F_{2m,s}(m)}{s+1}=\sum\nolimits_{k=0}^{\left\lfloor s/2 \right\rfloor}\frac{s!k!}{(s-2k)!(2k+1)!}(m)_{s-k}.$$
\end{lemma}

\noindent {\bf Proof.}
%\begin{proof}
By Equalities~(\ref{eq:falling-1}) and (\ref{eq:beta-integral}), we have
\begin{align}\label{eq:CRKstUpper-Eq-2-1}
\frac{F_{2m,s}(m)}{s+1}= &~\frac{1}{s+1}\sum\nolimits_{j=0}^{s}(m)_j(m)_{s-j} = \frac{1}{s+1}\sum\nolimits_{j=0}^{s}{m\choose j}j!{m\choose s-j}(s-j)! \nonumber\\
= &~\frac{1}{s+1}\sum\nolimits_{j=0}^{s}{m\choose j}{m\choose s-j}(s+1)!\int_{0}^{1} x^{j}(1-x)^{s-j}dx \nonumber\\
= &~s!\int_{0}^{1} \left(\sum\nolimits_{j=0}^{s}{m\choose j}{m\choose s-j}x^{j}(1-x)^{s-j}\right) dx.
\end{align}
Note that
$$\sum\nolimits_{j=0}^{s}{m\choose j}{m\choose s-j}x^{j}(1-x)^{s-j} = [z^s]\Big(\big(1+xz\big)^m\big(1+(1-x)z\big)^m\Big) = [z^s]\big(1+z+x(1-x)z^2\big)^m.$$
On the other hand, to obtain $z^s$ in $(1+z+x(1-x)z^2)^m$, we choose $k$ factors contributing $x(1-x)z^2$ and $s-2k$ factors contributing $z$.
By Equalities~(\ref{eq:falling-1}) and (\ref{eq:falling-2}), we have
\begin{align*}
~ &~[z^s]\big(1+z+x(1-x)z^2\big)^m = \sum\nolimits_{k=0}^{\left\lfloor s/2 \right\rfloor} {m\choose k}{m-k \choose s-2k}(x(1-x))^{k} \nonumber\\
= &~\sum\nolimits_{k=0}^{\left\lfloor s/2 \right\rfloor} \frac{(m)_k}{k!}\frac{(m-k)_{s-2k}}{(s-2k)!}(x(1-x))^{k} = \sum\nolimits_{k=0}^{\left\lfloor s/2 \right\rfloor} \frac{(m)_{s-k}}{k!(s-2k)!}x^k(1-x)^k.
\end{align*}
Hence, we have
\begin{equation}\label{eq:CRKstUpper-Eq-2-2}
\sum\nolimits_{j=0}^{s}{m\choose j}{m\choose s-j}x^{j}(1-x)^{s-j} = \sum\nolimits_{k=0}^{\left\lfloor s/2 \right\rfloor} \frac{(m)_{s-k}}{k!(s-2k)!}x^k(1-x)^k.
\end{equation}
By Equalities~(\ref{eq:beta-integral}), (\ref{eq:CRKstUpper-Eq-2-1}) and (\ref{eq:CRKstUpper-Eq-2-2}), we have
\begin{align*}
\frac{F_{2m,s}(m)}{s+1}= &~s!\int_{0}^{1} \left(\sum\nolimits_{k=0}^{\left\lfloor s/2 \right\rfloor} \frac{(m)_{s-k}}{k!(s-2k)!}x^k(1-x)^k\right) dx  = \sum\nolimits_{k=0}^{\left\lfloor s/2 \right\rfloor} \frac{s!(m)_{s-k}}{k!(s-2k)!}\int_{0}^{1} x^k(1-x)^k dx \nonumber\\
= &~ \sum\nolimits_{k=0}^{\left\lfloor s/2 \right\rfloor} \frac{s!(m)_{s-k}}{k!(s-2k)!} \frac{k!k!}{(2k+1)!}  = \sum\nolimits_{k=0}^{\left\lfloor s/2 \right\rfloor}\frac{s!k!}{(s-2k)!(2k+1)!}(m)_{s-k}.
\end{align*}
The proof of Lemma~\ref{le:CRKstUpper-Eq-2} is complete.
%\end{proof}
\hfill$\blacksquare$
\vspace{0.2cm}

The following lemma provides a lower bound on $A_s(q)$.
In our proof of Theorem~\ref{thm:CRKstUpper}, we will use Lemmas~\ref{le:CRKstUpper-2colored} and \ref{le:CRKstUpper-Asq} to find an orderable $K_{s,t}$.

\begin{lemma}\label{le:CRKstUpper-Asq}
For integers $q> s\geq 2$, we have $A_s(q)\geq \alpha_s(q-D_s)$, with equality if and only if $s=2$ and $q$ is even.
In particular, for $q> s\geq 3$, we have $A_s(q)> \alpha_s(q-D_s)$.
\end{lemma}

\noindent {\bf Proof.}
%\begin{proof}
For $s\geq 2$ and $0\leq k\leq \left\lfloor s/2 \right\rfloor$, let
$$\kappa_s\colonequals \frac{s(s-1)}{6}, ~~~~~~ \beta_{s,k}\colonequals \frac{s!}{4^kk!(s-2k)!}, ~~~~~~ \rho_k\colonequals \frac{4^k(k!)^2}{(2k+1)!}.$$
Note that $\kappa_s, \beta_{s,k}, \rho_k$ are positive.
Moreover, $D_s=s-1+\kappa_s$ and $(q-D_s)(q)_{s-1}=(q-s+1-\kappa_s)(q)_{s-1} = (q)_s-\kappa_s(q)_{s-1}$.
Thus the inequality $A_s(q)\geq \alpha_s(q-D_s)$ is equivalent to $\frac{F_{q,s}(\lfloor q/2\rfloor)}{s+1} \geq \frac{1}{2^s} \left((q)_s-\kappa_s(q)_{s-1}\right).$
It suffices to show that
\begin{equation}\label{eq:CRKstUpper-Asq-1}
 \frac{F_{q,s}(\lfloor q/2\rfloor)}{s+1}- \frac{1}{2^s} \big((q)_s-\kappa_s(q)_{s-1}\big) \geq 0,
\end{equation}
with equality if and only if $s=2$ and $q$ is even.
We consider the even and odd cases separately.
\vspace{0.2cm}

{\bf Case~1.} $q$ is even.
\vspace{0.2cm}

In this case, we may assume that $q=2m$ for some positive integer $m$.
Let $$\Delta^{\mathrm e}_{s,m} \colonequals \frac{F_{2m,s}(m)}{s+1} - \frac{1}{2^s} \big((2m)_s-\kappa_s(2m)_{s-1}\big).$$
Then Inequality~(\ref{eq:CRKstUpper-Asq-1}) is equivalent to $\Delta^{\mathrm e}_{s,m}\geq 0$.
We now show that $\Delta^{\mathrm e}_{s,m}\geq 0$, with equality if and only if $s=2$.

By Lemma~\ref{le:CRKstUpper-Eq-2}, for $0\leq k\leq \left\lfloor s/2 \right\rfloor$, the coefficient of $(m)_{s-k}$ in the expansion of $\frac{F_{2m,s}(m)}{s+1}$ is
$$\frac{s!\,k!}{(s-2k)!(2k+1)!}= \frac{s!}{4^kk!(s-2k)!} \frac{4^k(k!)^2}{(2k+1)!}= \beta_{s,k}\rho_k.$$
By Lemma~\ref{le:CRKstUpper-Eq-1}, for $0\leq k\leq \left\lfloor s/2 \right\rfloor$, the coefficient of $(m)_{s-k}$ in the expansion of $\frac{1}{2^s} (2m)_s$ is
$$\frac{1}{2^s} \frac{s! 2^{s-2k}}{k!(s-2k)!} = \frac{s!}{4^kk!(s-2k)!} = \beta_{s,k}.$$
We next consider $\frac{1}{2^s} \kappa_s(2m)_{s-1}$.
Applying Lemma~\ref{le:CRKstUpper-Eq-1}, the terms with nonzero coefficients in the expansion of $(2m)_{s-1}$ are
$$(m)_{s-1}, (m)_{s-2}, \ldots, (m)_{s-\left\lfloor (s-1)/2 \right\rfloor}, (m)_{s-1-\left\lfloor (s-1)/2 \right\rfloor}.$$
In particular, there is no $(m)_{s}$, and for odd $s$, the term $(m)_{s-1-\left\lfloor (s-1)/2 \right\rfloor}=(m)_{(s-1)/2}=(m)_{s-(s+1)/2}$ does not appear in the expansion of $\frac{F_{2m,s}(m)}{s+1}$ or $\frac{1}{2^s} (2m)_s$.
For $1\leq k\leq \left\lfloor (s-1)/2 \right\rfloor+1$, the coefficient of $(m)_{s-k}=(m)_{(s-1)-(k-1)}$ in the expansion of $\frac{1}{2^s} \kappa_s(2m)_{s-1}$ is $\frac{\kappa_s}{2^s} \frac{(s-1)! 2^{(s-1)-2(k-1)}}{(k-1)!(s-1-2(k-1))!}.$
When $s$ is even, or $s$ is odd and $1\leq k\leq \left\lfloor (s-1)/2 \right\rfloor$, we have
\begin{align*}
 ~ &~\frac{\kappa_s}{2^s} \frac{(s-1)! 2^{(s-1)-2(k-1)}}{(k-1)!(s-1-2(k-1))!} = \frac{\kappa_s}{2^s} \frac{(s-1)! 2^{s-2k+1}}{(k-1)!(s-2k+1)!} = \frac{\kappa_s}{2^s} \frac{(s-1)! 2^{s-2k+1}}{(k-1)!(s-2k+1)!}\frac{1}{\beta_{s,k}}\beta_{s,k} \\
 = &~\frac{1}{2^s} \frac{s(s-1)}{6} \frac{(s-1)! 2^{s-2k+1}}{(k-1)!(s-2k+1)!}\frac{4^kk!(s-2k)!}{s!}\beta_{s,k} = \frac{k(s-1)}{3(s-2k+1)}\beta_{s,k}.
\end{align*}
Since $\frac{k(s-1)}{3(s-2k+1)}\beta_{s,k}=0$ when $k=0$, we can also view the coefficient of $(m)_{s}$ as $\frac{k(s-1)}{3(s-2k+1)}\beta_{s,k}$.
Therefore, for even $s$, we have
\begin{equation}\label{eq:CRKstUpper-Asq-2}
\Delta^{\mathrm e}_{s,m}=\sum\nolimits_{k=0}^{\left\lfloor s/2 \right\rfloor} \beta_{s,k} \left(\rho_k-1+\frac{k(s-1)}{3(s-2k+1)}\right)(m)_{s-k},
\end{equation}
and for odd $s$, we have
\begin{equation}\label{eq:CRKstUpper-Asq-3}
\Delta^{\mathrm e}_{s,m}=\sum\nolimits_{k=0}^{\left\lfloor s/2 \right\rfloor} \beta_{s,k} \left(\rho_k-1+\frac{k(s-1)}{3(s-2k+1)}\right)(m)_{s-k} + \frac{\kappa_s}{2^s} \frac{(s-1)!}{((s-1)/2)!}(m)_{(s-1)/2}.
\end{equation}

\begin{claim}\label{cl:CRKstUpper-Asq-1}
For $k\in \{0,1\}$, $\rho_k-1+\frac{k(s-1)}{3(s-2k+1)}=0$; for $2\leq k\leq \left\lfloor s/2 \right\rfloor$, $\rho_k-1+\frac{k(s-1)}{3(s-2k+1)}>0$.
\end{claim}

\begin{proof}
For $k=0$, we have $\rho_k-1+\frac{k(s-1)}{3(s-2k+1)}=1-1+0=0$.
For $k=1$, we have $\rho_k-1+\frac{k(s-1)}{3(s-2k+1)}=\frac{2}{3}-1+\frac{1}{3}=0$.
For $2\leq k\leq \left\lfloor s/2 \right\rfloor$, we have $\frac{k(s-1)}{3(s-2k+1)}\geq \frac{k}{3}$, so $\rho_k-1+\frac{k(s-1)}{3(s-2k+1)}\geq \rho_k-1+\frac{k}{3}$.
If $k=2$, then $\rho_k-1+\frac{k}{3}=\frac{8}{15}-1+\frac{2}{3}>0$;
if $k\geq 3$, then $\rho_k-1+\frac{k}{3}\geq \rho_k>0$.
\end{proof}

For $s=2$, we have $\Delta^{\mathrm e}_{s,m}=0$ by Equality~(\ref{eq:CRKstUpper-Asq-2}) and Claim~\ref{cl:CRKstUpper-Asq-1}.
For even $s\geq 4$, we have $\beta_{s,k}>0$, $(m)_{s-k}\geq 0$ when $2\leq k\leq \left\lfloor s/2 \right\rfloor$, and $(m)_{s/2}> 0$.
Then, by Equality~(\ref{eq:CRKstUpper-Asq-2}) and Claim~\ref{cl:CRKstUpper-Asq-1}, we have $\Delta^{\mathrm e}_{s,m}>0$.
For odd $s\geq 3$, we have $\beta_{s,k}>0$, $(m)_{s-k}\geq 0$ when $2\leq k\leq \left\lfloor s/2 \right\rfloor$, and $\frac{\kappa_s}{2^s} \frac{(s-1)!}{((s-1)/2)!}(m)_{(s-1)/2}> 0$.
Then, by Equality~(\ref{eq:CRKstUpper-Asq-3}) and Claim~\ref{cl:CRKstUpper-Asq-1}, we have $\Delta^{\mathrm e}_{s,m}>0$.
\vspace{0.2cm}

{\bf Case~2.} $q$ is odd.
\vspace{0.2cm}

In this case, we may assume that $q=2m+1$ for some positive integer $m$.
Note that $q>s$ implies $2m\geq s$ since $q$ and $s$ are integers.
Let $$\Delta^{\mathrm o}_{s,m} \colonequals \frac{F_{2m+1,s}(m)}{s+1} - \frac{1}{2^s} \big((2m+1)_s-\kappa_s(2m+1)_{s-1}\big).$$
Then Inequality~(\ref{eq:CRKstUpper-Asq-1}) is equivalent to $\Delta^{\mathrm o}_{s,m}\geq 0$.
We now show that $\Delta^{\mathrm o}_{s,m}> 0$.
In order to apply Lemmas~\ref{le:CRKstUpper-Eq-1} and \ref{le:CRKstUpper-Eq-2}, we need the following claim.

\begin{claim}\label{cl:CRKstUpper-Asq-2}
$\Delta^{\mathrm o}_{s,m} = \frac{F_{2m,s}(m)}{s+1} + \frac{1}{2}F_{2m,s-1}(m) - \frac{1}{2^s}(2m)_s-\frac{s-\kappa_s}{2^s}(2m)_{s-1}+\frac{\kappa_s(s-1)}{2^s}(2m)_{s-2}.$
\end{claim}

\begin{proof}
By Equality~(\ref{eq:falling-3}), we have
\begin{align*}
~ &~(2m+1)_s-\kappa_s(2m+1)_{s-1} = (2m)_s+s(2m)_{s-1} - \kappa_s \big((2m)_{s-1}+(s-1)(2m)_{s-2}\big) \\
= &~(2m)_s +(s-\kappa_s)(2m)_{s-1} -\kappa_s(s-1)(2m)_{s-2}.
\end{align*}
Now it suffices to show that $\frac{F_{2m+1,s}(m)}{s+1}= \frac{F_{2m,s}(m)}{s+1} + \frac{1}{2}F_{2m,s-1}(m).$
To prove this, let $T_j\colonequals (m)_j(m)_{s-1-j}$.
Then
\begin{equation}\label{eq:CRKstUpper-Asq-4}
F_{2m,s-1}(m)=\sum\nolimits_{j=0}^{s-1}(m)_j(m)_{s-1-j}=\sum\nolimits_{j=0}^{s-1}T_j,
\end{equation}
and
\begin{align}\label{eq:CRKstUpper-Asq-5}
~ &~F_{2m+1,s}(m)- F_{2m,s}(m) =\sum\nolimits_{j=0}^{s} (m)_j(m+1)_{s-j} - \sum\nolimits_{j=0}^{s} (m)_j(m)_{s-j}  \nonumber\\
= &~\sum\nolimits_{j=0}^{s} (m)_j\big((m+1)_{s-j}-(m)_{s-j}\big) = \sum\nolimits_{j=0}^{s} (m)_j(s-j)(m)_{s-j-1} \nonumber\\
= &~\sum\nolimits_{j=0}^{s-1} (s-j)(m)_j(m)_{s-j-1} = \sum\nolimits_{j=0}^{s-1} (s-j)T_j = s\sum\nolimits_{j=0}^{s-1} T_j-\sum\nolimits_{j=0}^{s-1} jT_j,
\end{align}
where the second line follows from Equality~(\ref{eq:falling-3}).
Moreover, we have
$$\sum\nolimits_{j=0}^{s-1}jT_j=\sum\nolimits_{j=0}^{s-1}(s-1-j)T_{s-1-j}=\sum\nolimits_{j=0}^{s-1}(s-1-j)T_j=(s-1)\sum\nolimits_{j=0}^{s-1}T_j - \sum\nolimits_{j=0}^{s-1}jT_j,$$
where the second equality holds since $T_j=T_{s-1-j}$.
It follows that $\sum\nolimits_{j=0}^{s-1}jT_j = \frac{s-1}{2}\sum\nolimits_{j=0}^{s-1}T_j$.
Combining with Equalities~(\ref{eq:CRKstUpper-Asq-4}) and (\ref{eq:CRKstUpper-Asq-5}), we have
\begin{align*}
~ &~F_{2m+1,s}(m)- F_{2m,s}(m) = s\sum\nolimits_{j=0}^{s-1} T_j-\sum\nolimits_{j=0}^{s-1} jT_j = s\sum\nolimits_{j=0}^{s-1} T_j-\frac{s-1}{2}\sum\nolimits_{j=0}^{s-1}T_j\\
= &~\frac{s+1}{2}\sum\nolimits_{j=0}^{s-1}T_j = \frac{s+1}{2}F_{2m,s-1}(m).
\end{align*}
Hence, we have $\frac{F_{2m+1,s}(m)}{s+1}= \frac{F_{2m,s}(m)}{s+1} + \frac{1}{2}F_{2m,s-1}(m).$
\end{proof}

In Claim~\ref{cl:CRKstUpper-Asq-3} below, we use Lemma~\ref{le:CRKstUpper-Eq-2} to analyze the first two terms in the expression of $\Delta^{\mathrm o}_{s,m}$ given by Claim~\ref{cl:CRKstUpper-Asq-2}.
In Claim~\ref{cl:CRKstUpper-Asq-4} below, we use Lemma~\ref{le:CRKstUpper-Eq-1} to analyze the last three terms in the expression of $\Delta^{\mathrm o}_{s,m}$ given by Claim~\ref{cl:CRKstUpper-Asq-2}.

\begin{claim}\label{cl:CRKstUpper-Asq-3}
For even $s$, we have $$\frac{F_{2m,s}(m)}{s+1} + \frac{1}{2}F_{2m,s-1}(m)=(m)_s+\sum\nolimits_{k=1}^{s/2} \beta_{s,k} \rho_k \frac{s+2}{s-2k+1}(m)_{s-k}.$$
For odd $s$, we have $$\frac{F_{2m,s}(m)}{s+1} + \frac{1}{2}F_{2m,s-1}(m)=(m)_s+\sum\nolimits_{k=1}^{(s-1)/2} \beta_{s,k} \rho_k \frac{s+2}{s-2k+1}(m)_{s-k} + \frac{((s-1)/2)!}{2}(m)_{(s-1)/2}.$$
\end{claim}

\begin{proof}
By the same calculation as in Case~1, for $0\leq k\leq \left\lfloor s/2 \right\rfloor$, the coefficient of $(m)_{s-k}$ in the expansion of $\frac{F_{2m,s}(m)}{s+1}$ is $\beta_{s,k}\rho_k.$
Now we apply Lemma~\ref{le:CRKstUpper-Eq-2} to $\frac{1}{2}F_{2m,s-1}(m)=\frac{s}{2}\frac{F_{2m,s-1}(m)}{(s-1)+1}$.
The terms with nonzero coefficients in the expansion of $\frac{F_{2m,s-1}(m)}{(s-1)+1}$ are
$$(m)_{s-1}, (m)_{s-2}, \ldots, (m)_{s-\left\lfloor (s-1)/2 \right\rfloor}, (m)_{s-1-\left\lfloor (s-1)/2 \right\rfloor}.$$
In particular, there is no $(m)_{s}$, and for odd $s$, the term $(m)_{s-1-\left\lfloor (s-1)/2 \right\rfloor}=(m)_{(s-1)/2}=(m)_{s-(s+1)/2}$ does not appear in the expansions of $\frac{F_{2m,s}(m)}{s+1}$.
For $1\leq k\leq \left\lfloor (s-1)/2 \right\rfloor+1$, the coefficient of $(m)_{s-k}=(m)_{(s-1)-(k-1)}$ in $\frac{s}{2}\frac{F_{2m,s-1}(m)}{(s-1)+1}$ is
$$\frac{s}{2}\frac{(s-1)!(k-1)!}{(s-1-2(k-1))!(2(k-1)+1)!}=\frac{s!(k-1)!}{2(s-2k+1)!(2k-1)!}.$$
When $s$ is even, or $s$ is odd and $1\leq k\leq \left\lfloor (s-1)/2 \right\rfloor$, we have
\begin{align*}
 ~ &~\frac{s!(k-1)!}{2(s-2k+1)!(2k-1)!} = \frac{s!(k-1)!}{2(s-2k+1)!(2k-1)!}\frac{1}{\beta_{s,k}\rho_k}\beta_{s,k}\rho_k \\
 = &~\frac{s!(k-1)!}{2(s-2k+1)!(2k-1)!} \frac{4^kk!(s-2k)!}{s!} \frac{(2k+1)!}{4^k(k!)^2} \beta_{s,k}\rho_k = \frac{2k+1}{s-2k+1}\beta_{s,k}\rho_k.
\end{align*}
Therefore, for even $s$, we have
\begin{align*}
\frac{F_{2m,s}(m)}{s+1} + \frac{1}{2}F_{2m,s-1}(m)= &~\beta_{s,0} \rho_0 (m)_s+\sum\nolimits_{k=1}^{\left\lfloor s/2 \right\rfloor} \beta_{s,k} \rho_k \left(1+\frac{2k+1}{s-2k+1}\right)(m)_{s-k} \\
= &~(m)_s+\sum\nolimits_{k=1}^{s/2} \beta_{s,k} \rho_k \frac{s+2}{s-2k+1}(m)_{s-k},
\end{align*}
and for odd $s$, we have
\begin{align*}
~ &~\frac{F_{2m,s}(m)}{s+1} + \frac{1}{2}F_{2m,s-1}(m) \\
= &~\beta_{s,0} \rho_0(m)_s+\sum\nolimits_{k=1}^{\left\lfloor s/2 \right\rfloor} \beta_{s,k} \rho_k \left(1+\frac{2k+1}{s-2k+1}\right)(m)_{s-k} \\
~ &~+ \frac{s!\Big(\big((s-1)/2+1\big)-1\Big)!}{2\Big(s-2\big((s-1)/2+1\big)+1\Big)!\Big(2\big((s-1)/2+1\big)-1\Big)!}(m)_{s-(s+1)/2} \\
= &~(m)_s+\sum\nolimits_{k=1}^{(s-1)/2} \beta_{s,k} \rho_k \frac{s+2}{s-2k+1}(m)_{s-k} + \frac{((s-1)/2)!}{2}(m)_{(s-1)/2}.
\end{align*}
The proof of Claim~\ref{cl:CRKstUpper-Asq-3} is complete.
\end{proof}

\begin{claim}\label{cl:CRKstUpper-Asq-4}
For even $s$, we have
\begin{align*}
~ &~\frac{1}{2^s}(2m)_s+\frac{s-\kappa_s}{2^s}(2m)_{s-1}-\frac{\kappa_s(s-1)}{2^s}(2m)_{s-2} \\
= &~\sum\nolimits_{k=0}^{s/2} \beta_{s,k} \frac{s+2}{s-2k+1} \frac{3(s+1)-k(s+3)}{3(s-2k+2)}(m)_{s-k} - \frac{(s-1)s!}{6\cdot 2^s(s/2-1)!}(m)_{s/2-1}.
\end{align*}
For odd $s$, we have
\begin{align*}
~ &~\frac{1}{2^s}(2m)_s+\frac{s-\kappa_s}{2^s}(2m)_{s-1}-\frac{\kappa_s(s-1)}{2^s}(2m)_{s-2} \\
= &~\sum\nolimits_{k=0}^{(s-1)/2} \beta_{s,k} \frac{s+2}{s-2k+1} \frac{3(s+1)-k(s+3)}{3(s-2k+2)}(m)_{s-k} + \frac{s!(-s^2+s+6)}{6\cdot 2^s ((s-1)/2)!}(m)_{(s-1)/2}.
\end{align*}
\end{claim}

\begin{proof}
By the same argument as in Case~1, for $0\leq k\leq \left\lfloor s/2 \right\rfloor$, the coefficient of $(m)_{s-k}$ in the expansion of $\frac{1}{2^s} (2m)_s$ is $\beta_{s,k}.$
Now we apply Lemma~\ref{le:CRKstUpper-Eq-1} to $\frac{s-\kappa_s}{2^s}(2m)_{s-1}$.
The terms with nonzero coefficients in the expansion of $(2m)_{s-1}$ are
$$(m)_{s-1}, (m)_{s-2}, \ldots, (m)_{s-\left\lfloor (s-1)/2 \right\rfloor}, (m)_{s-1-\left\lfloor (s-1)/2 \right\rfloor}.$$
In particular, there is no $(m)_{s}$, and for odd $s$, the term $(m)_{s-1-\left\lfloor (s-1)/2 \right\rfloor}=(m)_{(s-1)/2}=(m)_{s-(s+1)/2}$ does not appear in the expansions of $\frac{1}{2^s} (2m)_s$.
For $1\leq k\leq \left\lfloor (s-1)/2 \right\rfloor+1$, the coefficient of $(m)_{s-k}=(m)_{(s-1)-(k-1)}$ in $\frac{s-\kappa_s}{2^s}(2m)_{s-1}$ is $\frac{s-\kappa_s}{2^s} \frac{(s-1)! 2^{(s-1)-2(k-1)}}{(k-1)!(s-1-2(k-1))!}=\frac{s(7-s)}{6\cdot 2^s} \frac{(s-1)! 2^{s-2k+1}}{(k-1)!(s-2k+1)!}.$
When $s$ is even, or $s$ is odd and $1\leq k\leq \left\lfloor (s-1)/2 \right\rfloor$, we have
\begin{align*}
 ~ &~\frac{s(7-s)}{6\cdot 2^s} \frac{(s-1)! 2^{s-2k+1}}{(k-1)!(s-2k+1)!} = \frac{s(7-s)}{6\cdot 2^s} \frac{(s-1)! 2^{s-2k+1}}{(k-1)!(s-2k+1)!}\frac{1}{\beta_{s,k}}\beta_{s,k} \\
 = &~\frac{s(7-s)}{6\cdot 2^s} \frac{(s-1)! 2^{s-2k+1}}{(k-1)!(s-2k+1)!}\frac{4^kk!(s-2k)!}{s!}\beta_{s,k} = \frac{k(7-s)}{3(s-2k+1)}\beta_{s,k}.
\end{align*}
Since $\frac{k(7-s)}{3(s-2k+1)}\beta_{s,k}=0$ when $k=0$, we can also view the coefficient of $(m)_{s}$ as $\frac{k(7-s)}{3(s-2k+1)}\beta_{s,k}$.

Next, we apply Lemma~\ref{le:CRKstUpper-Eq-1} to $\frac{\kappa_s(s-1)}{2^s}(2m)_{s-2}$.
The terms with nonzero coefficients in the expansion of $(2m)_{s-2}$ are
$$(m)_{s-2}, (m)_{s-3}, \ldots, (m)_{s-\left\lfloor (s-2)/2 \right\rfloor}, (m)_{s-1-\left\lfloor (s-2)/2 \right\rfloor}, (m)_{s-2-\left\lfloor (s-2)/2 \right\rfloor}.$$
In particular, there is no $(m)_{s}$ or $(m)_{s-1}$.
Moreover, for odd $s$, the term $(m)_{s-2-\left\lfloor (s-2)/2 \right\rfloor}=(m)_{s-2-(s-3)/2}=(m)_{(s-1)/2}$ does not appear in the expansion of $\frac{1}{2^s} (2m)_s$;
for even $s$, the term $(m)_{s-2-\left\lfloor (s-2)/2 \right\rfloor}=(m)_{s-2-(s-2)/2}=(m)_{s/2-1}$ does not appear in the expansion of $\frac{1}{2^s} (2m)_s$ or $\frac{s-\kappa_s}{2^s}(2m)_{s-1}$.
For $2\leq k\leq \left\lfloor (s-2)/2 \right\rfloor+2$, the coefficient of $(m)_{s-k}=(m)_{(s-2)-(k-2)}$ in $\frac{\kappa_s(s-1)}{2^s}(2m)_{s-2}$ is $\frac{\kappa_s(s-1)}{2^s} \frac{(s-2)!2^{(s-2)-2(k-2)}}{(k-2)!(s-2-2(k-2))!}=\frac{s(s-1)^2}{6\cdot 2^s} \frac{(s-2)!2^{s-2k+2}}{(k-2)!(s-2k+2)!}.$
When $2\leq k\leq \left\lfloor (s-2)/2 \right\rfloor+1$, we have
\begin{align*}
 ~ &~\frac{s(s-1)^2}{6\cdot 2^s} \frac{(s-2)!2^{s-2k+2}}{(k-2)!(s-2k+2)!} = \frac{s(s-1)^2}{6\cdot 2^s} \frac{(s-2)!2^{s-2k+2}}{(k-2)!(s-2k+2)!}\frac{1}{\beta_{s,k}}\beta_{s,k} \\
 = &~\frac{s(s-1)^2}{6\cdot 2^s} \frac{(s-2)!2^{s-2k+2}}{(k-2)!(s-2k+2)!}\frac{4^kk!(s-2k)!}{s!}\beta_{s,k}  = \frac{2k(k-1)(s-1)}{3(s-2k+1)(s-2k+2)}\beta_{s,k}.
\end{align*}
Since $\frac{2k(k-1)(s-1)}{3(s-2k+1)(s-2k+2)}\beta_{s,k}=0$ when $k\in \{0,1\}$, we can also view the coefficients of $(m)_{s}$ and $(m)_{s-1}$ as $\frac{2k(k-1)(s-1)}{3(s-2k+1)(s-2k+2)}\beta_{s,k}$.

In summary, for $0\leq k\leq \left\lfloor (s-2)/2 \right\rfloor+1=\left\lfloor s/2 \right\rfloor$, the coefficient of $(m)_{s-k}$ in $\frac{1}{2^s}(2m)_s+\frac{s-\kappa_s}{2^s}(2m)_{s-1}-\frac{\kappa_s(s-1)}{2^s}(2m)_{s-2}$ is $\left(1+\frac{k(7-s)}{3(s-2k+1)}-\frac{2k(k-1)(s-1)}{3(s-2k+1)(s-2k+2)}\right)\beta_{s,k}.$
Since
\begin{align*}
 ~ &~3(s-2k+1)(s-2k+2)+k(7-s)(s-2k+2)-2k(k-1)(s-1) \\
 = &~3s^2+9s+6-ks^2-5ks-6k = (s+2)\big(3(s+1)-k(s+3)\big),
\end{align*}
we have
\begin{align*}
 ~ &~1+\frac{k(7-s)}{3(s-2k+1)}-\frac{2k(k-1)(s-1)}{3(s-2k+1)(s-2k+2)} \\
 = &~\frac{(s+2)\big(3(s+1)-k(s+3)\big)}{3(s-2k+1)(s-2k+2)} =\frac{s+2}{s-2k+1}\frac{3(s+1)-k(s+3)}{3(s-2k+2)}.
\end{align*}
Therefore, for even $s$,
we have
\begin{align*}
~ &~\frac{1}{2^s}(2m)_s+\frac{s-\kappa_s}{2^s}(2m)_{s-1}-\frac{\kappa_s(s-1)}{2^s}(2m)_{s-2} \\
= &~\sum\nolimits_{k=0}^{\left\lfloor s/2 \right\rfloor} \beta_{s,k} \frac{s+2}{s-2k+1} \frac{3(s+1)-k(s+3)}{3(s-2k+2)}(m)_{s-k} \\
~ &~- \frac{s(s-1)^2}{6\cdot 2^s} \frac{(s-2)!2^{s-2((s-2)/2+2)+2}}{\big(((s-2)/2+2)-2\big)!\big(s-2((s-2)/2+2)+2\big)!}(m)_{s/2-1} \\
= &~\sum\nolimits_{k=0}^{s/2} \beta_{s,k} \frac{s+2}{s-2k+1} \frac{3(s+1)-k(s+3)}{3(s-2k+2)}(m)_{s-k}-\frac{(s-1)s!}{6\cdot 2^s(s/2-1)!}(m)_{s/2-1},
\end{align*}
and for odd $s$, we have
\begin{align*}
~ &~\frac{1}{2^s}(2m)_s+\frac{s-\kappa_s}{2^s}(2m)_{s-1}-\frac{\kappa_s(s-1)}{2^s}(2m)_{s-2} \\
= &~\sum\nolimits_{k=0}^{\left\lfloor s/2 \right\rfloor} \beta_{s,k} \frac{s+2}{s-2k+1} \frac{3(s+1)-k(s+3)}{3(s-2k+2)}(m)_{s-k} \\
~ &~+ \frac{s(7-s)}{6\cdot 2^s} \frac{(s-1)! 2^{s-2((s-1)/2+1)+1}}{(((s-1)/2+1)-1)!(s-2((s-1)/2+1)+1)!}(m)_{(s-1)/2} \\
~ &~- \frac{s(s-1)^2}{6\cdot 2^s} \frac{(s-2)!2^{s-2(\left\lfloor (s-2)/2 \right\rfloor+2)+2}}{((\left\lfloor (s-2)/2 \right\rfloor+2)-2)!(s-2(\left\lfloor (s-2)/2 \right\rfloor+2)+2)!}(m)_{(s-1)/2} \\
= &~\sum\nolimits_{k=0}^{(s-1)/2} \beta_{s,k} \frac{s+2}{s-2k+1} \frac{3(s+1)-k(s+3)}{3(s-2k+2)}(m)_{s-k} \\
~ &~+ \left(\frac{s(7-s)}{6\cdot 2^s} \frac{(s-1)!}{((s-1)/2)!}-\frac{s(s-1)^2}{6\cdot 2^s}\frac{2(s-2)!}{((s-3)/2)!}\right)(m)_{(s-1)/2} \\
= &~\sum\nolimits_{k=0}^{(s-1)/2} \beta_{s,k} \frac{s+2}{s-2k+1} \frac{3(s+1)-k(s+3)}{3(s-2k+2)}(m)_{s-k} + \frac{s!(-s^2+s+6)}{6\cdot 2^s ((s-1)/2)!}(m)_{(s-1)/2}.
\end{align*}
The proof of Claim~\ref{cl:CRKstUpper-Asq-4} is complete.
\end{proof}

By Claims~\ref{cl:CRKstUpper-Asq-2}, \ref{cl:CRKstUpper-Asq-3} and \ref{cl:CRKstUpper-Asq-4}, we can derive that
for even $s$,
\begin{align}\label{eq:CRKstUpper-Asq-6}
\Delta^{\mathrm o}_{s,m}= &~\left(1-\beta_{s,0} \frac{s+2}{s-2\cdot 0+1} \frac{3(s+1)-0\cdot(s+3)}{3(s-2\cdot 0+2)}\right)(m)_s \nonumber\\
~ &~+ \sum\nolimits_{k=1}^{s/2} \beta_{s,k}\frac{s+2}{s-2k+1} \left(\rho_k-\frac{3(s+1)-k(s+3)}{3(s-2k+2)}\right)(m)_{s-k}+\frac{(s-1)s!}{6\cdot 2^s(s/2-1)!}(m)_{s/2-1} \nonumber\\
= &~\sum\nolimits_{k=1}^{s/2} \beta_{s,k}\frac{s+2}{s-2k+1} \left(\rho_k-\frac{3(s+1)-k(s+3)}{3(s-2k+2)}\right)(m)_{s-k}+\frac{(s-1)s!}{6\cdot 2^s(s/2-1)!}(m)_{s/2-1},
\end{align}
and for odd $s$,
\begin{align}\label{eq:CRKstUpper-Asq-7}
\Delta^{\mathrm o}_{s,m}=&~\left(1-\beta_{s,0} \frac{s+2}{s-2\cdot 0+1} \frac{3(s+1)-0\cdot(s+3)}{3(s-2\cdot 0+2)}\right)(m)_s \nonumber\\
~ &~+ \sum\nolimits_{k=1}^{(s-1)/2} \beta_{s,k}\frac{s+2}{s-2k+1} \left(\rho_k-\frac{3(s+1)-k(s+3)}{3(s-2k+2)}\right)(m)_{s-k} \nonumber \\
~ &~+\left(\frac{((s-1)/2)!}{2}-\frac{s!(-s^2+s+6)}{6\cdot 2^s ((s-1)/2)!}\right)(m)_{(s-1)/2} \nonumber\\
= &~\sum\nolimits_{k=1}^{(s-1)/2} \beta_{s,k}\frac{s+2}{s-2k+1} \left(\rho_k-\frac{3(s+1)-k(s+3)}{3(s-2k+2)}\right)(m)_{s-k} \nonumber\\
~&~+\left(\frac{((s-1)/2)!}{2}-\frac{s!(-s^2+s+6)}{6\cdot 2^s ((s-1)/2)!}\right)(m)_{(s-1)/2}.
\end{align}

\begin{claim}\label{cl:CRKstUpper-Asq-5}
For $1\leq k\leq \left\lfloor s/2 \right\rfloor$, we have $\beta_{s,k}\frac{s+2}{s-2k+1} \left(\rho_k-\frac{3(s+1)-k(s+3)}{3(s-2k+2)}\right)(m)_{s-k}\geq 0$.
\end{claim}

\begin{proof}
Since $\beta_{s,k}\frac{s+2}{s-2k+1}(m)_{s-k}\geq 0$, it suffices to show that $\rho_k-\frac{3(s+1)-k(s+3)}{3(s-2k+2)}\geq 0$.
If $k=1$, then $\rho_1-\frac{3(s+1)-(s+3)}{3(s-2+2)}=\frac{2}{3}-\frac{2}{3}=0$.
If $k=2$, then $\rho_2-\frac{3(s+1)-2(s+3)}{3(s-2\cdot 2+2)}=\frac{8}{15}-\frac{s-3}{3(s-2)}\geq \frac{8}{15}-\frac{1}{3}> 0$.
If $3\leq k\leq \left\lfloor s/2 \right\rfloor$, then $3(s+1)-k(s+3)<0$ and $s-2k+2>0$, so $\rho_k-\frac{3(s+1)-k(s+3)}{3(s-2k+2)}> \rho_k > 0$.
\end{proof}

For even $s\geq 2$, we have $\frac{(s-1)s!}{6\cdot 2^s(s/2-1)!}(m)_{s/2-1}>0$.
Then, by Equality~(\ref{eq:CRKstUpper-Asq-6}) and Claim~\ref{cl:CRKstUpper-Asq-5}, we have $\Delta^{\mathrm o}_{s,m}>0$.
For odd $s\geq 3$, we have $-s^2+s+6\leq 0$, so $\left(\frac{((s-1)/2)!}{2}-\frac{s!(-s^2+s+6)}{6\cdot 2^s ((s-1)/2)!}\right)(m)_{(s-1)/2}\geq \frac{((s-1)/2)!}{2}(m)_{(s-1)/2}>0$.
Then, by Equality~(\ref{eq:CRKstUpper-Asq-7}) and Claim~\ref{cl:CRKstUpper-Asq-5}, we have $\Delta^{\mathrm o}_{s,m}>0$.
This completes the proof of Lemma~\ref{le:CRKstUpper-Asq}.
%\end{proof}
\hfill$\blacksquare$
\vspace{0.2cm}

Recall that for $s\geq 1$ and $t\geq 1$, $N_s(t)$ is defined as $N_1(t)\colonequals t+1$, $N_2(t)\colonequals \left\lfloor \frac{4t}{3}\right\rfloor+2$, and $N_s(t)\colonequals \left\lceil c_s(t-1)+D_s\right\rceil+1=\left\lceil \frac{2^s}{s+1}(t-1)+\frac{(s-1)(s+6)}{6}\right\rceil+1$ for $s\geq 3$.
We next state and prove two recursive results for $N_s(t)$ that will be used when we prove Theorem~\ref{thm:CRKstUpper} by induction on $s$ and $t$.
In the proofs of Lemmas~\ref{le:CRKstUpper-Nst-1} and \ref{le:CRKstUpper-Nst-2}, we shall use the fact that $c_s$ is increasing (since $\frac{c_{s+1}}{c_s}=\frac{2^{s+1}}{s+2}/\frac{2^s}{s+1}=\frac{2(s+1)}{s+2}>1$).

\begin{lemma}\label{le:CRKstUpper-Nst-1}
For integers $s\geq 2$, $t\geq 2$ and $1\leq h\leq t-1$, we have $N_s(t)-h\geq N_s(t-h).$
\end{lemma}

\noindent {\bf Proof.}
%\begin{proof}
For $s=2$, we have $N_2(t)-h= \left\lfloor \frac{4t}{3}\right\rfloor+2-h=\left\lfloor \frac{4t}{3}-h\right\rfloor+2 \geq \left\lfloor \frac{4(t-h)}{3}\right\rfloor+2 =N_2(t-h).$
For $s\geq 3$, we have $c_s=\frac{2^s}{s+1}\geq 2$, and thus
$N_s(t)-h=\left\lceil c_s(t-1)+D_s\right\rceil+1-h=\left\lceil c_s(t-1)+D_s-h\right\rceil+1 \geq \left\lceil c_s(t-h-1)+D_s\right\rceil+1 =N_s(t-h).$
%\end{proof}
\hfill$\blacksquare$
%\vspace{0.2cm}

\begin{lemma}\label{le:CRKstUpper-Nst-2}
For integers $s\geq 3$, $t\geq 1$ and $2\leq h\leq s-1$, we have $N_s(t)-h\geq N_{s-h}(2t-1).$
\end{lemma}

\noindent {\bf Proof.}
%\begin{proof}
If $h=s-1$, then $s-h=1$ and
\begin{align*}
N_s(t)-h=&~\left\lceil c_s(t-1)+D_s\right\rceil+1-h \geq \left\lceil c_3(t-1)+\frac{(s-1)(s+6)}{6}+1-h\right\rceil \\
= &~\left\lceil 2(t-1)+\frac{(s-1)(s+6)}{6}+1-(s-1)\right\rceil = \left\lceil 2t+\frac{(s-3)(s+2)}{6}\right\rceil \geq 2t = N_{1}(2t-1).
\end{align*}
If $h=s-2$, then $s-h=2$ and $s=h+2\geq 4$, so
\begin{align*}
N_s(t)-h=&~\left\lceil c_s(t-1)+D_s\right\rceil+1-h \geq \left\lceil c_4(t-1)+\frac{(s-1)(s+6)}{6}+1-h\right\rceil \\
= &~\left\lceil \frac{16}{5}(t-1)+\frac{(s-1)(s+6)}{6}+1-(s-2)\right\rceil \geq \left\lceil \frac{8}{3}(t-1)+\frac{s^2-s+12}{6}\right\rceil \\
\geq &~\left\lceil \frac{8}{3}(t-1)+4\right\rceil > \left\lceil \frac{8t}{3}\right\rceil = \left\lfloor \frac{8t+2}{3}\right\rfloor = \left\lfloor \frac{4(2t-1)}{3}\right\rfloor+2 = N_{2}(2t-1).
\end{align*}

Now we assume that $2\leq h\leq s-3$.
We first show that $c_s> 2c_{s-h}$ and $D_s-h> D_{s-h}$.
For the first inequality, we have $\frac{c_{s}}{2c_{s-h}}=\frac{2^{s}}{s+1}/\left(2\frac{2^{s-h}}{s-h+1}\right)=\frac{2^{h-1}(s-h+1)}{s+1}.$
Since $2\leq h\leq s-3$, we have $2^{h-1}(s-h+1)-(s+1)\geq h(s-h+1)-(s+1)=(h-1)(s-h)-1> 0$, so $c_s> 2c_{s-h}$.
For the second inequality, we have
$6(D_s-h- D_{s-h})=(s-1)(s+6)-6h-(s-h-1)(s-h+6)=h(2s-h-1)>0$,
so $D_s-h> D_{s-h}$.
Now
\begin{align*}
N_s(t)-h= &~\left\lceil c_s(t-1)+D_s\right\rceil+1-h \geq \left\lceil 2c_{s-h}(t-1)+D_{s-h}\right\rceil+1 \\
= &~\left\lceil c_{s-h}(2t-1-1)+D_{s-h}\right\rceil+1 = N_{s-h}(2t-1).
\end{align*}
The proof of Lemma~\ref{le:CRKstUpper-Nst-2} is complete.
%\end{proof}
\hfill$\blacksquare$
\vspace{0.2cm}

Now we have all the ingredients to present our proof of Theorem~\ref{thm:CRKstUpper}.
\vspace{0.2cm}

\noindent {\bf Proof of Theorem~\ref{thm:CRKstUpper}.}
We will prove the following slightly stronger form:
$$CR(K_{s,t}, K_3)\leq N_s(t) \mbox{~~~for $s\geq 1$ and $t\geq 1$}.$$
We prove the statement by induction first on $s$, and then for each fixed $s$, by induction on $t$.
For the base case $s=1$, we have $r'_2(K_{1,t})= CR(K_{1,t}, K_3)=t+1=N_1(t)$ by Proposition~\ref{prop:forest}.
Now assume the result holds for all values smaller than $s$, and we shall prove it for $s\geq 2$.
We first consider the case $s=2$.

%\begin{claim}\label{cl:CRKstUpper-K2t}
%For $t\geq 1$, we have $CR(K_{2,t}, K_3)\leq N_2(t)= \left\lfloor \frac{4t}{3}\right\rfloor+2.$
%\end{claim}

\begingroup
\renewcommand{\theclaim}{\ref{thm:CRKstUpper}.1}
\begin{claim}\label{cl:CRKstUpper-K2t}
For $t\geq 1$, we have $CR(K_{2,t}, K_3)\leq N_2(t)= \left\lfloor \frac{4t}{3}\right\rfloor+2.$
\end{claim}
\endgroup

\begin{proof}
We prove the claim by induction on $t$.
For the base case $t\in \{1,2\}$, we have $CR(K_{2,1}, K_3)=3=\left\lfloor \frac{4}{3}\right\rfloor+2$ and $CR(K_{2,2}, K_3)=4=\left\lfloor \frac{8}{3}\right\rfloor+2$ by Proposition~\ref{prop:forest} and Corollary~\ref{cor:forest}.
Now assume that the claim holds for all values smaller than $t$, and we shall prove it for $t\geq 3$.

For a contradiction, suppose that there exists an edge-colored complete graph $K_n$ with neither an orderable $K_{2,t}$ nor a rainbow $K_3$, where $n=N_2(t)$.
Let $V_1, V_2, \ldots, V_m$ ($m\geq 2$) be a Gallai partition of $V(K_n)$ given by Theorem~\ref{thm:Gallai}.
If $|V_i|=1$ for all $i\in [m]$, then the edge-coloring uses at most two colors.
Moreover, we have $n-1=\left\lfloor \frac{4t}{3}\right\rfloor+1> 2$ for $t\geq 3$.
By Lemma~\ref{le:CRKstUpper-Asq}, we have
$$A_2(n-1)\geq \alpha_2(n-1-D_2) = \frac{3}{4}\left(\left\lfloor \frac{4t}{3}\right\rfloor+1 - \frac{4}{3}\right) \geq \frac{3}{4}\left(\frac{4t-2}{3} - \frac{1}{3}\right) =t-\frac{3}{4} >t-1.$$
Then, by Lemma~\ref{le:CRKstUpper-2colored}, there exists an orderable $K_{2,t}$, a contradiction.
Thus there exists a part, say $V_1$, with $|V_1|\geq 2$.
Note that $V_1$ is a color-homogeneous set.
If $|V(K_n)\setminus V_1|\geq t$, then there exists an orderable $K_{2,t}$ by Fact~\ref{fa:homo}~(i), a contradiction.
Thus $|V(K_n)\setminus V_1|\leq t-1$.
Then, by Lemma~\ref{le:CRKstUpper-Nst-1}, we have
$$|V_1|=n-|V(K_n)\setminus V_1|=N_2(t)-|V(K_n)\setminus V_1|\geq N_2\left(t-|V(K_n)\setminus V_1|\right).$$
By the induction hypothesis, there exists an orderable $K_{2, t-|V(K_n)\setminus V_1|}$ within $V_1$.
Then there exists an orderable $K_{2,t}$ in $K_n$ by Fact~\ref{fa:homo}~(ii), a contradiction.
The proof of Claim~\ref{cl:CRKstUpper-K2t} is complete.
\end{proof}

In the following, we consider the case $s\geq 3$.
The proof is an extension of the case $s=2$.
We prove by induction on $t$.
For the base case $t=1$, we have $CR(K_{s,1}, K_3)=s+1\leq \left\lceil\frac{(s-1)(s+6)}{6}\right\rceil+1=N_s(1)$ by Proposition~\ref{prop:forest}.
Now assume that the claim holds for all values smaller than $t$, and we shall prove it for $t\geq 2$.
For a contradiction, suppose that there exists an edge-colored complete graph $K_n$ with neither an orderable $K_{s,t}$ nor a rainbow $K_3$, where $n=N_s(t)$.
Let $V_1, V_2, \ldots, V_m$ ($m\geq 2$) be a Gallai partition of $V(K_n)$ given by Theorem~\ref{thm:Gallai}.

%\begin{claim}\label{cl:CRKstUpper-Kst}
%There exists a part $V_i$ with $2\leq |V_i|\leq s-1$.
%\end{claim}

\begingroup
\renewcommand{\theclaim}{\ref{thm:CRKstUpper}.2}
\begin{claim}\label{cl:CRKstUpper-Kst}
There exists a part $V_i$ with $2\leq |V_i|\leq s-1$.
\end{claim}
\endgroup

\begin{proof}
If $|V_i|=1$ for all $i\in [m]$, then the edge-coloring uses at most two colors.
Moreover, we have $n-1=\left\lceil \frac{2^s}{s+1}(t-1)+\frac{(s-1)(s+6)}{6}\right\rceil> s$ for $s\geq 3$ and $t\geq 2$.
By Lemma~\ref{le:CRKstUpper-Asq}, we have
$$A_s(n-1)> \alpha_s(n-1-D_s) \geq \alpha_s\left(c_s(t-1)+D_s+1-1-D_s\right)=t-1.$$
Then, by Lemma~\ref{le:CRKstUpper-2colored}, there exists an orderable $K_{s,t}$, a contradiction.
Thus there exists a part $V_i$ with $|V_i|\geq 2$.

Suppose that $|V_i|\geq s$.
Note that $V_i$ is a color-homogeneous set.
If $|V(K_n)\setminus V_i|\geq t$, then there exists an orderable $K_{s,t}$ by Fact~\ref{fa:homo}~(i), a contradiction.
Thus $|V(K_n)\setminus V_i|\leq t-1$.
Then, by Lemma~\ref{le:CRKstUpper-Nst-1}, we have
$$|V_i|=n-|V(K_n)\setminus V_i|=N_s(t)-|V(K_n)\setminus V_i|\geq N_s\left(t-|V(K_n)\setminus V_i|\right).$$
By the induction hypothesis, there exists an orderable $K_{s, t-|V(K_n)\setminus V_i|}$ within $V_i$.
Then there exists an orderable $K_{s,t}$ in $K_n$ by Fact~\ref{fa:homo}~(ii), a contradiction.
Thus $|V_i|\leq s-1$.
The proof is complete.
\end{proof}

By Claim~\ref{cl:CRKstUpper-Kst}, we may assume that $2\leq |V_1|\leq s-1$ without loss of generality.
By Lemma~\ref{le:CRKstUpper-Nst-2}, we have
$$|V(K_n)\setminus V_1|=N_s(t)-|V_1|\geq N_{s-|V_1|}(2t-1).$$
By the induction hypothesis, there exists an orderable $K_{s-|V_1|, 2t-1}$ within $V(K_n)\setminus V_1$.
Let $X$ be its partite set of size $s-|V_1|$ and $Y$ be its partite set of size $2t-1$.
Note that for each vertex $v\in Y\subseteq V(K_n)\setminus V_1$, all edges between $v$ and $V_1$ are of the same color.
Moreover, there are at most two colors on edges between $V(K_n)\setminus V_1$ and $V_1$.
Thus, by the pigeonhole principle, there exists a subset $Y'\subseteq Y$ with $|Y'|\geq \left\lceil \frac{|Y|}{2}\right\rceil= \left\lceil \frac{2t-1}{2}\right\rceil =t$ such that all edges between $Y'$ and $V_1$ are of the same color.
Note that the edges between $X$ and $Y'$ form an orderable $K_{s-|V_1|, t}$.
Then we can obtain an orderable $K_{s,t}$ from this $K_{s-|V_1|, t}$ by putting the vertices of $V_1$ as the first $|V_1|$ vertices.
This contradiction completes the proof of Theorem~\ref{thm:CRKstUpper}.
\hfill$\blacksquare$

\subsection{Lower bounds on $r'_2(K_{s,t})$}
\label{subsec:pf_complete_bipar_lower}

In this subsection, we first provide a probabilistic lower bound on $r'_2(K_{s,t})$ for general $K_{s,t}$ using Hoeffding's inequality; see Lemma~\ref{le:CRKstlower}.
Then we present several constructive lower bounds on $r'_2(K_{2,t})$ and $r'_2(K_{3,t})$ using properties of strongly regular graphs, Hadamard matrices and conference matrices; see Lemmas~\ref{le:CRK2tlower-construction-SRG-1}, \ref{le:CRK2tlower-construction-Hadamard-1}, \ref{le:CRK2tlower-construction-conference-1} and \ref{le:CRK3tlower}.
These four lemmas constitute the proofs of Theorems~\ref{thm:CRK2t3t} and \ref{thm:CRK2t+}.

\begin{lemma}\label{le:CRKstlower}
For any fixed integer $s \geq 2$ and sufficiently large $t$, we have
$$r'_2(K_{s,t})>\left\lfloor\frac{2^s}{s+1}t-\frac{2^s\sqrt{2^s}}{s+1}\sqrt{t \ln t}\right\rfloor=\frac{2^s}{s+1}t-O_s\left(\sqrt{t \ln t}\right).$$
\end{lemma}

\noindent {\bf Proof.}
%\begin{proof}
Let $n=\Big\lfloor\frac{2^s}{s+1}t-\frac{2^s\sqrt{2^s}}{s+1}\sqrt{t \ln t}\Big\rfloor.$
Consider a random $2$-edge-coloring $\chi$ of $K_n$ where each edge is colored red or blue independently, each with probability $\frac{1}{2}$.
Fix an arbitrary ordered $s$-tuple $\mathbf{x}=(x_1, \ldots, x_s)$ of $s$ distinct vertices.
For each vertex $y\in V(K_n)\setminus \{x_1, \ldots, x_s\}$,
the vector $\big(\mathbf{1}_{\scriptsize\{\mbox{$\chi(x_1y)$ is red}\}}, \ldots, \mathbf{1}_{\scriptsize\{\mbox{$\chi(x_sy)$ is red}\}}\big)$ is uniformly distributed in $\{0,1\}^{s}$.
Since exactly $s+1$ of the $2^s$ vectors in $\{0,1\}^{s}$ belong to $\mathcal{C}_s$, we have $|Y_{\chi}(\mathbf{x})|\sim \binomial(n-s, p)$, where $p=\frac{s+1}{2^s}$.
By Lemma~\ref{le:Chernoff}, we have
$$\pr\left[\left|Y_{\chi}(\mathbf{x})\right|\geq p(n-s)+\sqrt{\frac{s+1}{2}(n-s)\ln n}\right]\leq \exp\left(-\frac{2}{n-s}\frac{s+1}{2}(n-s)\ln n\right)=n^{-(s+1)}.$$
Note that there are $(n)_s$ ordered $s$-tuples $\mathbf{x}$ of $s$ distinct vertices.
Applying the union bound, the probability that at least one ordered $s$-tuple $\mathbf{x}$ satisfies $|Y_{\chi}(\mathbf{x})|\geq p(n-s)+\sqrt{\frac{s+1}{2}(n-s)\ln n}$ is at most $(n)_s\cdot n^{-(s+1)}<n^s\cdot n^{-(s+1)}=\frac{1}{n}<1$.
Thus, with positive probability, no ordered $s$-tuple $\mathbf{x}$ satisfies $|Y_{\chi}(\mathbf{x})|\geq p(n-s)+\sqrt{\frac{s+1}{2}(n-s)\ln n}$.
Since $n\leq \frac{2^s}{s+1}t-\frac{2^s\sqrt{2^s}}{s+1}\sqrt{t \ln t}<\frac{2^s}{s+1}t$ and $t$ is sufficiently large, we have $\ln \left(\frac{2^s}{s+1}t\right)\leq 2\ln t$ and
\begin{align*}
  ~ &~p(n-s)+\sqrt{\frac{s+1}{2}(n-s)\ln n} < pn+\sqrt{\frac{s+1}{2}n\ln n} \\
  < &~\frac{s+1}{2^s}\left(\frac{2^s}{s+1}t-\frac{2^s\sqrt{2^s}}{s+1}\sqrt{t \ln t}\right)+\sqrt{\frac{s+1}{2}\frac{2^s}{s+1}t\ln \left(\frac{2^s}{s+1}t\right)} \\
  \leq &~t-\sqrt{2^st\ln t}+\sqrt{2^{s-1}t\cdot 2\ln t} = t.
\end{align*}
Thus, with positive probability, no ordered $s$-tuple $\mathbf{x}$ satisfies $|Y_{\chi}(\mathbf{x})|\geq t$.
By Lemma~\ref{le:Y}, there exists a 2-edge-coloring of $K_n$ without orderable $K_{s,t}$.
Therefore, we have $r'_2(K_{s,t})> n = \left\lfloor\frac{2^s}{s+1}t-\frac{2^s\sqrt{2^s}}{s+1}\sqrt{t \ln t}\right\rfloor=\frac{2^s}{s+1}t-O_s\left(\sqrt{t \ln t}\right).$
The proof of Lemma~\ref{le:CRKstlower} is complete.
%\end{proof}
\hfill$\blacksquare$
\vspace{0.2cm}

Before providing a lower bound on $r'_2(K_{2,t})$, we first present a necessary and sufficient condition for the existence of an orderable $K_{2,t}$ in a 2-edge-colored graph.
Let $G$ be a 2-edge-colored $K_n$ using red and blue.
For any ordered pair $(x_1, x_2)$ of distinct vertices in $G$, let
\begin{align*}
  V_{rr}(x_1, x_2) &\colonequals \{y\in V(G)\setminus \{x_1, x_2\}\colon\, \mbox{$\chi(x_1y)$ is red, $\chi(x_2y)$ is red}\},\hspace{-0.5cm} & & n_{rr}(x_1, x_2)\colonequals |V_{rr}|, \\
  V_{rb}(x_1, x_2) &\colonequals \{y\in V(G)\setminus \{x_1, x_2\}\colon\, \mbox{$\chi(x_1y)$ is red, $\chi(x_2y)$ is blue}\},\hspace{-0.5cm} & & n_{rb}(x_1, x_2)\colonequals |V_{rb}|, \\
  V_{br}(x_1, x_2) &\colonequals \{y\in V(G)\setminus \{x_1, x_2\}\colon\, \mbox{$\chi(x_1y)$ is blue, $\chi(x_2y)$ is red}\},\hspace{-0.5cm} & & n_{br}(x_1, x_2)\colonequals |V_{br}|, \\
  V_{bb}(x_1, x_2) &\colonequals \{y\in V(G)\setminus \{x_1, x_2\}\colon\, \mbox{$\chi(x_1y)$ is blue, $\chi(x_2y)$ is blue}\},\hspace{-0.5cm} & & n_{bb}(x_1, x_2)\colonequals |V_{bb}|.
\end{align*}
For brevity, when the context is clear, we omit $(x_1, x_2)$ from this notation.
Note that $n_{rr}+n_{rb}+n_{br}+n_{bb}=n-2$.

\begin{lemma}\label{le:r2K2tcondition}
Let $G$ be a 2-edge-colored $K_n$ using red and blue.
Then $G$ contains an orderable $K_{2,t}$ if and only if there exists an ordered pair $(x_1, x_2)$ of distinct vertices in $G$ such that $n-2-\min\{n_{rb}, n_{br}\}\geq t$.
\end{lemma}

\noindent {\bf Proof.}
%\begin{proof}
First, we assume that there exists an ordered pair $(x_1, x_2)$ of distinct vertices in $G$ such that $n-2-\min\{n_{rb}, n_{br}\}\geq t$.
Without loss of generality, we may assume that $n_{rb}\leq n_{br}$.
Then $n_{rr}+n_{br}+n_{bb}=n-2-n_{rb}\geq t$.
Let $Y\subseteq V_{rr}\cup V_{br}\cup V_{bb}$ with $|Y|=t$.
Then the edges between $\{x_1, x_2\}$ and $Y$ induce an orderable $K_{2,t}$ under the ordering $(Y\cap (V_{rr}\cup V_{bb}), x_1, Y\cap V_{br}, x_2)$,
where the vertices inside each of $Y\cap (V_{rr}\cup V_{bb})$ and $Y\cap V_{br}$ are ordered arbitrarily.

Next, we assume that $G$ contains an orderable $K_{2,t}$ with bipartition $X\cup Y$, where $X=\{x_1, x_2\}$ and $Y=\{y_1, y_2, \ldots, y_t\}$.
Then for any pair $(y_i, y_j)$ of distinct vertices in $Y$, the edges between $\{x_1, x_2\}$ and $\{y_i, y_j\}$ induce an orderable $K_{2,2}$ (i.e., an orderable $C_4$) that inherits the ordering from the orderable $K_{2,t}$.
Since the alternating $C_4$ is not an orderable $C_4$, we have either $Y\cap V_{rb}=\emptyset$ or $Y\cap V_{br}=\emptyset$.
Thus $t\leq n_{rr}+n_{bb}+\max\{n_{rb}, n_{br}\}=n-2-\min\{n_{rb}, n_{br}\}$.
%\end{proof}
\hfill$\blacksquare$
%\vspace{0.2cm}

\begin{corollary}\label{co:CRK2tlower-construction-SRG}
If there exists an $SRG\left(n, k, \lambda, \mu\right)$ with $\min\{k-1-\lambda, k-\mu\}\geq n-1-t$, then
$r'_2(K_{2,t})\geq n+1.$
\end{corollary}

\noindent {\bf Proof.}
Let $G$ be an $SRG\left(n, k, \lambda, \mu\right)$ with $\min\{k-1-\lambda, k-\mu\}\geq n-1-t$.
We construct a 2-edge-colored $K_n$ by coloring all edges of $G$ with red and all edges of the complement $\overline{G}$ with blue.
Fix an ordered pair $(x_1, x_2)$ of distinct vertices arbitrarily.
If $x_1x_2$ is red, then $x_1$ and $x_2$ are adjacent in $G$, so $n_{rr}=|N_G(x_1)\cap N_G(x_2)|=\lambda$.
Then $n_{rb}=d_G(x_1)-|\{x_2\}|-n_{rr}=k-1-\lambda \geq n-1-t$ and $n_{br}=d_G(x_2)-|\{x_1\}|-n_{rr}=k-1-\lambda \geq n-1-t$.
If $x_1x_2$ is blue, then $x_1$ and $x_2$ are nonadjacent in $G$, so $n_{rr}=|N_G(x_1)\cap N_G(x_2)|=\mu$.
Then $n_{rb}=d_G(x_1)-n_{rr}=k-\mu\geq n-1-t$ and $n_{br}=d_G(x_2)-n_{rr}=k-\mu\geq n-1-t$.
In either case, we have $n-2-\min\{n_{rb}, n_{br}\}\leq n-2-(n-1-t)= t-1$.
By Lemma~\ref{le:r2K2tcondition}, there is no orderable $K_{2,t}$, and thus $r'_2(K_{2,t})\geq n+1.$
\hfill$\blacksquare$
\vspace{0.2cm}

Now we have all the ingredients to present lower bound constructions on $r'_2(K_{2,t})$.
Lemma~\ref{le:CRK2tlower-construction-SRG-1} is a strengthening of Theorem~\ref{thm:CRK2t3t} (i); it not only provides a sufficient condition for $r'_2(K_{2,t})= CR(K_{2,t}, K_3)=\frac{4t}{3}+2$, but actually gives a necessary and sufficient condition for this equality.

\begin{lemma}\label{le:CRK2tlower-construction-SRG-1}
For any positive integer $t$ with $t\equiv 0 \pmod{3}$, let $n=\frac{4t}{3}+1$.
Then $r'_2(K_{2,t})= CR(K_{2,t}, K_3)=\frac{4t}{3}+2$ if and only if there exists an $SRG\left(n, \frac{n-1}{2}, \frac{n-5}{4}, \frac{n-1}{4}\right).$
\end{lemma}

\noindent {\bf Proof.}
%\begin{proof}
First, assume that there exists an $SRG\left(n, \frac{n-1}{2}, \frac{n-5}{4}, \frac{n-1}{4}\right)$.
Since $n=\frac{4t}{3}+1$, we have $t=\frac{3(n-1)}{4}$ and $n-1-t=\frac{n-1}{4}$.
Then $\frac{n-1}{2}-1- \frac{n-5}{4}=\frac{n-1}{4}=n-1-t$ and $\frac{n-1}{2}-\frac{n-1}{4}=\frac{n-1}{4}=n-1-t$.
By Corollary~\ref{co:CRK2tlower-construction-SRG}, we have $r'_2(K_{2,t})\geq n+1= \frac{4t}{3}+2.$
Combining with Theorem~\ref{thm:CRKstUpper}, we have $r'_2(K_{2,t})= CR(K_{2,t}, K_3)=\frac{4t}{3}+2$.

Conversely, assume that $r'_2(K_{2,t})= CR(K_{2,t}, K_3)=\frac{4t}{3}+2$.
Then there exists a 2-edge-coloring of $K_{n}$ with $n=\frac{4t}{3}+1$ that contains no orderable $K_{2,t}$.
We proceed by double-counting $\sum\nolimits_{(x_1, x_2)}n_{rb}(x_1, x_2)$, where the sum is taken over all ordered pairs $(x_1, x_2)$ of distinct vertices of $V(K_n)$.
On the one hand, since there is no orderable $K_{2,t}$, Lemma~\ref{le:r2K2tcondition} gives $n-2-\min\{n_{rb}, n_{br}\}\leq t-1$ for all ordered pairs $(x_1, x_2)$ of distinct vertices.
Thus $n_{rb}\geq \min\{n_{rb}, n_{br}\}\geq n-2-(t-1)=\frac{t}{3}$, so
\begin{equation}\label{eq:CRK2tlower-construction-SRG-1-1}
\sum\nolimits_{(x_1, x_2)}n_{rb}(x_1, x_2)\geq n(n-1)\frac{t}{3},
\end{equation}
with equality if and only if $n_{rb}=\frac{t}{3}=\frac{n-1}{4}$ for all ordered pairs $(x_1, x_2)$ of distinct vertices.
On the other hand, denote the two colors in the edge-coloring by red and blue.
For each vertex $z\in V(K_n)$, let $d_r(z)$ (resp., $d_b(z)$) be the number of red (resp., blue) edges incident with $z$.
Then
\begin{align}\label{eq:CRK2tlower-construction-SRG-1-2}
\sum\nolimits_{(x_1, x_2)}n_{rb}(x_1, x_2)= &~\sum\nolimits_{z\in V(K_n)}d_r(z)d_b(z) = \sum\nolimits_{z\in V(K_n)}d_r(z)(n-1-d_r(z)) \nonumber\\
\leq &~n\frac{(n-1)^2}{4} = n(n-1)\frac{t}{3},
\end{align}
with equality if and only if $d_r(z)=d_b(z)=\frac{n-1}{2}$ for all vertices $z\in V(K_n)$.
By Inequalities~(\ref{eq:CRK2tlower-construction-SRG-1-1}) and (\ref{eq:CRK2tlower-construction-SRG-1-2}), we have $\sum\nolimits_{(x_1, x_2)}n_{rb}(x_1, x_2)=n(n-1)\frac{t}{3}$, $n_{rb}=\frac{t}{3}=\frac{n-1}{4}$ for all ordered pairs $(x_1, x_2)$ of distinct vertices, and $d_r(z)=d_b(z)=\frac{n-1}{2}$ for all vertices $z\in V(K_n)$.
Let $G$ be the spanning subgraph whose edges are precisely the red edges.
Then $G$ is an $n$-vertex $\frac{n-1}{2}$-regular graph.
For any ordered pair $(x_1, x_2)$ of distinct vertices, if $x_1x_2\in E(G)$ (i.e., $x_1x_2$ is red), then
$$|N_G(x_1)\cap N_G(x_2)|=n_{rr}=d_r(x_1)-|\{x_2\}|-n_{rb}=\frac{n-1}{2}-1-\frac{n-1}{4}=\frac{n-5}{4};$$
if $x_1x_2\notin E(G)$ (i.e., $x_1x_2$ is blue), then
$$|N_G(x_1)\cap N_G(x_2)|=n_{rr}=d_r(x_1)-n_{rb}=\frac{n-1}{2}-\frac{n-1}{4}=\frac{n-1}{4}.$$
Hence, $G$ is a strongly regular graph with parameters $\left(n, \frac{n-1}{2}, \frac{n-5}{4}, \frac{n-1}{4}\right)$.
%\end{proof}
\hfill$\blacksquare$
%\vspace{0.2cm}

\begin{lemma}\label{le:CRK2tlower-construction-Hadamard-1}
Let $m$ be a positive integer.
\begin{itemize}
\item[{\rm (i)}] If there exists a graphical Hadamard matrix of order $4m^2$, then
$$r'_2(K_{2,3m^2-1})= CR(K_{2,3m^2-1}, K_3)=4m^2.$$
\item[{\rm (ii)}] If there exists a Hadamard matrix of order $m$ $(m\geq 2)$, then
$$r'_2\Big(K_{2,\frac{3m^2}{4}-1}\Big)= CR\Big(K_{2,\frac{3m^2}{4}-1}, K_3\Big)=m^2.$$
\end{itemize}
\end{lemma}

\noindent {\bf Proof.}
(i) If $m=1$, then $r'_2(K_{2,2})= CR(K_{2,2}, K_3)=4$ by Corollary~\ref{cor:forest}.
Now assume that $m\geq 2$.
By Theorem~\ref{thm:CRKstUpper}, we have $r'_2(K_{2,3m^2-1})\leq CR(K_{2,3m^2-1}, K_3)\leq \left\lfloor\frac{4(3m^2-1)}{3}\right\rfloor+2=4m^2$.
Moreover, it is well-known that for $4m^2>4$, if there exists a graphical Hadamard matrix of order $4m^2$, then there exists an $SRG(4m^2-1, 2m^2, m^2, m^2)$ (see, for example, \cite[Section~4]{AbBH}).
Note that $\min\{2m^2-1-m^2, 2m^2-m^2\}=m^2-1= 4m^2-1-1-(3m^2-1)$.
By Corollary~\ref{co:CRK2tlower-construction-SRG}, we have $r'_2(K_{2,3m^2-1})\geq 4m^2-1+1= 4m^2.$
Therefore, we have $r'_2(K_{2,3m^2-1})= CR(K_{2,3m^2-1}, K_3)=4m^2.$

(ii) Since $m\geq 2$ and the order of a Hadamard matrix must be 1, 2, or a multiple of 4 (see, for example, \cite[Chapter~2]{Hor}), we have that $m$ is even.
Let $m=2m_1$.
Then $\frac{3m^2}{4}-1=3m_1^2-1$.
Goethals and Seidel~\cite[Theorem~4.4]{GoSe} proved that if there exists a Hadamard matrix of order $m$, then there exists a graphical Hadamard matrix of order $m^2$.
Thus there exists a graphical Hadamard matrix of order $m^2=4m_1^2$.
By (i), we have $r'_2(K_{2,3m_1^2-1})= CR(K_{2,3m_1^2-1}, K_3)=4m_1^2.$
Hence, $r'_2\Big(K_{2,\frac{3m^2}{4}-1}\Big)= CR\Big(K_{2,\frac{3m^2}{4}-1}, K_3\Big)=m^2.$
\hfill$\blacksquare$
%\vspace{0.2cm}

\begin{lemma}\label{le:CRK2tlower-construction-conference-1}
Let $m$ be a positive integer.
\begin{itemize}
\item[{\rm (i)}] If there exists a conference matrix of order $m+1$, then
$$r'_2\Big(K_{2,\frac{3m^2+1}{4}}\Big)= CR\Big(K_{2,\frac{3m^2+1}{4}}, K_3\Big)=m^2+2.$$
\item[{\rm (ii)}]If there exists a regular symmetric conference matrix of order $m^2+1$, then
$$r'_2\Big(K_{2,\frac{3m^2+1}{4}}\Big)= CR\Big(K_{2,\frac{3m^2+1}{4}}, K_3\Big)=m^2+2.$$
\end{itemize}
\end{lemma}

\noindent {\bf Proof.}
We first show that $\frac{3m^2+1}{4}$ is an integer under the assumptions.
Belevitch~\cite{Bel} proved that the order of a conference matrix must be even.
Thus if there exists a conference matrix of order $m+1$, then $m$ is odd, say $m=2k+1$.
Then $3m^2+1=3(2k+1)^2+1=12k^2+12k+4=4(3k^2+3k+1)$, so $\frac{3m^2+1}{4}$ is an integer.
Moreover, if $n$ is the order of a symmetric conference matrix, then $n\equiv 2 \pmod{4}$; see~\cite{Bel}.
Thus if there exists a regular symmetric conference matrix of order $m^2+1$, then $m^2 \equiv 1 \pmod{4}$, so $3m^2+1 \equiv 0 \pmod{4}$, i.e., $\frac{3m^2+1}{4}$ is an integer.

We now show that $r'_2\Big(K_{2,\frac{3m^2+1}{4}}\Big)= CR\Big(K_{2,\frac{3m^2+1}{4}}, K_3\Big)=m^2+2$.
For $m=1$, we have $r'_2(K_{2,1})=CR(K_{2,1}, K_3)=3=1^2+2$ by Proposition~\ref{prop:forest}.
Hence, we assume that $m>1$.
By Theorem~\ref{thm:CRKstUpper}, we have $r'_2\Big(K_{2,\frac{3m^2+1}{4}}\Big)\leq CR\Big(K_{2,\frac{3m^2+1}{4}}, K_3\Big)\leq \left\lfloor\frac{4}{3}\cdot\frac{3m^2+1}{4}\right\rfloor+2=m^2+2.$
For the lower bound,
it is known that if there exists a conference matrix of order $m+1$ (resp., a regular symmetric conference matrix of order $m^2+1$), then there exists an $SRG\left(m^2+1,\frac{m^2-m}{2}, \frac{m^2-2m-3}{4}, \frac{m^2-2m+1}{4}\right)$; see, for example, \cite[Section~8.2]{BrVaMa} (resp., \cite[Section~4.2]{GHKS}).
Since
$\frac{m^2-m}{2}-1-\frac{m^2-2m-3}{4}=\frac{m^2-1}{4}=m^2+1-1-\frac{3m^2+1}{4}$
and
$\frac{m^2-m}{2}-\frac{m^2-2m+1}{4}=\frac{m^2-1}{4}=m^2+1-1-\frac{3m^2+1}{4}$,
we have $r'_2\Big(K_{2,\frac{3m^2+1}{4}}\Big)\geq m^2+1+1= m^2+2$ by Corollary~\ref{co:CRK2tlower-construction-SRG}.
The result follows.
\hfill$\blacksquare$
\vspace{0.2cm}

Finally, we present a lower bound construction for $r'_2(K_{3,t})$.

\begin{lemma}\label{le:CRK3tlower}
For any positive integer $t$ with $t\equiv 0 \pmod{2}$, if there exists an $SRG\left(2t+1, t, \frac{t}{2}-1, \frac{t}{2}\right)$, then
$$r'_2(K_{3,t})= CR(K_{3,t}, K_3)=2t+2.$$
\end{lemma}

Before proving Lemma~\ref{le:CRK3tlower}, we establish the following auxiliary lemma.
Given a graph $H$ and any ordered triple $\mathbf{x}=(x_1, x_2, x_3)$ of three distinct vertices of $H$, define
$$L_{H}(\mathbf{x})\colonequals \Big|\left\{y\in V(H)\setminus \{x_1, x_2, x_3\}\colon\, \left(\mathbf{1}_{\{x_1y\in E(H)\}}, \mathbf{1}_{\{x_2y\in E(H)\}}, \mathbf{1}_{\{x_3y\in E(H)\}}\right)\in \mathcal{C}_3\right\}\Big|.$$

\begin{lemma}\label{le:CRK3tlower-SRG}
Let $H$ be an $SRG(n, k, \lambda, \mu)$.
Then
\begin{align*}
L_{H}(\mathbf{x})= &~n-3-2k+2\mu+(1+\lambda-\mu)\mathbf{1}_{\{x_1x_2\in E(H)\}}+(2+\lambda-\mu)\mathbf{1}_{\{x_2x_3\in E(H)\}} \\
~&~+\mathbf{1}_{\{x_1x_3\in E(H)\}}\left(1-\mathbf{1}_{\{x_1x_2\in E(H)\}}-\mathbf{1}_{\{x_2x_3\in E(H)\}}\right).
\end{align*}
\end{lemma}

\noindent {\bf Proof.}
For brevity, we write $\mathbf{1}_{i,j}$ for $\mathbf{1}_{\{x_ix_j\in E(H)\}}$, and $\mathbf{1}_{i,y}$ for $\mathbf{1}_{\{x_iy\in E(H)\}}$.
Then $$L_{H}(\mathbf{x})=\sum\nolimits_{y\in V(H)\setminus \{x_1, x_2, x_3\}}\mathbf{1}_{\left\{\left(\mathbf{1}_{1,y}, \mathbf{1}_{2,y}, \mathbf{1}_{3,y}\right)\in \mathcal{C}_3\right\}}.$$
By direct computation, we obtain
\begin{equation*}
1-\mathbf{1}_{2,y}-\mathbf{1}_{3,y}+\mathbf{1}_{2,y}\mathbf{1}_{3,y}+\mathbf{1}_{1,y}\mathbf{1}_{2,y} =
\left\{
   \begin{aligned}
    &1 & & \mbox{if $\left(\mathbf{1}_{1,y}, \mathbf{1}_{2,y}, \mathbf{1}_{3,y}\right)\in \mathcal{C}_3=\{000, 100, 110, 111\}$},\\
    &0 & & \mbox{if $\left(\mathbf{1}_{1,y}, \mathbf{1}_{2,y}, \mathbf{1}_{3,y}\right)\in \{0,1\}^3\setminus \mathcal{C}_3=\{001, 010, 011, 101\}$}.
   \end{aligned}
\right.
\end{equation*}
Hence, we have $\mathbf{1}_{\left\{\left(\mathbf{1}_{1,y}, \mathbf{1}_{2,y}, \mathbf{1}_{3,y}\right)\in \mathcal{C}_3\right\}}=1-\mathbf{1}_{2,y}-\mathbf{1}_{3,y}+\mathbf{1}_{2,y}\mathbf{1}_{3,y}+\mathbf{1}_{1,y}\mathbf{1}_{2,y}.$

Since $H$ is $k$-regular, we have
$$\sum\nolimits_{y\in V(H)\setminus \{x_1, x_2, x_3\}}\mathbf{1}_{2,y}=\left|N_H(x_2)\setminus \{x_1, x_2, x_3\}\right| =k-\mathbf{1}_{1,2}-\mathbf{1}_{2,3},$$
$$\sum\nolimits_{y\in V(H)\setminus \{x_1, x_2, x_3\}}\mathbf{1}_{3,y}=\left|N_H(x_3)\setminus \{x_1, x_2, x_3\}\right| =k-\mathbf{1}_{1,3}-\mathbf{1}_{2,3}.$$
Moreover, we have
$\left|N_H(x_2)\cap N_{H}(x_3)\right|=\lambda \mathbf{1}_{2,3}+\mu (1-\mathbf{1}_{2,3})= \mu + (\lambda-\mu)\mathbf{1}_{2,3}$ and $\mathbf{1}_{\{x_1\in N_H(x_2)\cap N_{H}(x_3)\}}=\mathbf{1}_{1,2}\mathbf{1}_{1,3}.$
Thus
$$\sum\nolimits_{y\in V(H)\setminus \{x_1, x_2, x_3\}}\mathbf{1}_{2,y}\mathbf{1}_{3,y}=\left|(N_H(x_2)\cap N_{H}(x_3))\setminus \{x_1, x_2, x_3\}\right| =\mu + (\lambda-\mu)\mathbf{1}_{2,3}-\mathbf{1}_{1,2}\mathbf{1}_{1,3}.$$
Similarly, we have
$$\sum\nolimits_{y\in V(H)\setminus \{x_1, x_2, x_3\}}\mathbf{1}_{1,y}\mathbf{1}_{2,y}=\left|(N_H(x_1)\cap N_{H}(x_2))\setminus \{x_1, x_2, x_3\}\right| =\mu + (\lambda-\mu)\mathbf{1}_{1,2}-\mathbf{1}_{1,3}\mathbf{1}_{2,3}.$$
Therefore, we have
\begin{align*}
L_{H}(\mathbf{x})= &~\sum\nolimits_{y\in V(H)\setminus \{x_1, x_2, x_3\}}\mathbf{1}_{\left\{\left(\mathbf{1}_{1,y}, \mathbf{1}_{2,y}, \mathbf{1}_{3,y}\right)\in \mathcal{C}_3\right\}} \\
= &~\sum\nolimits_{y\in V(H)\setminus \{x_1, x_2, x_3\}}\left(1-\mathbf{1}_{2,y}-\mathbf{1}_{3,y}+\mathbf{1}_{2,y}\mathbf{1}_{3,y}+\mathbf{1}_{1,y}\mathbf{1}_{2,y}\right) \\
= &~(n-3)-(k-\mathbf{1}_{1,2}-\mathbf{1}_{2,3})-(k-\mathbf{1}_{1,3}-\mathbf{1}_{2,3})+(\mu + (\lambda-\mu)\mathbf{1}_{2,3}-\mathbf{1}_{1,2}\mathbf{1}_{1,3}) \\
~ &~+(\mu + (\lambda-\mu)\mathbf{1}_{1,2}-\mathbf{1}_{1,3}\mathbf{1}_{2,3}) \\
= &~n-3-2k+2\mu+(1+\lambda-\mu)\mathbf{1}_{1,2}+(2+\lambda-\mu)\mathbf{1}_{2,3} +\mathbf{1}_{1,3}\left(1-\mathbf{1}_{1,2}-\mathbf{1}_{2,3}\right).
\end{align*}
The proof is complete.
\hfill$\blacksquare$
\vspace{0.2cm}

Now we present our proof of Lemma~\ref{le:CRK3tlower}.
\vspace{0.2cm}

\noindent {\bf Proof of Lemma~\ref{le:CRK3tlower}.}
By Theorem~\ref{thm:CRKstUpper}, we have $r'_2(K_{3,t})\leq CR(K_{3,t}, K_3)\leq 2t+2.$
Next, we show that if there exists an $SRG\left(2t+1, t, \frac{t}{2}-1, \frac{t}{2}\right)$, then $r'_2(K_{3,t})\geq 2t+2.$
Let $H$ be a strongly regular graph with parameters $(n, k, \lambda, \mu)=\left(2t+1, t, \frac{t}{2}-1, \frac{t}{2}\right)$.
We construct a 2-edge-coloring $\chi$ of $K_n$ by coloring all edges of $H$ with red and all edges of the complement $\overline{H}$ with blue.
Then for any pair $(x,y)$ of distinct vertices, we have $\mathbf{1}_{\scriptsize\{\mbox{$xy$ is red}\}}=\mathbf{1}_{\{xy\in E(H)\}}$.
Hence, for any ordered triple $\mathbf{x}=(x_1, x_2, x_3)$ of distinct vertices, we have $\left|Y_{\chi}(\mathbf{x})\right|=L_{H}(\mathbf{x})$.
Then, by Lemma~\ref{le:CRK3tlower-SRG}, we have
\begin{align*}
\left|Y_{\chi}(\mathbf{x})\right|= &~n-3-2k+2\mu+(1+\lambda-\mu)\mathbf{1}_{\{x_1x_2\in E(H)\}}+(2+\lambda-\mu)\mathbf{1}_{\{x_2x_3\in E(H)\}} \\
~&~+\mathbf{1}_{\{x_1x_3\in E(H)\}}\left(1-\mathbf{1}_{\{x_1x_2\in E(H)\}}-\mathbf{1}_{\{x_2x_3\in E(H)\}}\right) \\
= &~2t+1-3-2t+t+\mathbf{1}_{\{x_2x_3\in E(H)\}}+\mathbf{1}_{\{x_1x_3\in E(H)\}}\left(1-\mathbf{1}_{\{x_1x_2\in E(H)\}}-\mathbf{1}_{\{x_2x_3\in E(H)\}}\right) \\
= &~t-2+\mathbf{1}_{\{x_2x_3\in E(H)\}}+\mathbf{1}_{\{x_1x_3\in E(H)\}}\left(1-\mathbf{1}_{\{x_1x_2\in E(H)\}}-\mathbf{1}_{\{x_2x_3\in E(H)\}}\right).
\end{align*}
If $\mathbf{1}_{\{x_1x_3\in E(H)\}}=0$, then $$\mathbf{1}_{\{x_2x_3\in E(H)\}}+\mathbf{1}_{\{x_1x_3\in E(H)\}}\left(1-\mathbf{1}_{\{x_1x_2\in E(H)\}}-\mathbf{1}_{\{x_2x_3\in E(H)\}}\right)=\mathbf{1}_{\{x_2x_3\in E(H)\}}\leq 1;$$
if $\mathbf{1}_{\{x_1x_3\in E(H)\}}=1$, then $$\mathbf{1}_{\{x_2x_3\in E(H)\}}+\mathbf{1}_{\{x_1x_3\in E(H)\}}\left(1-\mathbf{1}_{\{x_1x_2\in E(H)\}}-\mathbf{1}_{\{x_2x_3\in E(H)\}}\right)=1-\mathbf{1}_{\{x_1x_2\in E(H)\}}\leq 1.$$
Thus $\left|Y_{\chi}(\mathbf{x})\right|\leq t-2+1=t-1.$
By Lemma~\ref{le:Y}, there is no orderable $K_{3,t}$.
Hence, we have $r'_2(K_{3,t})\geq 2t+2.$
The proof of Lemma~\ref{le:CRK3tlower} is complete.
\hfill$\blacksquare$
%\vspace{0.2cm}

\subsection{Proof of $r'_2(K_{s,t})=CR(K_{s,t}, K_3)$ for large $t$}
\label{subsec:pf_complete_bipar_equal}

In this subsection, we show that $r'_2(K_{s,t})=CR(K_{s,t}, K_3)$ for sufficiently large $t$.

\begin{theorem}\label{thm:CRKst-equal}
For any fixed integer $s \geq 1$ and sufficiently large $t$, we have $r'_2(K_{s,t})= CR(K_{s,t}, K_3)$.
\end{theorem}

The proof proceeds by showing that if $K_n$ admits a Gallai-coloring that uses at least three colors and contains no orderable $K_{s,t}$, then $n<r'_2(K_{s,t})$.
One of the main ingredients is a weighted argument related to $Y_{\chi}(\mathbf{x})$.
Recall that $Y_{\chi}(\mathbf{x})$ is defined in the setting of 2-edge-colorings.
However, a Gallai-coloring $\chi$ of $K_n$ can use any number of colors.
To adapt the argument involving $Y_{\chi}(\mathbf{x})$ to the Gallai-coloring setting,
we consider a Gallai partition $V_1, V_2, \ldots, V_m$ ($m\geq 2$) of $V(K_n)$.
Note that the edges between the parts use at most two colors (say red and blue).
Thus we can define a {\it reduced graph} $R$ of the partition with a 2-edge-coloring $\chi'$ as follows:
the graph $R$ is an $m$-vertex complete graph with $V(R)=[m]$, and for each edge $ij$ we have $\chi'(ij)=\chi(V_i, V_j)$.
Now we can apply an argument related to $Y_{\chi'}(\mathbf{x})$ to $R$.
To return to the original Gallai-coloring of $K_n$, we actually need a weighted version of the argument.

By our earlier definition, for an ordered $s$-tuple $\mathbf{x}=(x_1, \ldots, x_s)$ of $s$ distinct vertices of $R$,
$$Y_{\chi'}(\mathbf{x})= \left\{y\in V(R)\setminus \{x_1, \ldots, x_s\}\colon\, \left(\mathbf{1}_{\scriptsize\{\mbox{$\chi'(x_1y)$ is red}\}}, \ldots, \mathbf{1}_{\scriptsize\{\mbox{$\chi'(x_sy)$ is red}\}}\right)\in \mathcal{C}_s\right\}.$$
Let $M$ be a positive integer.
For each vertex $i\in V(R)=[m]$, we assign a weight $w_i$ ($0\leq w_i\leq M$) to $i$.
Define $W(R)\colonequals \sum\nolimits_{i\in [m]}w_i$
and $L_{\chi',w}(\mathbf{x})\colonequals \sum\nolimits_{y\in Y_{\chi'}(\mathbf{x})}w_y$.
We shall use the following two lemmas.

\begin{lemma}\label{le:CRKst-equal-weight-1}
Let $\chi'$ be a 2-edge-coloring of a complete graph $R$ with $V(R)=[m]$ and $m\geq s+1$.
Fix a vertex $y\in V(R)$.
We choose $\mathbf{x}=(x_1, \ldots, x_s)$ from all ordered $s$-tuples of $s$ distinct vertices of $V(R)\setminus\{y\}$ uniformly at random.
Let $E$ be the event that $\left(\mathbf{1}_{\scriptsize\{\mbox{\rm $\chi'(x_1y)$ is red}\}}, \ldots, \mathbf{1}_{\scriptsize\{\mbox{\rm $\chi'(x_sy)$ is red}\}}\right)\in \mathcal{C}_s.$
Then $\pr\left[E\right]\geq \frac{s+1}{2^s}-\frac{1}{m-1}{s\choose 2}.$
\end{lemma}

\noindent {\bf Proof.}
Assume that the colors of $\chi'$ lie in $\{\mbox{red, blue}\}$.
Let $d_r(y)$ be the number of red edges incident with $y$, and let $p\colonequals \frac{d_r(y)}{m-1}$.
We first choose vertices $x_1, \ldots, x_s$ independently and uniformly from $V(R)\setminus\{y\}$, with replacement.
Note that the variables $\mathbf{1}_{\scriptsize\{\mbox{$\chi'(x_1y)$ is red}\}}, \ldots, \mathbf{1}_{\scriptsize\{\mbox{$\chi'(x_sy)$ is red}\}}$ are independent Bernoulli random variables with $\pr\left[\mathbf{1}_{\scriptsize\{\mbox{$\chi'(x_iy)$ is red}\}}=1\right]=p$ for each $i\in [s]$.
Let $A$ be the event that $\left(\mathbf{1}_{\scriptsize\{\mbox{$\chi'(x_1y)$ is red}\}}, \ldots, \mathbf{1}_{\scriptsize\{\mbox{$\chi'(x_sy)$ is red}\}}\right)\in \mathcal{C}_s$.
Note that for each $0\leq j\leq s$, we have $\pr\left[\left(\mathbf{1}_{\scriptsize\{\mbox{$\chi'(x_1y)$ is red}\}}, \ldots, \mathbf{1}_{\scriptsize\{\mbox{$\chi'(x_sy)$ is red}\}}\right)=1^{j}0^{s-j}\right]=p^{j}(1-p)^{s-j}.$
Thus $\pr[A]=\sum\nolimits_{j=0}^{s}p^{j}(1-p)^{s-j}.$

\begin{claim}\label{cl:CRKst-equal-weight-1}
$\pr[A]\geq \frac{s+1}{2^s}$.
\end{claim}

\begin{proof}
For $p\in [0,1]$, let $\Psi_s(p)\colonequals \sum\nolimits_{j=0}^{s}p^{j}(1-p)^{s-j}$.
We shall show that $\Psi_s(p)\geq \frac{s+1}{2^s}$.
Since the function $\Psi_s(p)$ is symmetric in the sense $\Psi_s(p)=\Psi_s(1-p)$, it suffices to consider $0\leq p\leq \frac{1}{2}$.
Let $x=\frac{p}{1-p}$, so $0\leq x\leq 1$, $p=\frac{x}{1+x}$ and $1-p=\frac{1}{1+x}.$
Then $\Psi_s(p)=\sum\nolimits_{j=0}^{s}\left(\frac{x}{1+x}\right)^{j}\left(\frac{1}{1+x}\right)^{s-j}=\frac{1+x+\cdots+x^s}{(1+x)^s}$.
Let $P(x)\colonequals 1+x+\cdots+x^s$ and $f(x)\colonequals \frac{P(x)}{(1+x)^s}= \Psi_s(p)$.
Then
$$f'(x)=\frac{P'(x)(1+x)^s-s(1+x)^{s-1}P(x)}{(1+x)^{2s}}=\frac{(1+x)P'(x)-sP(x)}{(1+x)^{s+1}},$$
and
\begin{align*}
  (1+x)P'(x)-sP(x) = &~(1+x)(1+2x+\cdots+sx^{s-1})-s(1+x+\cdots+x^s) \\
  = &~(1+2x+\cdots+sx^{s-1}) + (x+2x^2+\cdots+sx^{s})-s(1+x+\cdots+x^s) \\
  = &~(1-s)\cdot 1 + (2+1-s)x +\cdots + (s+s-1+s)x^{s-1} + (sx^{s}-sx^{s}) \\
  = &~\sum\nolimits_{j=0}^{s-1}(2j+1-s)x^j = \sum\nolimits_{j\in \{0, 1, \ldots, s-1\}\setminus \{\frac{s-1}{2}\}}(2j+1-s)x^j \\
  = &~\sum\nolimits_{0\leq j<\frac{s-1}{2}}(2j+1-s)x^j + \sum\nolimits_{\frac{s-1}{2}<j\leq s-1}(2j+1-s)x^j \\
  = &~\sum\nolimits_{0\leq j<\frac{s-1}{2}}(2j+1-s)x^j + \sum\nolimits_{0\leq i<\frac{s-1}{2}}(2(s-1-i)+1-s)x^{s-1-i} \\
  = &~\sum\nolimits_{0\leq j<\frac{s-1}{2}}(2j+1-s)x^j + \sum\nolimits_{0\leq i<\frac{s-1}{2}}(-(2i+1-s))x^{s-1-i} \\
  = &~\sum\nolimits_{0\leq j<\frac{s-1}{2}}(2j+1-s)(x^j-x^{s-1-j}).
\end{align*}
Note that for $0\leq j<\frac{s-1}{2}$, we have $2j+1-s<0$ and $j<s-1-j$, so $x^j\geq x^{s-1-j}$ for $0\leq x\leq 1$.
Then $(1+x)P'(x)-sP(x)=\sum\nolimits_{0\leq j<\frac{s-1}{2}}(2j+1-s)(x^j-x^{s-1-j})\leq 0$, so $f'(x)\leq 0$.
Thus $f(x)$ is nonincreasing.
Therefore, we have $\pr[A]=\Psi_s(p)=f(x)\geq f(1)= \frac{s+1}{2^s}$.
\end{proof}

Recall that the vertices $x_1, \ldots, x_s$ are chosen with replacement.
We actually need $x_1, \ldots, x_s$ to be distinct.
Let $D$ be the event that the vertices $x_1, \ldots, x_s$ are distinct.
Note that the conditional distribution of the with-replacement sample given $D$ is uniform over all ordered $s$-tuples of $s$ distinct vertices of $V(R)\setminus\{y\}$, so $\pr[E]=\pr[A~|~D].$
For every pair $(i,j)$ with $1\leq i<j\leq s$, we have $\pr[x_i=x_j]=\frac{m-1}{(m-1)^2}=\frac{1}{m-1}$.
Applying the union bound, we have $\pr\Big[\overline{D}\Big]\leq \sum\nolimits_{1\leq i<j\leq s}\pr[x_i=x_j] = \frac{1}{m-1}{s\choose 2}.$
Combining with Claim~\ref{cl:CRKst-equal-weight-1}, we have
$$\pr[E]=\pr[A~|~D]=\frac{\pr[A\wedge D]}{\pr[D]}\geq \pr[A\wedge D]\geq \pr[A]-\pr\Big[\overline{D}\Big]\geq \frac{s+1}{2^s}-\frac{1}{m-1}{s\choose 2}.$$
The proof of Lemma~\ref{le:CRKst-equal-weight-1} is complete.
\hfill$\blacksquare$
%\vspace{0.2cm}

\begin{lemma}\label{le:CRKst-equal-weight-2}
Let $\chi'$ be a 2-edge-coloring of a complete graph $R$ with $V(R)=[m]$.
For each vertex $i\in V(R)=[m]$, we assign a weight $w_i$ $(0\leq w_i\leq M)$ to $i$.
If $m\geq s+1\geq 3$, then there exists an ordered $s$-tuple $\mathbf{x}$ of $s$ distinct vertices satisfying $L_{\chi',w}(\mathbf{x})\geq \frac{s+1}{2^s}W(R)-M\left(\frac{s(s+1)}{2^s}+{s\choose 2}\right).$
\end{lemma}

\noindent {\bf Proof.}
We choose an ordered $s$-tuple $\mathbf{x}=(x_1, \ldots, x_s)$ uniformly at random from the set of all ordered $s$-tuples of $s$ distinct vertices of $V(R)$.
For each vertex $y\in V(R)$, define an indicator variable as follows:
$$I_y \colonequals
\left\{
   \begin{aligned}
    &1 & & \mbox{if $y\notin \{x_1, \ldots, x_s\}$ and $\left(\mathbf{1}_{\scriptsize\{\mbox{$\chi'(x_1y)$ is red}\}}, \ldots, \mathbf{1}_{\scriptsize\{\mbox{$\chi'(x_sy)$ is red}\}}\right)\in \mathcal{C}_s$},\\
    &0 & & \mbox{otherwise}.
   \end{aligned}
   \right.$$
Then
\begin{equation}\label{eq:CRKst-equal-weight-2-1}
L_{\chi',w}(\mathbf{x})=\sum\nolimits_{y\in Y_{\chi'}(\mathbf{x})}w_y =\sum\nolimits_{y\in V(R)}w_yI_y.
\end{equation}

Fix a vertex $y\in V(R)$.
The probability that none of $x_1, \ldots, x_s$ is $y$ is $\pr[y\notin \{x_1, \ldots, x_s\}]=\frac{(m-1)_s}{(m)_s}=\frac{m-s}{m}.$
Conditional on $y\notin \{x_1, \ldots, x_s\}$, $\mathbf{x}$ is uniformly distributed over all ordered $s$-tuples of $s$ distinct vertices of $V(R)\setminus \{y\}$.
Thus by Lemma~\ref{le:CRKst-equal-weight-1}, we have
\begin{equation}\label{eq:CRKst-equal-weight-2-2}
\pr[I_y=1]\geq \frac{m-s}{m}\left(\frac{s+1}{2^s}-\frac{1}{m-1}{s\choose 2}\right).
\end{equation}

Since $0\leq w_i\leq M$ for all $i\in [m]$, we have $W(R)= \sum\nolimits_{i\in [m]}w_i\leq mM.$
Combining with Equality~(\ref{eq:CRKst-equal-weight-2-1}), Inequality~(\ref{eq:CRKst-equal-weight-2-2}) and using linearity of expectation, we have
\begin{align*}
\mathds{E}\left[L_{\chi',w}(\mathbf{x})\right]= &~\sum\nolimits_{y\in V(R)}w_y\mathds{E}[I_y] = \sum\nolimits_{y\in V(R)}w_y\pr[I_y=1] \\
\geq &~\sum\nolimits_{y\in V(R)}w_y\frac{m-s}{m}\left(\frac{s+1}{2^s}-\frac{1}{m-1}{s\choose 2}\right) = W(R)\frac{m-s}{m}\left(\frac{s+1}{2^s}-\frac{1}{m-1}{s\choose 2}\right) \\
= &~\frac{s+1}{2^s}W(R)-\frac{W(R)}{m}\frac{s(s+1)}{2^s} -\frac{W(R)}{m}\frac{m-s}{m-1}{s\choose 2} \\
\geq &~\frac{s+1}{2^s}W(R)-\frac{mM}{m}\frac{s(s+1)}{2^s} -\frac{mM}{m}\frac{m-s}{m-1}{s\choose 2} \\
\geq &~\frac{s+1}{2^s}W(R)-M\frac{s(s+1)}{2^s} -M{s\choose 2} = \frac{s+1}{2^s}W(R)-M\left(\frac{s(s+1)}{2^s}+{s\choose 2}\right).
\end{align*}
Therefore, there exists an $\mathbf{x}$ satisfying $L_{\chi',w}(\mathbf{x})\geq \mathds{E}\left[L_{\chi',w}(\mathbf{x})\right] \geq \frac{s+1}{2^s}W(R)-M\left(\frac{s(s+1)}{2^s}+{s\choose 2}\right).$
\hfill$\blacksquare$
\vspace{0.2cm}

We next state and prove two additional lemmas.

\begin{lemma}\label{le:CRKst-equal-1}
For any fixed integer $s\geq 3$ and sufficiently large $N$,
suppose that a Gallai-colored $K_N$ admits a Gallai partition $V_1, V_2, \ldots, V_m$ with $2\leq \max\nolimits_{i\in [m]}|V_i|\leq s-1$.
Then it contains an orderable $K_{s, \ell}$ with $\ell\geq \gamma_sN-O_s(1)$, where $\gamma_3=\frac{3}{4}$ and $\gamma_s=\frac{s-1}{2^{s-1}}$ for $s\geq 4$.
\end{lemma}

\noindent {\bf Proof.}
Let $\chi$ be a Gallai-coloring of $K_N$ that admits a Gallai partition $V_1, V_2, \ldots, V_m$ with $2\leq \max\nolimits_{i\in [m]}|V_i|\leq s-1$.
Without loss of generality, we may assume that $|V_1|=\max\nolimits_{i\in [m]}|V_i|$ and the colors between the parts belong to $\{\mbox{red, blue}\}$.
Then $2\leq |V_1|\leq s-1$.
Let $R$ be the reduced graph with $V(R)=[m]$ and $\chi'$ be the 2-edge-coloring of $R$.
For each vertex $i\in V(R)=[m]$, we assign a weight $w_i=|V_i|$ to $i$.
Since $\max\nolimits_{i\in [m]}|V_i|\leq s-1$, we have $0\leq w_i \leq s-1$ for all $i\in [m]$.
Let
\begin{align*}
  A &\colonequals \{i\in [m]\setminus \{1\}\colon\, \mbox{$\chi(V_i, V_1)$ is red}\},\hspace{-1.5cm} & &W_A\colonequals \sum\nolimits_{i\in A} w_i = \sum\nolimits_{i\in A}|V_i|, \\
  B &\colonequals \{i\in [m]\setminus \{1\}\colon\, \mbox{$\chi(V_i, V_1)$ is blue}\},\hspace{-1.5cm} & &W_B\colonequals \sum\nolimits_{i\in B} w_i = \sum\nolimits_{i\in B}|V_i|.
\end{align*}
Then $W(R)=\sum\nolimits_{i\in [m]}w_i = W_A+W_B+|V_1|=N.$
We divide the rest of the proof into two cases.
\vspace{0.2cm}

{\bf Case~1.} $s\geq 4$.
\vspace{0.2cm}

Without loss of generality, we may assume that $W_A\geq W_B$.
Then $W_A\geq \frac{N-|V_1|}{2}\geq \frac{N-(s-1)}{2}$.
Since $\max\nolimits_{i\in [m]}|V_i|\leq s-1$, we further have $|A|\geq \frac{W_A}{s-1}\geq \frac{N-(s-1)}{2(s-1)}\geq s-1$ for sufficiently large $N$.
Let $R_A$ be the subgraph of $R$ induced by $A$, and let $\chi'_A$ be the 2-edge-coloring of $R_A$ inherited from $\chi'$.
Then $R_A$ is a complete graph of order $|A|\geq s-1=(s-2)+1.$
Applying Lemma~\ref{le:CRKst-equal-weight-2} to $R_A$, we obtain an ordered $(s-2)$-tuple $\mathbf{x}=(x_1, \ldots, x_{s-2})$ of $s-2$ distinct vertices of $R_A$ satisfying \begin{align}\label{eq:CRKst-equal-1-1}
\sum\nolimits_{y\in Y_{\chi'_A}(\mathbf{x})}|V_y|= &~\sum\nolimits_{y\in Y_{\chi'_A}(\mathbf{x})}w_y=L_{\chi'_A,w}(\mathbf{x}) \nonumber\\
\geq &~\frac{(s-2)+1}{2^{s-2}}W_A-(s-1)\left(\frac{(s-2)((s-2)+1)}{2^{s-2}}+{s-2\choose 2}\right) \nonumber\\
\geq &~\frac{s-1}{2^{s-2}}\frac{N-(s-1)}{2}-O_s(1) = \frac{s-1}{2^{s-1}}N -O_s(1).
\end{align}
Let $Z\colonequals \bigcup_{y\in Y_{\chi'_A}(\mathbf{x})}V_y.$
For each $i\in [s-2]$, we choose one vertex $v_i$ from $V_{x_i}$.
Then every edge between $Z$ and $\{v_1, \ldots, v_{s-2}\}$ is either red or blue,
and for every $z\in Z$, the vector $\left(\mathbf{1}_{\scriptsize\{\mbox{$\chi(v_1z)$ is red}\}}, \ldots, \mathbf{1}_{\scriptsize\{\mbox{$\chi(v_{s-2}z)$ is red}\}}\right)$ belongs to $\mathcal{C}_{s-2}$.
By the implication (iv)$\Rightarrow$(i) in Lemma~\ref{le:lonesum}, those edges form an orderable $K_{s-2, \ell}$, where $\ell=|Z|\geq \frac{s-1}{2^{s-1}}N -O_s(1)$ by Inequality~(\ref{eq:CRKst-equal-1-1}).
Note that every edge between $V_1$ and $Z$ is red.
Then we can obtain an orderable $K_{s,\ell}$ from this $K_{s-2,\ell}$ by adding two vertices of $V_1$ as the first two vertices.
\vspace{0.2cm}

{\bf Case~2.} $s=3$.
\vspace{0.2cm}

In this case, we have $|V_1|=2$.
For each $i\in [m]\setminus \{1\}$, let
\begin{align*}
  a_i &\colonequals \sum\nolimits_{\scriptsize\mbox{$j\in A$, $j\neq i$, $\chi'(ij)$ is blue}}w_j=\sum\nolimits_{\scriptsize\mbox{$j\in A$, $j\neq i$,  $\chi'(ij)$ is blue}}|V_j|, \\
  b_i &\colonequals \sum\nolimits_{\scriptsize\mbox{$j\in B$, $j\neq i$, $\chi'(ij)$ is red}}w_j=\sum\nolimits_{\scriptsize\mbox{$j\in B$, $j\neq i$, $\chi'(ij)$ is red}}|V_j|.
\end{align*}

\begin{claim}\label{cl:CRKst-equal-1-1}
There exists an $i\in [m]\setminus \{1\}$ with $\min\{a_i, b_i\}\leq \frac{N-2}{4}$.
\end{claim}

\begin{proof}
If $A=\emptyset$ or $B=\emptyset$, then the claim holds trivially, so we may assume that $A\neq \empty$ and $B\neq \emptyset$.
Since $W_A+W_B=N-|V_1|=N-2$, we have $\frac{1}{N-2}W_AW_B=\frac{1}{N-2}W_A(N-2-W_A)\leq \frac{1}{N-2}\frac{(N-2)^2}{4}=\frac{N-2}{4}$.
Hence, it suffices to show that there exists an $i\in [m]\setminus \{1\}$ with $\min\{a_i, b_i\}\leq \frac{W_AW_B}{N-2}$.
Suppose for a contradiction that for all $i\in [m]\setminus \{1\}$, we have $a_i> \frac{W_AW_B}{N-2}$ and $b_i> \frac{W_AW_B}{N-2}$.
We proceed by double-counting $X\colonequals \sum\nolimits_{\scriptsize\mbox{$x\in A$, $y\in B$, $\chi'(xy)$ is red}}w_xw_y$.
Since $b_i> \frac{W_AW_B}{N-2}$ for all $i\in [m]\setminus \{1\}$, we have
\begin{align*}
X= &~\sum\nolimits_{x\in A}\sum\nolimits_{\scriptsize\mbox{$y\in B$, $\chi'(xy)$ is red}}w_xw_y =\sum\nolimits_{x\in A}w_x\sum\nolimits_{\scriptsize\mbox{$y\in B$, $\chi'(xy)$ is red}}w_y \\
= &~\sum\nolimits_{x\in A}w_x b_x > \sum\nolimits_{x\in A}w_x \frac{W_AW_B}{N-2} =W_A\frac{W_AW_B}{N-2}.
\end{align*}
Since $a_i> \frac{W_AW_B}{N-2}$ for all $i\in [m]\setminus \{1\}$, we have
\begin{align*}
X= &~\sum\nolimits_{y\in B}\sum\nolimits_{\scriptsize\mbox{$x\in A$, $\chi'(xy)$ is red}}w_xw_y =\sum\nolimits_{y\in B}w_y\sum\nolimits_{\scriptsize\mbox{$x\in A$, $\chi'(xy)$ is red}}w_x \\
= &~\sum\nolimits_{y\in B}w_y\left(W_A-\sum\nolimits_{\scriptsize\mbox{$x\in A$, $\chi'(xy)$ is blue}}w_x\right) = \sum\nolimits_{y\in B}w_y\left(W_A-a_y\right) \\
< &~\sum\nolimits_{y\in B}w_y\left(W_A-\frac{W_AW_B}{N-2}\right) = W_B\left(W_A-\frac{W_AW_B}{N-2}\right).
\end{align*}
Then $W_A\frac{W_AW_B}{N-2}<X<W_B\left(W_A-\frac{W_AW_B}{N-2}\right).$
But $W_A\frac{W_AW_B}{N-2}-W_B\left(W_A-\frac{W_AW_B}{N-2}\right)=(W_A+W_B)\frac{W_AW_B}{N-2}-W_AW_B=(N-2)\frac{W_AW_B}{N-2}-W_AW_B=0,$ a contradiction.
This completes the proof of Claim~\ref{cl:CRKst-equal-1-1}.
\end{proof}

Without loss of generality, we may assume that for $i=2$ we have $\min\{a_i, b_i\}\leq \frac{N-2}{4}$ by Claim~\ref{cl:CRKst-equal-1-1}.
By symmetry, we further assume that $a_2\leq b_2$, so $a_2\leq \frac{N-2}{4}.$
Let $V_1=\{z_1, z_2\}$, $z_3\in V_2$,
\begin{align*}
Y'= &~\bigcup\nolimits_{\scriptsize\mbox{$i\in A\cap ([m]\setminus \{1,2\})$, $\chi(V_2, V_i)$ is red}}V_i, \\
Y''= &~\bigcup\nolimits_{\scriptsize\mbox{$i\in B\cap ([m]\setminus \{1,2\})$, $\chi(V_2, V_i)$ is red}}V_i, \\
Y'''= &~\bigcup\nolimits_{\scriptsize\mbox{$i\in B\cap ([m]\setminus \{1,2\})$, $\chi(V_2, V_i)$ is blue}}V_i.
\end{align*}
Then $|Y'\cup Y''\cup Y'''|=N-|V_1|-|V_2|-a_2\geq N-2-2-\frac{N-2}{4}=\frac{3}{4}N-\frac{7}{2}$.
Note that the edges between $\{z_1, z_2, z_3\}$ and $Y'\cup Y''\cup Y'''$ form an orderable $K_{3, \ell}$ with $\ell\geq \frac{3}{4}N-\frac{7}{2}$ under the ordering $(Y', Y''', z_3, Y'', z_1, z_2)$, with arbitrary ordering within each of $Y', Y'', Y'''$.
The proof of Lemma~\ref{le:CRKst-equal-1} is complete.
\hfill$\blacksquare$
\vspace{0.2cm}

For integers $s\geq 1$ and $t\geq 1$, let $M_s(t)$ be the largest integer $n$ for which there exists a Gallai-coloring of $K_n$ that uses at least three colors and contains no orderable copy of $K_{s,t}$.
If no such Gallai-coloring exists, define $M_s(t)\colonequals 0.$

\begin{lemma}\label{le:CRKst-equal-2}
For any fixed integer $s\geq 2$ and sufficiently large $t$, we have $M_s(t)\leq \theta_s t +O_s(1)$, where $\theta_2=\frac{5}{4}$, $\theta_3=\frac{3}{2}$ and $\theta_s=\frac{2^{s-1}}{s-1}$ for $s\geq 4$.
\end{lemma}

\noindent {\bf Proof.}
Let $G$ be a Gallai-colored $K_n$ using at least three colors.
Assume that $G$ contains no orderable $K_{s,t}$, and we shall show that $n\leq \theta_s t +O_s(1).$
We recursively construct triples $(X, Y, \tau)$, where $X$ is a color-homogeneous set, $Y=V(G)\setminus X$, and $\tau$ is a positive integer satisfying $\tau +|Y|=t$.

Initially, $X=V(G)$, $Y=V(G)\setminus X=\emptyset$ and $\tau=t$; note that $V(G)$ is a color-homogeneous set, which holds vacuously since no vertex lies outside $V(G)$.
Take a Gallai partition $V_1, V_2, \ldots, V_m$ of $V(G[X])$.
Note that each part $V_i$ is a color-homogeneous set.
If there exists a part $V_i$ with $|V_i|\geq s$, then we must have $\tau-|X\setminus V_i|>0$
(otherwise if $|X\setminus V_i|\geq \tau$, then $|(X\setminus V_i)\cup Y|\geq \tau + |Y|=t$, and thus the edges between $V_i$ and $(X\setminus V_i)\cup Y$ form an orderable $K_{s,t}$ by Fact~\ref{fa:homo}~(i), a contradiction).
Thus we can replace
$$Y\leftarrow Y\cup (X\setminus V_i), ~~~~~~ \tau\leftarrow \tau-|X\setminus V_i|, ~~~~~~ X\leftarrow V_i.$$
We repeat this procedure until all parts in a Gallai partition of $V(G[X])$ have size at most $s-1$.
Note that in each step, the new set $X=V_i$ is still color-homogeneous in the original graph $G$.
Indeed, for every vertex $y$ newly added to $Y$, all edges between $y$ and $V_i$ are of the same color by the Gallai partition property;
for every old vertex $y$ of $Y$, all edges between $y$ and the old $X$ (and thus new $X=V_i$) are of the same color since the old $X$ was color-homogeneous.

Since the Gallai partition has at least two nonempty parts, the replacement strictly decreases $|X|$, and thus the procedure terminates.
If we make a more careful distinction, then when the procedure terminates, exactly one of the following occurs:
\begin{itemize}
\item[] (Type I) All parts in a Gallai partition of $V(G[X])$ have size at most $s-1$, and at least one part has size at least 2.
\item[] (Type II) All parts in a Gallai partition of $V(G[X])$ have size 1.
\end{itemize}
We divide the rest of the proof into two cases.
\vspace{0.2cm}

{\bf Case~1.} Type I occurs.
\vspace{0.2cm}

Since $\tau +|Y|=t$ and $\tau$ is positive, we have $|Y|<t$.
If $|X|=O_s(1)$, then $n=|X|+|Y|\leq t+O_s(1)\leq \theta_s t +O_s(1)$, and we are done.
Hence, we may assume that $|X|$ is sufficiently large.
Moreover, Type I can occur only when $s\geq 3$.
By Lemma~\ref{le:CRKst-equal-1}, $G[X]$ contains an orderable $K_{s, \ell}$ with $\ell\geq \gamma_s|X|-O_s(1)$, where $\gamma_3=\frac{3}{4}$ and $\gamma_s=\frac{s-1}{2^{s-1}}$ for $s\geq 4$.
By Fact~\ref{fa:homo}~(ii), $G$ contains an orderable $K_{s,\ell+|Y|}$ since $X$ is color-homogeneous.
Since $G$ contains no orderable $K_{s,t}$, we have $\gamma_s|X|-O_s(1)+|Y|\leq \ell+|Y|\leq t-1$.
Thus $|X|\leq \frac{1}{\gamma_s}\left(t-|Y|+O_s(1)\right)\leq \theta_s\left(t-|Y|+O_s(1)\right)$ and
$$n=|X|+|Y|\leq \theta_s\left(t-|Y|+O_s(1)\right)+|Y|\leq \theta_s t +O_s(1).$$
%\vspace{0.2cm}

{\bf Case~2.} Type II occurs.
\vspace{0.2cm}

In this case, the edge-coloring of $G[X]$ uses at most two colors.
Since the original edge-coloring of $G$ uses at least three colors, we must have $Y=V(G)\setminus X\neq \emptyset$.
By Fact~\ref{fa:homo}~(ii) and since $G$ contains no orderable $K_{s,t}$, $G[X]$ contains no orderable $K_{s, t-|Y|}$ or $K_{s-1, t}$.
Combining with Theorem~\ref{thm:CRKstUpper} and $r'_2(K_{1,t})=t+1$, we have
\begin{equation}\label{eq:CRKst-equal-2-1}
|X|\leq r'_2(K_{s,t-|Y|})-1\leq c_s(t-|Y|)+C_s,
\end{equation}
\begin{equation}\label{eq:CRKst-equal-2-2}
|X|\leq r'_2(K_{s-1,t})-1\leq c_{s-1}t+C'_s,
\end{equation}
where $c_s= \frac{2^s}{s+1}$, and $C_s$ and $C'_s$ are some fixed constants.
Notice that if $t-|Y|<s$, Theorem~\ref{thm:CRKstUpper} cannot be applied directly to $K_{s, t-|Y|}$. Nevertheless, Inequality~(\ref{eq:CRKst-equal-2-1}) still holds in this case, since we have $r'_2(K_{s,t-|Y|})\leq r'_2(K_{s,s})=O_s(1)$ by Theorem~\ref{thm:CRKstUpper}.

If $|Y|\leq s-1$, then by Inequality~(\ref{eq:CRKst-equal-2-2}), we have
\begin{equation}\label{eq:CRKst-equal-2-3}
n=|Y|+|X|\leq s-1+ c_{s-1}t+C'_s =c_{s-1}t+O_s(1).
\end{equation}
If $|Y|\geq s$, then to avoid an orderable $K_{s,t}$, we have $|X|\leq t-1$ by Fact~\ref{fa:homo}~(i).
Combining with Inequality~(\ref{eq:CRKst-equal-2-1}), we have
$$n=|Y|+|X|\leq |Y|+\min\{t-1, c_s(t-|Y|)+C_s\}.$$
When $t-1\leq c_s(t-|Y|)+C_s$, we have $|Y|\leq \left(1-\frac{1}{c_s}\right)t+\frac{C_s+1}{c_s}$, so
$n\leq |Y|+t-1\leq \left(1-\frac{1}{c_s}\right)t+\frac{C_s+1}{c_s}+t-1=\left(2-\frac{1}{c_s}\right)t+O_s(1)$;
when $t-1\geq c_s(t-|Y|)+C_s$, we have $|Y|\geq \left(1-\frac{1}{c_s}\right)t+\frac{C_s+1}{c_s}$, so
$n\leq |Y|+c_s(t-|Y|)+C_s=c_s t -(c_s-1)|Y|+C_s\leq c_s t -(c_s-1)\left(\left(1-\frac{1}{c_s}\right)t+\frac{C_s+1}{c_s}\right)+C_s=\left(2-\frac{1}{c_s}\right)t+O_s(1).$
Hence, in the case $|Y|\geq s$, we have
\begin{equation}\label{eq:CRKst-equal-2-4}
n\leq \left(2-\frac{1}{c_s}\right)t+O_s(1).
\end{equation}

By Inequalities~(\ref{eq:CRKst-equal-2-3}) and (\ref{eq:CRKst-equal-2-4}), we have
$$n\leq \max \left\{c_{s-1}, 2-\frac{1}{c_s}\right\}\cdot t+O_s(1).$$
For $s=2$, we have $\max \left\{c_{s-1}, 2-\frac{1}{c_s}\right\}=\max \left\{1, 2-\frac{3}{4}\right\}=\frac{5}{4}=\theta_2$.
For $s=3$, we have $\max \left\{c_{s-1}, 2-\frac{1}{c_s}\right\}=\max \left\{\frac{4}{3}, 2-\frac{1}{2}\right\}=\frac{3}{2}=\theta_3$.
For $s\geq 4$, we have $\max \left\{c_{s-1}, 2-\frac{1}{c_s}\right\}\leq \max \left\{\frac{2^{s-1}}{s}, 2\right\}=\frac{2^{s-1}}{s}<\frac{2^{s-1}}{s-1}=\theta_s$.
The proof of Lemma~\ref{le:CRKst-equal-2} is complete.
\hfill$\blacksquare$
\vspace{0.2cm}

Now we have all the ingredients to present our proof of Theorem~\ref{thm:CRKst-equal}.
\vspace{0.2cm}

\noindent {\bf Proof of Theorem~\ref{thm:CRKst-equal}.}
Since $r'_2(K_{1,t})=CR(K_{1,t}, K_3)=t+1$ and $r'_2(K_{s,t})\leq CR(K_{s,t}, K_3)$ for any $t\geq s\geq 2$, it suffices to show that $CR(K_{s,t}, K_3)\leq r'_2(K_{s,t})$ for fixed $s \geq 2$ and sufficiently large $t$.
Consider an edge-coloring $\chi$ of $K_n$ with $n=r'_2(K_{s,t})$, where $s \geq 2$ is fixed and $t$ is sufficiently large.
For a contradiction, suppose that it contains neither an orderable $K_{s,t}$ nor a rainbow $K_3$.
Then $\chi$ uses at least three colors (otherwise there is an orderable $K_{s,t}$ since $n=r'_2(K_{s,t})$).
By Lemma~\ref{le:CRKst-equal-2}, we have $n\leq M_s(t)\leq \theta_s t +O_s(1)$, where $\theta_2=\frac{5}{4}$, $\theta_3=\frac{3}{2}$ and $\theta_s=\frac{2^{s-1}}{s-1}$ for $s\geq 4$.
On the other hand, by Lemma~\ref{le:CRKstlower}, we have $n=r'_2(K_{s,t}) \geq \frac{2^s}{s+1} t-O_s\left(\sqrt{t \ln t}\right)$ for any fixed integer $s \geq 2$ and sufficiently large $t$.
Hence, we have $$\frac{2^s}{s+1} t-O_s\left(\sqrt{t \ln t}\right)\leq n\leq \theta_s t +O_s(1).$$
But $\theta_2=\frac{5}{4}<\frac{4}{3}$, $\theta_3=\frac{3}{2}<2$, and $\theta_s=\frac{2^{s-1}}{s-1}<\frac{2^s}{s+1}$ for $s\geq 4$,
where the last inequality follows from $\frac{2^{s-1}}{s-1}/\frac{2^s}{s+1}=\frac{s+1}{2(s-1)}<1$.
Hence, for sufficiently large $t$, we have $$\theta_s t +O_s(1)<\frac{2^s}{s+1} t-O_s\left(\sqrt{t \ln t}\right).$$
This contradiction completes the proof of Theorem~\ref{thm:CRKst-equal}.
\hfill$\blacksquare$

\section{Concluding remarks}
\label{sec:conclu}

In this paper, we obtain exact values or bounds for $r'_2(G)$ and $CR(G, K_3)$, primarily when $G$ is a complete graph or a complete bipartite graph.
From the definitions, we know that $r'_2(G)\leq CR(G, K_3)$.
For complete graphs $K_s$ with $s\geq 6$, we have $r'_2(K_s)< \frac{13}{12}\left(\frac{2}{\sqrt{5}}\right)^{s-3}\cdot CR(K_s, K_3)+1$ by Theorems~\ref{thm:r2s} and \ref{thm:CRs3};
whereas for complete bipartite graphs $K_{s,t}$ with $s$ fixed and $t$ large, we have
$r'_2(K_{s,t})=CR(K_{s,t}, K_3)=\left(\frac{2^s}{s+1}+o(1)\right)t$
by Theorem~\ref{thm:CRKst}.
In light of this, we propose the following problem.

\begin{problem}\label{prob:Kst-1}
For integers $t\geq s\geq 1$, what is the minimum integer $N$ such that, for every $n\geq N$, every edge-coloring of $K_n$ using at least three colors contains either an orderable copy of $K_{s,t}$ or a rainbow copy of $K_3$?
\end{problem}

Note that the value $N$ in Problem~\ref{prob:Kst-1} is in fact $M_s(t)+1$, where $M_s(t)$ is defined in Section~\ref{subsec:pf_complete_bipar_equal}.
For fixed $s\geq 2$ and sufficiently large $t$, Lemma~\ref{le:CRKst-equal-2} shows $M_s(t)\leq \theta_s t +O_s(1)$, where $\theta_s$ is strictly smaller than $\frac{2^s}{s+1}$.

Using properties of strongly regular graphs, Hadamard matrices and conference matrices, we show that $r'_2(K_{2,t})= CR(K_{2,t}, K_3)=\frac{4t}{3}+2$ and $r'_2(K_{3,t})= CR(K_{3,t}, K_3)=2t+2$ for infinitely many values of $t$.
We now remark on the challenges in employing our current arguments to determine exact values of $r'_2(K_{s,t})$ and $CR(K_{s,t}, K_3)$ for $s\geq 4$.
Essentially, our approach is to give an upper bound for $\left|Y_{\chi}(\mathbf{x})\right|$ or $L_{H}(\mathbf{x})$.
Both quantities can be expressed as sums, over vertices $y$, of indicators of the form $\mathbf{1}_{\left\{\left(\mathbf{1}_{\{x_1y \in E(H)\}}, \ldots, \mathbf{1}_{\{x_sy\in E(H)\}}\right)\in \mathcal{C}_s\right\}}$.
For $s\in \{2,3\}$, this can be controlled by vertex degrees and numbers of common neighbors of two vertices, so the parameters of a strongly regular graph provide exactly the required information.
For $s\geq 4$, this depends on numbers of common neighbors of more than two vertices.
Strongly regular graphs, Hadamard matrices and conference matrices do not in general determine the required information.

Another natural problem is to study the $k$-color variants $r'_k(G)$ and $CR_k(G, H)$.
For any integer $k\geq 2$ and graph $G$, we define $r'_k(G)$ as the minimum integer $n$ such that every $k$-edge-coloring of $K_n$ contains an orderable copy of $G$.
For any integer $k\geq 2$ and graphs $G$ and $H$, we define $CR_k(G, H)$ as the minimum integer $n$ such that every $k$-edge-coloring of $K_n$ contains either an orderable copy of $G$ or a rainbow copy of $H$.
Further research in this direction could be fruitful.

\vspace{-0.2cm}
\section*{Acknowledgements}
\vspace{-0.3cm}

This work was supported by the National Natural Science Foundation of China (Grant No. 12501492),
Shaanxi Fundamental Science Research Project for Mathematics and Physics (Grant No. 25JSQ043),
and the Fundamental Research Funds for the Central Universities (Grant No. GK202506024).

\vspace{-0.4cm}
\section*{Declaration of competing interests}
\vspace{-0.4cm}

The author declares no competing interests.
\vspace{-0.4cm}

\section*{Data availability}
\vspace{-0.4cm}

No data were used for the research described in the article.

\begin{spacing}{0.8} %行间距

\end{spacing}

\end{document}